\pdfoutput=1
\documentclass[a4paper,reqno]{amsart}
\usepackage[margin=3cm]{geometry}

\newcommand{\ifarticle}[2]{
    \csname@ifclassloaded\endcsname{beamer}{#2}{#1}
}

\newcommand{\ifbook}[2]{
    \csname@ifclassloaded\endcsname{amsbook}{#1}{#2}
}

    \usepackage{xparse} 
    \usepackage{xspace} 
    \usepackage{etoolbox} 
    \usepackage{xpatch} 
    \usepackage{pgffor} 

    \usepackage{amsmath,amssymb,amsthm} 
    \allowdisplaybreaks[1]

    \usepackage{mathtools} 
    \usepackage{mathrsfs} 
    \usepackage{stmaryrd} 
    \usepackage{bbm} 
    \usepackage{quiver} 
    \usepackage{xcolor} 
    \usepackage{scalerel} 
    \usepackage{booktabs} 
    \usepackage{nameref} 
    \usepackage[british]{babel} 
    \usepackage{csquotes} 
    \usepackage[UKenglish]{isodate} 
    \usepackage[T1]{fontenc}
    \usepackage{lmodern}       
    \usepackage{microtype}     
    \usepackage[splitrule]{footmisc} 

    \ifarticle{
        \PassOptionsToPackage{hyphens}{url}
        \usepackage[pdfusetitle,linktoc=all,colorlinks,citecolor=cyan,linkcolor=purple,urlcolor=blue]{hyperref} 
        \urlstyle{rm}

        \usepackage[nameinlink,noabbrev,capitalize]{cleveref} 
        \usepackage{footnotebackref} 

        \usepackage[shortlabels]{enumitem} 
        \setlist{topsep=2pt,itemsep=2pt,partopsep=2pt,parsep=2pt} 

        \setcounter{tocdepth}{1}
    }{}
    \usepackage[style=alphabetic,backref=true,backrefstyle=three,maxnames=9,minalphanames=3,maxalphanames=4]{biblatex} 
    \DeclareUnicodeCharacter{0301}{\TODO[Invalid symbol.]} 

    \DeclareFieldFormat*{citetitle}{\emph{#1}} 
    \DeclareCiteCommand{\citetitle}
        {\boolfalse{citetracker}%
        \boolfalse{pagetracker}%
        \usebibmacro{prenote}}
        {\ifciteindex
            {\indexfield{indextitle}}
            {}%
        \printtext[bibhyperref]{\printfield[citetitle]{labeltitle}}}
        {\multicitedelim}
        {\usebibmacro{postnote}}
    \DeclareCiteCommand*{\citeauthor}
        {\defcounter{maxnames}{99}%
        \defcounter{minnames}{99}%
        \defcounter{uniquename}{2}%
        \boolfalse{citetracker}%
        \boolfalse{pagetracker}%
        \usebibmacro{prenote}}
        {\ifciteindex{\indexnames{labelname}}{}%
        \printnames{labelname}}
        {\multicitedelim}
        {\usebibmacro{postnote}}

    \usepackage{calc} 
    \usepackage{tabularray} 
    \usepackage{ebproof} 
    \usepackage{forest} 
    \usepackage{adjustbox} 
    \usepackage{float} 
    \usepackage{stfloats}

    \makeatletter
    \ifarticle{
        \xpretocmd{\@adminfootnotes}{\let\@makefntext\BHFN@OldMakefntext}{}{}
        \renewcommand\@makefntext[1]{%
        \@ifundefined{@makefnmark}
            {}
            {%
            \renewcommand\@makefnmark{%
            \mbox{%
                \textsuperscript{%
                \normalfont
                \hyperref[\BackrefFootnoteTag]{\@thefnmark}%
                }%
            }\,%
            }%
            \BHFN@OldMakefntext{#1}%
        }%
        }

        \LetLtxMacro{\BHFN@Old@footnotemark}{\@footnotemark}
        \renewcommand*{\@footnotemark}{%
            \refstepcounter{BackrefHyperFootnoteCounter}%
            \xdef\BackrefFootnoteTag{bhfn:\theBackrefHyperFootnoteCounter}%
            \label{\BackrefFootnoteTag}%
            \BHFN@Old@footnotemark
        }

        \makeatletter
        \def\paragraph{\@startsection{paragraph}{4}%
          \z@\z@{-\fontdimen2\font}%
          {\normalfont\bfseries}}
        \makeatother
    }{}
    \makeatother

    \ifarticle{
        \theoremstyle{plain}
        \ifbook{
            \newtheorem{theorem}{Theorem}[chapter]
        }{
            \newtheorem{theorem}{Theorem}[section]
        }
        \newtheorem{proposition}[theorem]{Proposition}
        \newtheorem{lemma}[theorem]{Lemma}
        \newtheorem{corollary}[theorem]{Corollary}

        \newtheorem*{theorem*}{Theorem}
        \newtheorem*{corollary*}{Corollary}

        \theoremstyle{definition}
        \newtheorem{definition}[theorem]{Definition}
        
        \newtheorem{example}[theorem]{Example}

        \newtheorem{remark}[theorem]{Remark}

        \newenvironment{sketch}{\proof}{\endproof}

        \Crefname{theoremenumi}{Theorem}{Theorems}
        \AtBeginEnvironment{theorem}{%
            \crefalias{enumi}{theoremenumi}%
            \setlist[enumerate,1]{
                ref={\csname thetheorem\endcsname.(\arabic*)}
            }%
            \crefalias{enumii}{theoremenumi}%
            \setlist[enumerate,2]{
                ref={\thetheorem.(\arabic*).(\alph*)}
            }%
        }

        \Crefname{propositionenumi}{Proposition}{Propositions}
        \AtBeginEnvironment{proposition}{%
            \crefalias{enumi}{propositionenumi}%
            \setlist[enumerate,1]{
                ref={\csname theproposition\endcsname.(\arabic*)}
            }%
            \crefalias{enumii}{propositionenumi}%
            \setlist[enumerate,2]{
                ref={\theproposition.(\arabic*).(\alph*)}
            }%
        }

        \Crefname{lemmaenumi}{Lemma}{Lemmas}
        \AtBeginEnvironment{lemma}{%
            \crefalias{enumi}{lemmaenumi}%
            \setlist[enumerate,1]{
                ref={\csname thelemma\endcsname.(\arabic*)}
            }%
            \crefalias{enumii}{lemmaenumi}%
            \setlist[enumerate,2]{
                ref={\thelemma.(\arabic*).(\alph*)}
            }%
        }

        \Crefname{corollaryenumi}{Corollary}{Corollaries}
        \AtBeginEnvironment{corollary}{%
            \crefalias{enumi}{corollaryenumi}%
            \setlist[enumerate,1]{
                ref={\csname thecorollary\endcsname.(\arabic*)}
            }%
            \crefalias{enumii}{corollaryenumi}%
            \setlist[enumerate,2]{
                ref={\thecorollary.(\arabic*).(\alph*)}
            }%
        }

        \Crefname{definitionenumi}{Definition}{Definitions}
        \AtBeginEnvironment{definition}{%
            \crefalias{enumi}{definitionenumi}%
            \setlist[enumerate,1]{
                ref={\csname thedefinition\endcsname.(\arabic*)}
            }%
            \crefalias{enumii}{definitionenumi}%
            \setlist[enumerate,2]{
                ref={\thedefinition.(\arabic*).(\alph*)}
            }%
        }

        \Crefname{exampleenumi}{Example}{Examples}
        \AtBeginEnvironment{example}{%
            \crefalias{enumi}{exampleenumi}%
            \setlist[enumerate,1]{
                ref={\csname theexample\endcsname.(\arabic*)}
            }%
            \crefalias{enumii}{exampleenumi}%
            \setlist[enumerate,2]{
                ref={\theexample.(\arabic*).(\alph*)}
            }%
        }

        \newcommand{\qedshift}{\vspace*{-\baselineskip}}

        \foreach \env in {definition,remark,example,notation}{
        \AtBeginEnvironment{\env}{%
          \pushQED{\qed}%
        }
        \AtEndEnvironment{\env}{\popQED\endexample}
        }
    }{}

    \ExplSyntaxOn
    \NewDocumentCommand{\mathcommand}{mO{0}m}
     {
      \exp_args:Nc \NewCommandCopy {khue_kept_\cs_to_str:N #1} { #1 }
      \exp_args:Nc \newcommand {khue_new_\cs_to_str:N #1}[#2]{#3}
      \DeclareDocumentCommand {#1} {}
       {
        \mode_if_math:TF
         {
          \use:c {khue_new_\cs_to_str:N #1}
         }
         {
          \use:c {khue_kept_\cs_to_str:N #1}
         }
       }
     }
    \ExplSyntaxOff

    \newsavebox\tikzcdbox

    \newcommand{\nbh}{\nobreakdash-\hspace{0pt}}

    \mathcommand{\h}{\textup{-}}

    \newcommand{\tx}{\mathrm}
    \mathcommand{\b}{\mathbf}
    
    \newcommand{\cl}{\mathcal}
    \mathcommand{\bb}{\mathbb}
    \DeclareMathAlphabet{\bbn}{U}{bbold}{m}{n}
    
    \mathcommand{\sf}{\mathsf}
    \mathcommand{\u}{\underline}

    \newcommand{\TODO}[1][TODO]{\textcolor{orange}{\textup{#1}}\xspace}

    \newcommand{\datetoday}{\date{\cleanlookdateon\today}}

    \newcommand{\defeq}{\mathrel{:=}}
    \mathcommand{\d}{\mathbin{;}}
    \mathcommand{\c}{\circ}
    \newcommand{\ph}[1][]{{({-}_{#1})}}
    
    \newcommand{\iso}{\cong}
    \renewcommand{\equiv}{\simeq}

    \newcommand{\xto}{\xrightarrow}
    \newcommand{\xfrom}{\xleftarrow}
    \newcommand{\tto}{\Rightarrow}

    \newcommand{\ffto}{\hookrightarrow}
    
    \newcommand{\epito}{\twoheadrightarrow}
    
    \newcommand{\monoto}{\rightarrowtail}

    \newcommand*\cocolon{%
            \nobreak
            \mskip6mu plus1mu
            \mathpunct{}%
            \nonscript
            \mkern-\thinmuskip
            {:}%
            \mskip2mu
            \relax
    }

    \makeatletter
    \def\slashedarrowfill@#1#2#3#4#5{%
    $\m@th\thickmuskip0mu\medmuskip\thickmuskip\thinmuskip\thickmuskip
    \relax#5#1\mkern-7mu%
    \cleaders\hbox{$#5\mkern-2mu#2\mkern-2mu$}\hfill
    \mathclap{#3}\mathclap{#2}%
    \cleaders\hbox{$#5\mkern-2mu#2\mkern-2mu$}\hfill
    \mkern-7mu#4$%
    }
    \def\rightslashedarrowfill@{%
    \slashedarrowfill@\relbar\relbar\mapstochar\rightarrow}
    \newcommand\xslashedrightarrow[2][]{%
    \ext@arrow 0055{\rightslashedarrowfill@}{#1}{#2}}
    \def\leftslashedarrowfill@{%
    \slashedarrowfill@\leftarrow\relbar\mapsfromchar\relbar}
    \newcommand\xslashedleftarrow[2][]{%
    \ext@arrow 0055{\leftslashedarrowfill@}{#1}{#2}}
    \makeatother

    \newcommand{\xlto}{\xslashedrightarrow}
    \newcommand{\lto}{\xlto{}}

    \newcommand{\inv}{^{-1}}
    \newcommand{\op}{{}^\tx{op}}
    \newcommand{\co}{{}^\tx{co}}

    \newcommand{\tp}[1]{\langle#1\rangle}
    \newcommand{\unit}{{\tp{}}}
    
    \newcommand{\lx}{\mathbin{\rhd}}
    
    \newcommand{\plx}{\mathbin{{\rhd}\mathclap{\mspace{-17.5mu}\cdot}}}

    \newcommand{\adj}{\dashv}
    
    \newcommand{\ob}[1]{|#1|}

    \DeclareFontFamily{U}{min}{}
    \DeclareFontShape{U}{min}{m}{n}{<-> udmj30}{}

    \mathcommand{\comma}{\downarrow}

    \newsavebox{\whitecircstar}\sbox{\whitecircstar}{\kern.075em\tikz{\node[draw, circle,line width=.36pt, inner sep=0]{$*$};}\kern.075em}
    
    \newsavebox{\blackcircstar}\sbox{\blackcircstar}{\kern.075em\tikz{\node[fill, circle, line width=.36pt, inner sep=0, text=white]{$*$};}\kern.075em}

    \makeatletter
    \def\widebreve{\mathpalette\wide@breve}
    \def\wide@breve#1#2{\sbox\z@{$#1#2$}%
         \mathop{\vbox{\m@th\ialign{##\crcr
    \kern0.08em\brevefill#1{0.8\wd\z@}\crcr\noalign{\nointerlineskip}%
                        $\hss#1#2\hss$\crcr}}}\limits}
    \def\brevefill#1#2{$\m@th\sbox\tw@{$#1($}%
      \hss\resizebox{#2}{\wd\tw@}{\rotatebox[origin=c]{90}{\upshape(}}\hss$}
    \makeatother

    \NewDocumentCommand{\jrule}{om}{%
        \IfNoValueTF{#1}
            {\textsc{#2}}
            {$#1$-\textsc{#2}}%
    }

    \newcommand{\Set}{{\b{Set}}}
    
    \newcommand{\Cat}{\b{Cat}}
    \newcommand{\CAT}{\b{CAT}}

    \newcommand{\Mnd}{\b{Mnd}}
    \newcommand{\RMnd}{\b{RMnd}}
    
    \newcommand{\Alg}{\b{Alg}}
    
    \newcommand{\ff}{fully faithful}
    
    \newcommand{\ffness}{full faithfulness}

    \newcommand{\ioo}{identity-on-objects}

    \newcommand{\eg}{e.g.\@\xspace}
    \newcommand{\ie}{i.e.\@\xspace}
    
    \newcommand{\cf}{cf.\@\xspace}
    
    \newcommand{\aka}{a.k.a.\@\xspace}
    \NewDocumentCommand{\etc}{t.}{etc.\@\xspace}
    \NewDocumentCommand{\ibid}{t.}{ibid.\@\xspace}
    \NewDocumentCommand{\loccit}{t.}{loc.\ cit.\@\xspace}

\makeatletter

\define@key{beamerframe}{c}[true]{
    \beamer@frametopskip=0pt plus 1fill\relax%
    \beamer@framebottomskip=0pt plus 1fill\relax%
}

\patchcmd{\beamer@sectionintoc}{\vfill}{\vskip\itemsep}{}{}

\makeatother

\ifarticle{}{

  \colorlet{colour-bg}{black!85} 
  \definecolor{colour-primary}{HTML}{cc80ff} 
  \colorlet{colour-text}{black!10} 
  \colorlet{colour-subtle}{black!40} 
  \colorlet{colour-block-bg}{black!80} 
  \definecolor{colour-warning-bg}{HTML}{ffea80} 
  \definecolor{colour-warning-primary}{HTML}{e08152} 

  \hypersetup{linktoc=all,colorlinks, citecolor={colour-primary}, linkcolor={colour-primary}, urlcolor={colour-primary}} 
  \urlstyle{sf}

  \usepackage{changepage} 

  \usepackage{ragged2e} 

  \setbeamertemplate{section in toc}[sections numbered]

  \apptocmd{\frame}{}{\justifying}{}

  \usefonttheme[onlymath]{serif}

  \beamertemplatenavigationsymbolsempty
  \setbeamertemplate{footline}{
    \hfill%
    \usebeamercolor[fg]{page number in head/foot}%
    \usebeamerfont{page number in head/foot}%
    \setbeamertemplate{page number in head/foot}[pagenumber]%
    \usebeamertemplate*{page number in head/foot}\kern.2cm\vskip.2cm%
  }

  \setbeamertemplate{frametitle}[default][center]
  \addtobeamertemplate{frametitle}{\vspace*{1ex}}{\vspace*{1ex}}

  \setbeamertemplate{itemize item}{$\bullet$}

  \setbeamertemplate{theorems}[numbered]
  \newtheorem{proposition}[theorem]{\translate{Proposition}}

  \addtobeamertemplate{block begin}
  {}
  {\vspace{0ex} 
  \begin{adjustwidth}{1em}{1em} 
  \tikzcdset{background color=colour-block-bg}
  }
  \addtobeamertemplate{block end}{\end{adjustwidth}%
  \vspace{1ex}} 
  {}
  \renewenvironment<>{block}[1]{%
      \begin{actionenv}#2%
        \def\insertblocktitle{\begin{adjustwidth}{.8em}{.8em}#1\end{adjustwidth}}%
        \par%
        \usebeamertemplate{block begin}}
      {\par%
        \usebeamertemplate{block end}%
      \end{actionenv}}

  \addtobeamertemplate{block example begin}
  {}
  {\vspace{0ex} 
  \begin{adjustwidth}{1em}{1em} 
  \tikzcdset{background color=colour-block-bg}
  }
  \addtobeamertemplate{block example end}{\end{adjustwidth}%
  \vspace{1ex}} 
  {}
  \renewenvironment<>{exampleblock}[1]{%
      \begin{actionenv}#2%
          \def\insertblocktitle{\begin{adjustwidth}{.8em}{.8em}#1\end{adjustwidth}}%
          \par%
          \only<presentation>{
            \setbeamercolor{local structure}{parent=example text}}%
          \usebeamertemplate{block example begin}}
        {\par%
          \usebeamertemplate{block example end}%
        \end{actionenv}}

    \setbeamercolor{background canvas}{bg=colour-bg}
    \tikzcdset{background color=colour-bg}

    \setbeamercolor{block title}{fg=colour-primary,bg=colour-block-bg}
    \setbeamercolor{block title example}{fg=colour-primary,bg=colour-block-bg}
    \setbeamercolor{block body}{bg=colour-block-bg}
    \setbeamercolor{block body example}{bg=colour-block-bg}

    \setbeamercolor{title}{fg=colour-primary}
    \setbeamercolor{frametitle}{fg=colour-primary}
    \setbeamercolor{alerted text}{fg=colour-primary}

    \setbeamercolor{normal text}{fg=colour-text}
    \setbeamercolor{item}{fg=colour-text}
    \setbeamercolor{section in toc}{fg=colour-text}

    \setbeamercolor{page number in head/foot}{fg=colour-subtle}

    \setbeamercolor*{bibliography entry author}{fg=colour-text}
    \setbeamercolor*{bibliography entry note}{fg=colour-text}

}

\bibliography{references}

\usepackage{centernot}
\usepackage{comment}

\newcommand{\A}{{\cl A}}
\newcommand{\B}{{\cl B}}
\newcommand{\C}{{\cl C}}
\newcommand{\E}{{\cl E}}

\newcommand{\Kl}{\tx{Kl}}

\newcommand{\KlT}{\Kl(T)}
\newcommand{\StrAlg}{\tx{Alg}}
\newcommand{\PsAlg}{\tx{PsAlg}}
\newcommand{\Ins}{\tx{Ins}}
\newcommand{\AlgT}{\PsAlg(T)}
\newcommand{\AlgTl}{\PsAlg_l(T)}
\newcommand{\AlgTc}{\PsAlg_c(T)}
\newcommand{\AlgTs}{\PsAlg_s(T)}
\newcommand{\StrAlgT}{\StrAlg(T)}
\newcommand{\StrAlgTl}{\StrAlg_l(T)}
\newcommand{\StrAlgTc}{\StrAlg_c(T)}
\newcommand{\StrAlgTs}{\StrAlg_s(T)}
\newcommand{\Psh}{\b{Psh}}
\newcommand{\J}{J \colon \A \to \E}
\newcommand{\Hom}{\mathbf{Bicat}}
\newcommand{\COC}{\b{COC}}
\newcommand{\Fam}{{\b{Fam}}}
\renewcommand{\Im}{\tx{Im}}
\newcommand{\Res}{{\tx{Res}}}
\newcommand{\Dist}{{\b{Dist}}}
\newcommand{\DFib}{{\b{DFib}}}
\newcommand{\TSDFib}{{\b{TSDFib}}}
\newcommand{\PhiP}{\overline\Phi}
\newcommand{\El}{{\tx{El}}}
\newcommand{\term}{t}

\newcommand{\natof}[1]{\text{nat.\ of }#1}
\newcommand{\mndtheta}[1]{\hyperref[psmnd-theta]{#1\text{ unit.}}}
\newcommand{\algmu}[1]{\hyperref[psalg-mu]{#1\text{ assoc.}}}
\newcommand{\algtheta}[1]{\hyperref[psalg-theta]{#1\text{ unit.}}}
\newcommand{\algeta}[1]{\hyperref[psalg-eta]{#1\text{ unit.$'$}}}
\newcommand{\morassoc}[1]{\hyperref[lax-morph-hats]{#1\text{ assoc.}}}
\newcommand{\morunit}[1]{\hyperref[lax-morph-tildes]{#1\text{ unit.}}}
\newcommand{\trans}[1]{\hyperref[transformation]{#1\text{ trans.}}}

\DeclareMathOperator{\colim}{colim}

\usepackage{graphicx}

\newcommand{\algA}{(A, {-}^a, \hat a, \tilde a)}
\newcommand{\algB}{(B, {-}^b, \hat b, \tilde b)}
\newcommand{\algC}{(C, {-}^c, \hat c, \tilde c)}
\newcommand{\algLxA}{(A, {-}^a, \tilde a)}
\newcommand{\algLxB}{(B, {-}^b, \tilde b)}

\title{Bicategories of algebras for relative pseudomonads}
\author{Nathanael Arkor}
\address{Department of Software Science, Tallinn University of Technology, Estonia}
\author{Philip Saville}
\address{Department of Informatics, University of Sussex, United Kingdom}
\author{Andrew Slattery}
\address{Department of Computer Science and Technology, University of Cambridge, United Kingdom}
\thanks{N.\ Arkor was supported by a departmental postdoctoral grant from the Department of Software Science at Tallinn University of Technology. P.\ Saville was supported by the Air Force Office of Scientific Research (award number FA9550-21-1-0038)}
\subjclass{18C15, 18C20, 18D60, 18M50, 18N10, 18N15, 18N20}
\datetoday

\begin{document}

\begin{abstract}
    We introduce pseudoalgebras for relative pseudomonads and develop their theory. For each relative pseudomonad $T$, we construct a free--forgetful relative pseudoadjunction that exhibits the bicategory of $T$-pseudoalgebras as terminal among resolutions of $T$. The Kleisli bicategory for $T$ thus embeds into the bicategory of pseudoalgebras as the sub-bicategory of free pseudoalgebras. We consequently obtain a coherence theorem that implies, for instance, that the bicategory of distributors is biequivalent to the 2-category of presheaf categories. In doing so, we extend several aspects of the theory of pseudomonads to relative pseudomonads, including doctrinal adjunction, transport of structure, and lax-idempotence. As an application of our general theory, we prove that, for each class of colimits $\Phi$, there is a correspondence between monads relative to free $\Phi$-cocompletions, and $\Phi$-cocontinuous monads on free $\Phi$-cocompletions.
\end{abstract}

\maketitle

\tableofcontents


\section{Introduction}

Two-dimensional monad theory, which is the abstract study of two-dimensional algebraic structure on categories (or, more generally, the objects of a bicategory), is one of the cornerstones of two-dimensional category theory. Many fundamental structures in category theory -- such as monoidal categories, categories with a class of limits and/or colimits, and fibrations -- may be captured as the algebras for pseudomonads~\cite{street1980fibrations,blackwell1989two,lack2010companion}. The theory of two-dimensional monads and their algebras has been extensively developed (\eg~\cite{kelly1973coherence,kelly1974doctrinal,bunge1974coherent,blackwell1989two,power1989general,kelly1997property,le2002beck,lack2002codescent,bourke2013semiflexible,bourke2014two,walker2019distributive}), providing a powerful framework to study such structures.

\subsection*{Relative pseudomonads}

However, there are various important constructions in category theory that do not form pseudomonads yet are, in a certain sense, not far from doing so. A notable example is the presheaf construction, which associates to each small category $A$ the category of presheaves on $A$, \ie{} the functor category $[A\op, \Set]$. The presheaf construction is ``almost'' a pseudomonad, except that the presheaf category is generally only locally small\footnote{The exception being when $A = 0$, in which case $[A\op, \Set] \equiv 1$ is the terminal category.}, and so the construction may not be iterated. This prompted \citeauthor{fiore2018relative} to introduce \emph{relative pseudomonads}~\cite{fiore2018relative}, a generalisation to two dimensions of the relative monads of \textcite{altenkirch2010monads}. The notion of relative pseudomonad generalises that of a pseudomonad by relaxing the requirement that the underlying pseudofunctor be an endo-pseudofunctor. Thus, while the presheaf construction does not form a pseudomonad, it does give an example of a pseudomonad \emph{relative} to the inclusion $\Cat \ffto \CAT$ of small categories into locally small categories. Extending the classical setting of pseudomonads, the authors of \cite{fiore2018relative} showed that to each $J$-relative pseudomonad $T$ one may associate a \emph{Kleisli bicategory} $\KlT$, whose 1-cells $X \lto Y$ are precisely 1-cells $J(X) \to T(Y)$ in the base. In particular, the Kleisli bicategory for the presheaf construction is precisely \citeauthor{benabou1973distributeurs}'s bicategory $\Dist$ of small categories, distributors (\aka{} profunctors or (bi)modules), and natural transformations~\cite{benabou1973distributeurs}.\footnote{This fact appears to have been first observed informally by \citeauthor{guitart1975involutives} in example 1.b) of \cite[47]{guitart1975involutives}.} As explored in \cite{fiore2018relative} and subsequent work, this perspective has useful consequences in a range of areas of mathematics and computer science, for instance in the theory of combinatorial species~\cite{fiore2008cartesian}, concurrency~\cite{cattani2005profunctors}, linear logic~\cite{galal2020profunctorial,olimpieri2021intersection}, generalised operads and multicategories~\cite{cruttwell2010unified,curien2012operads,hyland2014elements}, and set theory~\cite{lewicki2020categories}.

\subsection*{Algebras for relative pseudomonads}

While \cite{fiore2018relative} develops a substantial portion of the theory of relative pseudomonads, a notable omission is any treatment of algebras. Algebras play a prominent role in two-dimensional monad theory: just as the algebras for a monad generalise the concept of algebraic structure on a set (such as monoids or rings), so algebras for a pseudomonad generalise the concept of algebraic structure on a category (such as monoidal or duoidal categories). It is therefore desirable to have a theory of algebras for relative pseudomonads, analogous to that in the one-dimensional setting~\cite{altenkirch2015monads}. This is the purpose of the present paper.

We begin by defining the notion of a pseudoalgebra for a relative pseudomonad $T$ (\cref{pseudoalgebra}) and constructing a bicategory of pseudoalgebras $\AlgT$ (\cref{AlgT}). In fact, in the spirit of two-dimensional algebra~\cite{blackwell1989two}, we define several bicategories of pseudoalgebras, in which the \mbox{1-cells} are the strict, pseudo, lax, and colax morphisms of pseudoalgebras respectively. Pseudoalgebras for relative pseudomonads jointly generalise the algebras for relative 2-monads considered by \textcite{arkor2024formal}, and the pseudoalgebras for no-iteration pseudomonads considered by \textcite{marmolejo2013no}.

We justify our definition by showing that various key results in the theory of pseudomonads extend to our setting. Often, though not always, the proof strategies in the relative setting are similar to those in the classical setting. However, care must be taken to rework statements and proofs to avoid iteration of the pseudomonad. We consider four aspects of the theory of pseudoalgebras in particular.

First, in \cref{algebras-for-relative-pseudomonads} we show that the bicategory of pseudoalgebras is part of a free--forgetful relative pseudoadjunction inducing $T$ (\cref{sec:free-forgetful-pseudoadjunction}), and take a step towards its characterisation as a three-dimensional limit (\cref{sec:bicategory-of-pseudoalgebras-as-a-limit}).
\begin{equation}
\label{eq:EM-resolution}
\begin{tikzcd}
    & \AlgT \\
    \A && \E
    \arrow[""{name=0, anchor=center, inner sep=0}, "{U_T}", from=1-2, to=2-3]
    \arrow[""{name=1, anchor=center, inner sep=0}, "{F_T}", from=2-1, to=1-2]
    \arrow["J"', from=2-1, to=2-3]
    \arrow["\dashv"{anchor=center}, shift right=2, draw=none, from=1, to=0]
\end{tikzcd}\end{equation}

Second, we study \emph{doctrinal adjunction}, which is the study of structure-preserving adjunctions between pseudoalgebras~\cite{kelly1974doctrinal}. In \cref{sec:doctrinal-adjunction} we extend the main results of doctrinal adjunction to relative pseudomonads, and use it to establish \emph{transport of structure}, which shows that pseudoalgebra structure may be transferred across equivalences -- and, more generally, across (co)reflections.

Third, we study \emph{lax-idempotence} for relative pseudomonads (\cref{lax-idempotence}). In comparison to the theory of lax-idempotent pseudomonads, which were introduced to capture pseudomonads exhibiting free cocompletions under classes of colimits~\cite{kock1995monads,zoberlein1976doctrines}, the relative setting is rather subtle. The authors of \cite{fiore2018relative} proposed a definition of lax-idempotence for relative pseudomonads and showed that the presheaf construction is lax-idempotent in this sense. However, we show that this notion is too weak to capture theorems of interest. To rectify this, we introduce lax-idempotence for pseudoalgebras and use it to recover aspects of the theory known for non-relative pseudomonads. We show that a relative pseudomonad is lax-idempotent in the sense of \cite{fiore2018relative} if and only if every \emph{free} pseudoalgebra is lax-idempotent; in contrast to the non-relative setting~\cite{kock1995monads}, this is not enough to ensure every non-free pseudoalgebra is lax-idempotent.

Finally, we prove two-dimensional analogues of the well-known facts that, among adjunctions inducing a fixed monad, the Kleisli adjunction is initial and the Eilenberg--Moore adjunction is terminal~\cite{kleisli1965every,eilenberg1965adjoint}. In \cite[Theorem~4.4]{fiore2018relative}, the Kleisli bicategory for a relative pseudomonad $T$ was shown to be part of a relative pseudoadjunction inducing $T$.
\begin{equation}
\label{eq:Kleisli-resolution}
\begin{tikzcd}
    & {\Kl(T)} \\
    \A && \E
    \arrow[""{name=0, anchor=center, inner sep=0}, "{K_T}", from=2-1, to=1-2]
    \arrow[""{name=1, anchor=center, inner sep=0}, "{V_T}", from=1-2, to=2-3]
    \arrow["J"', from=2-1, to=2-3]
    \arrow["\dashv"{anchor=center}, shift right=2, draw=none, from=0, to=1]
\end{tikzcd}
\end{equation}
We show that, among relative pseudoadjunctions inducing $T$, \eqref{eq:Kleisli-resolution} is bi-initial (\cref{initial-resolution}) and \eqref{eq:EM-resolution} is 2-terminal (\cref{terminal-resolution}). As a consequence, there is a unique, \ff{} pseudofunctor $\Kl(T) \to \PsAlg(T)$ commuting with the left and right pseudoadjoints, whose image is the full sub-bicategory of free pseudoalgebras. While this result appears analogous to the one-dimensional setting, it has a useful corollary that only emerges in two dimensions. When the codomain of $T$ is a 2-category, $\PsAlg(T)$ is also a 2-category, whereas $\Kl(T)$ need not be. Embedding of the Kleisli bicategory into the 2-category of pseudoalgebras thus exhibits a coherence theorem for $\Kl(T)$ that captures many naturally occurring situations.

This brings us to our examples. As mentioned above, a motivating example of a relative pseudomonad is the presheaf construction. It is clear what the pseudoalgebras for the presheaf construction \emph{ought} to be: namely, the locally small categories admitting small colimits. Perhaps surprisingly, establishing this is not straightforward -- nevertheless, it is true, as we show in \cref{presheaves-and-cocompleteness}. As a consequence of our coherence theorem, we conclude that the bicategory of distributors is biequivalent to the full sub-2-category of the 2-category of cocomplete categories spanned by the presheaf categories (a result present in the categorical folklore, but for which, to our understanding, no rigorous proof has previously been given). In fact, we work more generally, and characterise the pseudoalgebras not just for the presheaf relative pseudomonad but for the relative pseudomonads exhibiting the free cocompletions of small categories under a class $\Phi$ of colimits (\cref{Phi-pseudoalgebras}).

We conclude the paper in \cref{relative-monads-and-monads} with an application that demonstrates how the theory of relative \emph{pseudomonads} sheds light on the theory of relative \emph{monads}. In particular, by considering monads in the respective bicategories, we use the embedding of the Kleisli bicategory into the bicategory of pseudoalgebras to deduce that monads relative to free $\Phi$-cocompletions are equivalent to $\Phi$-cocontinuous monads on free $\Phi$-cocompletions.

\subsection*{Future directions}

This paper is a first step towards two-dimensional relative monad theory. For reasons of space there are many topics we have not covered here. For instance, a key theme in the theory of 2-monads is the study of weak structures in terms of stricter structures (\eg~representing lax morphisms as certain strict morphisms~\cite{blackwell1989two}). It would be useful to extend these ideas to the relative setting.

One of the motivations for our development is in connection to the theory of pseudocommutativity. In \cite{slattery2023pseudocommutativity,slattery2024commutativity}, the third-named author initiated the study of pseudocommutativity for relative pseudomonads. In the one-dimensional setting, it is known that commutative monads equip their categories of algebras, under mild assumptions, with monoidal structure~\cite{kock1970monads}; aspects of this theory have recently been extended to the setting of monoidal pseudomonads~\cite{miranda2024eilenberg}. We expect something similar to be true for pseudocommutative relative pseudomonads, which in particular include all lax-idempotent relative pseudomonads. This would give, for instance, an abstract construction of the tensor product of categories equipped with a class of small colimits~\cite[\S6.5]{kelly1982basic}.

Finally, in future work we intend to study the pseudoalgebras arising from distributive laws of relative pseudomonads, so as to capture a variety of structures in category theory that are sensitive to size~\cite{gabriel1971lokal,im1986universal,garner2012lex}. Such distributive laws are also relevant to recent developments in the semantics of linear and modal logics~\cite{fiore2008cartesian,olimpieri2021intersection,galal2020profunctorial,kavvos2024two1,kavvos2024two2}, and it would be natural to consider the corresponding bicategories of pseudoalgebras as candidates for models of linear logic, in the style of call-by-name models or adjunction models~\cite{levy2003call,curien2016theory}.

\subsection*{Acknowledgements}

Parts of this paper were developed as part of the third-named author's doctoral thesis~\cite{slattery2024commutativity}, under the supervision of Nicola Gambino. The authors thank Ivan Di Liberti for suggesting a simpler proof of \cref{colimit-of-dense-functor}.

\section{Background}
\label{background}

\subsection{Bicategories and pseudofunctors}
\label{notation}

We shall assume basic familiarity with two-dimensional category theory, including the definitions of 2-categories and bicategories, 2-functors and pseudofunctors, pseudonatural transformations, and modifications, as well as the theory of mates for adjunctions (see \cite{johnson20212} for a general reference).

For the most part, our notation is chosen for consistency with that of \cite{fiore2018relative}. We denote generic 2-categories and bicategories by calligraphic script, \eg $\A, \B, \C$; and their objects by uppercase Latin letters, \eg $X, Y, Z$. We denote by $\CAT$ the 2-category of locally small categories, functors, and natural transformations; and by $\Cat$ the full sub-2-category spanned by the small categories. Given a bicategory $\A$ and a pair of objects $X, Y \in \A$, we denote by $\A[X, Y]$ the category of 1-cells from $X$ to $Y$ (denoted by lowercase Latin letters, \eg $f, g, h$), and the 2-cells therebetween (denoted by lowercase Greek letters, \eg $\phi, \gamma, \eta$). Given bicategories $\A, \B$, we denote by $\Hom(\A, \B)$ the bicategory of pseudofunctors from $\A$ to $\B$ (denoted by uppercase Latin letters, \eg $F, G, H$), and the pseudonatural transformations and modifications therebetween; note that $\Hom(\A, \B)$ is a 2-category if $\B$ is. We denote the underlying collection of objects of a (bi)category $A$ by $\ob{A}$, and similarly write $\ob{F}$ for the underlying action on objects of a (pseudo)functor.

We shall denote the structural natural isomorphisms for associativity and unitality of composition in a bicategory by unnamed isomorphisms,
\begin{align*}
    (hg)f & \xto\iso h(gf) &
    1 f & \xto\iso f &
    f & \xto\iso f 1
\end{align*}
and similarly for the compositor and unitor natural isomorphisms for a pseudofunctor. A pseudofunctor is \emph{strict} when these isomorphisms are in fact identities.
\begin{align*}
    F(g) F(f) & \xto\iso F(gf) &
    1 & \xto \iso F(1)
\end{align*}
In pasting diagrams, we allow ourselves to elide the associator isomorphisms where it aids readability, and simply write, for instance, $hgf$; this is formally justified by the coherence theorem for bicategories~\cite[\S8.4]{johnson20212}. Similarly, diagrams that commute as a consequence of coherence for bicategories will generally not be explicitly labelled. We denote horizontal composition of 2-cells $\alpha$ and $\beta$ by $\beta \cdot \alpha$. We shall employ a diagrammatic convention in which rounded arrows on the outside of a diagram indicate definitional equality. For instance, if $h \defeq g \c f$, we may write the following. This makes it easier to distinguish the non-trivial equalities in a commutative diagram.
\[\begin{tikzcd}
	\cdot & \cdot & \cdot
	\arrow["f"', from=1-1, to=1-2]
	\arrow["g"', from=1-2, to=1-3]
    \arrow[
        "h",
        rounded corners,
        to path=
        { -- ([yshift=.6cm]\tikztostart.center)
        -- ([yshift=.6cm]\tikztotarget.center)
        \tikztonodes
        -- (\tikztotarget.north) },
        from=1-1, to=1-3
    ]
\end{tikzcd}\]

Finally, in an adjunction, we refer to the triangle law $\ell \tto \ell r \ell \tto \ell$ as the \emph{left triangle law}, and $r \tto r \ell r \tto r$ as the \emph{right triangle law}.

\subsection{Relative pseudoadjunctions}

We now recall the fundamental concepts introduced in \cite{fiore2018relative}. The central notions are that of a \emph{relative pseudoadjunction} and of a \emph{relative pseudomonad}.
A relative pseudoadjunction is a joint generalisation of the notion of relative adjunction~\cite{ulmer1968properties,altenkirch2015monads} and the notion of pseudoadjunction~\cite{street1980fibrations}.

\begin{definition}[{\cite[Definition~3.6]{fiore2018relative}, \cf{}~\cite[Definition~2.12]{przybylek2014analysis}}]
    \label{relative-pseudoadjunction}
     Let $\J$ be a pseudofunctor between bicategories. A \emph{$J$-relative pseudoadjunction} comprises
     \begin{enumerate}
         \item a pseudofunctor $R \colon \C \to \E$;
         \item an object $LX \in \C$ for each $X \in \A$;
         \item a 1-cell $i_X \colon JX \to RLX$ in $\E$ for each $X \in \A$;
         \item \label{relative pseudoadjunction local adjunction} an adjoint equivalence
        \[\begin{tikzcd}[column sep=7em]
        	{\E[J X, R Y]} & {\C[L X, Y]}
        	\arrow[""{name=0, anchor=center, inner sep=0}, "{\ph^\sharp_{X, Y}}", shift left=2, from=1-1, to=1-2]
        	\arrow[""{name=1, anchor=center, inner sep=0}, "{R_{LX, Y}\!\ph \c i_X}", shift left=2, from=1-2, to=1-1]
        	\arrow["\dashv"{anchor=center, rotate=-90}, draw=none, from=0, to=1]
        \end{tikzcd}\]
        for each $X \in \A$ and $Y \in \C$. \qedhere
     \end{enumerate}
\end{definition}

Recall that the presheaf construction sends a small category $X$ to the locally small presheaf category $[X\op, \Set]$. This assignment has a universal property: it is the free cocompletion of $X$ under small colimits~\cite{andre1966categories}. However, note that this universal property identifies the free cocompletion up to equivalence rather than up to isomorphism: thus, it is properly described as a bicategorical universal property. Furthermore, since the presheaf construction takes small categories to locally small categories, it does not exhibit a pseudoadjunction, but rather a relative pseudoadjunction (while the presheaf construction is defined also on large categories, it only satisfies the universal property on the small categories: see \cref{presheaf-construction-pseudomonad}). The presheaf construction thus provides one of the motivating examples of a relative pseudoadjunction~\cite{fiore2018relative}.

\begin{definition}
    \label{COC}
    Denote by $\COC$ the 2-category whose objects are locally small categories equipped with a choice of small colimits: specifically, an object is a locally small category $A$ along with, for each small category $D$ and functor $f \colon D \to A$, a cocone under $f$ that is colimiting. A 1-cell is a functor that sends the colimiting cocones to colimiting cocones (not necessarily the chosen ones). The 2-cells are natural transformations.
\end{definition}

\begin{example}[{\cite[Example~3.9]{fiore2018relative}}]
    \label{presheaf-construction-pseudoadjunction}
    Denote by $J \colon \Cat \ffto \CAT$ the inclusion of the \mbox{2-category} of small categories into the 2-category of locally small categories. The presheaf construction $X \mapsto [X\op, \Set]$ exhibits a left $J$-relative pseudoadjoint $P \colon \Cat \to \COC$ to the forgetful 2-functor $U \colon \COC \to \CAT$.
    \[\begin{tikzcd}
        & \COC \\
        \Cat && \CAT
        \arrow[""{name=0, anchor=center, inner sep=0}, "P", from=2-1, to=1-2]
        \arrow[""{name=1, anchor=center, inner sep=0}, "U", from=1-2, to=2-3]
        \arrow["J"', from=2-1, to=2-3]
        \arrow["\dashv"{anchor=center}, shift right=2, draw=none, from=0, to=1]
    \end{tikzcd}\]
    The unit $y_X \colon X \to [X\op, \Set]$ is given by the Yoneda embedding. For each small category $X$ and cocomplete locally small category $A$, the functor
    \[\CAT[X, U(A)] \xfrom{U\ph \c y_X} \COC[P(X), A]\]
    admits a left adjoint equivalence, given by left extension along $y_A$.
\end{example}

\begin{example}
    When $\J$ is a functor between categories (viewed as a 2-functor between locally discrete 2-categories), a $J$-relative pseudoadjunction is the same as a $J$-relative adjunction~\cite[Definition~2.2]{ulmer1968properties}.
\end{example}

\begin{example}
    When $\J$ is a 2-functor between 2-categories and the adjoint equivalences appearing in \cref{relative-pseudoadjunction} are isomorphisms of categories, we obtain the notion of \emph{relative 2-adjunction}, \ie{} a $\CAT$-enriched relative adjunction~\cite[Example~8.11]{arkor2024formal}.
\end{example}

For every relative pseudoadjunction as in \cref{relative-pseudoadjunction}, it follows from \cite[Lemma~3.7]{fiore2018relative} that the object assignment $L$ extends essentially uniquely to a pseudofunctor such that the family $\{ i_X \colon JX \to RLX \}_{X \in \A}$ may be equipped with the structure of a pseudonatural transformation. The action of the pseudofunctor $L$ at $X, Y \in \A$ is given by the following composite.
\begin{equation}
    \label{action-of-left-pseudoadjoint}
    L_{X, Y} \defeq
    \begin{tikzcd}[column sep=huge]
        {\A[X, Y]} & {\E[JX, JY]} & {\E[JX, RLY]} & {\C[LX, LY]}
        \arrow["{J_{X, Y}}", from=1-1, to=1-2]
        \arrow["{\E[JX, i_Y]}", from=1-2, to=1-3]
        \arrow["{\sharp_{X, LY}}", from=1-3, to=1-4]
    \end{tikzcd}
\end{equation}

While it is not proven \ibid, it also follows that $\sharp$ is pseudonatural with respect to this choice; the following generalises \cite[Lemmas~9.14 \& 9.15]{fiore2006pseudo} from pseudoadjunctions to relative pseudoadjunctions.

\begin{lemma}
    \label{sharp-is-pseudonatural}
    Let $(L, R, i, \sharp, \eta, \varepsilon)$ be a $J$-relative pseudoadjunction as in \cref{relative-pseudoadjunction}. The family of functors
    $\big\{ \sharp_{X, Y} \colon \E[JX, RY] \to \C[LX, Y] \big\}_{X \in \A, Y \in \C}$
    admits the structure of a pseudonatural equivalence $\sharp \colon \E[J{-}_1, R{-}_2] \tto \C[L{-}_1, {-}_2] \colon \A\op \times \C \to \Cat$.
\end{lemma}

\begin{proof}
    The following composite admits a canonical pseudonatural transformation structure,
    \[
        \C[L{-}_1, {-}_2] \xto{R_{L{-}_1, {-}_2}} \E[RL{-}_1, R{-}_2] \xto{\E[i{-}_1, R{-}_2]} \E[J{-}_1, R{-}_2]
    \]
    since both postcomposition and $i$ are pseudonatural~\cites[Lemma~8.2.9]{saville2019cartesian}[Lemma~3.7]{fiore2018relative}. Furthermore, since each of the components of the pseudonatural transformation are adjoint equivalences, the family of adjoints to the components is canonically equipped with the structure of a pseudonatural transformation in the opposite direction~\cite[Lemma~2.1.16]{saville2019cartesian}. Consequently, $\sharp$ is pseudonatural.
\end{proof}

\subsection{Relative pseudomonads}
\label{relative-pseudomonads}

A relative pseudomonad is a joint generalisation of the notion of relative monad~\cite{altenkirch2010monads,altenkirch2015monads} and the notion of pseudomonad~\cite{bunge1974coherent}.

\begin{definition}[{\cite[Definition~3.1]{fiore2018relative}}]
    \label{relative-pseudomonad}
    Let $\J$ be a pseudofunctor between bicategories. A \emph{$J$-relative pseudomonad} comprises
    \begin{enumerate}
        \item an object $TX \in \E$ for each $X \in \A$;
        \item a functor $\ph^*_{X, Y} \colon \E[JX, TY] \to \E[TX, TY]$ for each $X, Y \in \A$, the \emph{extension operator} (we allow ourselves to elide the subscripts where they may be inferred from context);
        \item a 1-cell $i_X \colon JX \to TX$ for each $X \in \A$, the \emph{unit};
        \item a natural family of invertible 2-cells $\mu_{g, f} \colon (g^* f)^* \tto g^* f^*$ for each $f \colon JX \to TY$ and $g \colon JY \to TZ$ in $\E$;
        \[\begin{tikzcd}[column sep=large]
            TX && TZ \\
            & TY
            \arrow["{f^*}"', curve={height=6pt}, from=1-1, to=2-2]
            \arrow["{g^*}"', curve={height=6pt}, from=2-2, to=1-3]
            \arrow[""{name=0, anchor=center, inner sep=0}, "{(g^* f)^*}", from=1-1, to=1-3]
            \arrow["{\mu_{g, f}}"', shorten <=3pt, shorten >=3pt, Rightarrow, from=0, to=2-2]
        \end{tikzcd}\]
        \item a natural family of invertible 2-cells $\eta_f \colon f \tto f^* i_X$ for each $f \colon JX \to TY$ in $\E$;
        \[\begin{tikzcd}[column sep=large]
            JX && TY \\
            & TX
            \arrow["{i_X}"', curve={height=6pt}, from=1-1, to=2-2]
            \arrow["{f^*}"', curve={height=6pt}, from=2-2, to=1-3]
            \arrow[""{name=0, anchor=center, inner sep=0}, "f", from=1-1, to=1-3]
            \arrow["{\eta_f}"', shorten <=3pt, shorten >=3pt, Rightarrow, from=0, to=2-2]
        \end{tikzcd}\]
        \item a family of invertible 2-cells $\theta_X \colon (i_X)^* \tto 1_{TX}$ for each $X \in \A$,
    \end{enumerate}
    such that
    \begin{enumerate}[resume]
        \item \label{psmnd-mu} for each $f \colon JW \to TX$, $g \colon JX \to TY$, and $h \colon JY \to TZ$ in $\E$, the following associativity law holds;
        \[\begin{tikzcd}[column sep=large]
            {((h^*g)^*f)^*} && {(h^*g)^*f^*} \\
            {((h^*g^*)f)^*} && {(h^*g^*)f^*} \\
            {(h^*(g^*f))^*} & {h^*(g^*f)^*} & {h^*(g^*f^*)}
            \arrow["{{\mu_{h^*g,f}}}", from=1-1, to=1-3]
            \arrow["{{\mu_{h,g}f^*}}", from=1-3, to=2-3]
            \arrow["\iso", from=2-3, to=3-3]
            \arrow["{{(\mu_{h,g}f)^*}}"', from=1-1, to=2-1]
            \arrow["\iso"', from=2-1, to=3-1]
            \arrow["{{\mu_{h,g^*f}}}"', from=3-1, to=3-2]
            \arrow["{{h^*\mu_{g,f}}}"', from=3-2, to=3-3]
        \end{tikzcd}\]
        \item \label{psmnd-theta} for each $f \colon JX \to TY$ in $\E$, the following unit law holds.
       \[\begin{tikzcd}
       	{f^*} & {(f^*i_{X})^*} & {f^*{(i_{X})}^*} \\
       	&& {f^*1_{TX}}
       	\arrow["{{(\eta_f)^*}}", from=1-1, to=1-2]
       	\arrow["\iso"', curve={height=12pt}, from=1-1, to=2-3]
       	\arrow["{{\mu_{f,i_{X}}}}", from=1-2, to=1-3]
       	\arrow["{f^*\theta_X}", from=1-3, to=2-3]
        \end{tikzcd}\qedshift\]
    \end{enumerate}
\end{definition}

By \cite[Theorem~3.8]{fiore2018relative}, each relative pseudoadjunction $(L, R, i, \sharp, \eta, \varepsilon)$ induces a relative pseudomonad whose underlying pseudofunctor is given by $RL$, whose extension operator is given by
\begin{equation}
    \ph^*_{X, Y} \defeq \E[JX, RLY] \xto{\ph^\sharp_{X, LY}} \E[LX, LY] \xto{R_{LX, LY}} \E[RLX, RLY]
\end{equation}
and whose unit is given by $i$.

\begin{example}
    \label{presheaf-construction-pseudomonad}
    The relative pseudoadjunction of \cref{presheaf-construction-pseudoadjunction} induces a $(\Cat \ffto \CAT)$-relative pseudomonad, which sends a small category $X$ to the presheaf category $P(X) \defeq [X\op, \Set]$. Note that relativity is crucial in this example: although the presheaf category $[A\op, \Set]$ is defined for arbitrary large categories $A$, the unit $y_A \colon A \to [A\op, \Set]$ is defined only when $A$ is locally small, and the requisite left extensions defining the extension operator exist generally only when $A$ is small. (However, there is a closely related construction -- the \emph{small presheaf} construction -- which does define a pseudomonad on $\CAT$; we discuss its relation to the presheaf construction in \cref{presheaves-and-cocompleteness}.)
\end{example}

\begin{example}
    \label{relative-monad}
    When $\J$ is a functor between categories (viewed as a 2-functor between locally discrete 2-categories), a $J$-relative pseudomonad is the same as a $J$-relative monad~\cite[Definition~2.1]{altenkirch2015monads}.
\end{example}

\begin{example}
    When $\J$ is a 2-functor between 2-categories and each of the families $\mu$, $\eta$, and $\theta$ appearing in \cref{relative-pseudoadjunction} are identities, we obtain the notion of \emph{relative 2-monad}, \ie{} a $\CAT$-enriched relative monad~\cite[Example~8.14]{arkor2024formal}.
\end{example}

For every relative pseudomonad as in \cref{relative-pseudomonad}, it follows from \cite[Proposition~4.7]{fiore2018relative} that the object assignment $T$ extends to a pseudofunctor, where the action on a 1-cell ${f \colon X \to Y}$ in $\A$ is given by $(i_Y \c Jf)^* \colon TX \to TY$, and that the families $\{ i_X \colon JX \to TX \}_{X \in \A}$ and $\{ \ph^*_{X, Y} \colon \E[J{-}_1, T{-}_2] \to \E[T{-}_1, T{-}_2] \}_{X, Y \in \A}$ may consequently be equipped with the structure of pseudonatural transformations.\footnote{\cite{fiore2018relative} do not show that the families of invertible 2-cells $\mu$, $\eta$, and $\theta$ become modifications under these assignments. That $\eta$ forms a modification will follow from \cref{algebras give pseudonat trans and modification}, and that $\theta$ forms a modification is trivial. However, to prove that $\mu$ forms a modification would require a detour into the theory of pseudodistributors, which we prefer to avoid here.}

\subsection{The Kleisli bicategory}

To every $(\J)$-relative pseudomonad $T$, we may associate a bicategory with the same objects as $\A$, whose 1-cells from $X$ to $Y$ are given by 1-cells $JX \to TY$ in $\E$.

\begin{definition}[{\cite[Theorem~4.1]{fiore2018relative}}]
    Let $\J$ be a pseudofunctor and let $T$ be a $J$-relative pseudomonad. The \emph{Kleisli bicategory} of $T$ has the same objects as $\A$ and $\KlT[X, Y] \defeq \E[JX, TY]$. The identity on an object $X \in \KlT$ is given by $i_X \colon JX \to TX$, and the composite of 1-cells $f \colon JX \to TY$ and $g \colon JY \to TZ$ is given by $g^* \c f \colon JX \to TZ$.
\end{definition}

Recall that a distributor (\aka{} profunctor) from $X$ to $Y$ is a functor $Y\op \times X \to \Set$~\cite{benabou1973distributeurs}. Small categories, distributors, and natural transformations assemble into a bicategory $\Dist$ (\aka{} $\b{Prof}$). One of the main contributions of \cite{fiore2018relative} was to give an abstract construction of the bicategory $\Dist$.

\begin{example}[{\cite[Example~4.2]{fiore2018relative}}]
    \label{Dist}
    Let $P$ be the $(\Cat \ffto \CAT)$-relative pseudomonad of \cref{presheaf-construction-pseudomonad}. Then $\Kl(P)$ is the bicategory $\Dist$.
\end{example}

Just as in the one-dimensional setting, the Kleisli bicategory exhibits every relative pseudomonad as induced by a relative pseudoadjunction.

\begin{theorem}[{\cite[Lemma~3.4, Theorem~4.4]{fiore2018relative}}]
    \label{Kleisli-resolution}
    Let $\J$ be a pseudofunctor. Every $J$-relative pseudomonad $T$ is induced by a relative pseudoadjunction
    \[\begin{tikzcd}
        & {\Kl(T)} \\
        \A && \E
        \arrow[""{name=0, anchor=center, inner sep=0}, "{K_T}", from=2-1, to=1-2]
        \arrow[""{name=1, anchor=center, inner sep=0}, "{V_T}", from=1-2, to=2-3]
        \arrow["J"', from=2-1, to=2-3]
        \arrow["\dashv"{anchor=center}, shift right=2, draw=none, from=0, to=1]
    \end{tikzcd}\]
    in which $K_T$ is identity-on-objects, and $V_T$ is given on objects by the action of $T$ and on morphisms by the extension operator $\ph^*$.
\end{theorem}

We will show in \cref{resolutions-and-coherence} that this relative pseudoadjunction is universal in an appropriate sense.

\section{Pseudoalgebras for relative pseudomonads}
\label{algebras-for-relative-pseudomonads}

We now turn to the central notion of this paper: that of a pseudoalgebra for a relative pseudomonad. While relative pseudomonads describe what it means to freely adjoin coherent structure to an object of a bicategory, a pseudoalgebra describes what it means for an object to be equipped with such structure, not necessarily freely. For instance, we shall prove in \cref{presheaves-and-cocompleteness} that the pseudoalgebras for the presheaf construction of \cref{presheaf-construction-pseudomonad} are precisely the locally small categories with a choice of small colimits.

We begin in \cref{sec:bicategories-of-pseudoalgebras} by introducing the definitions of pseudoalgebras, their (strict, pseudo, lax, and colax) morphisms, and transformations, and establishing that they form bicategories (\cref{AlgT}). Next, we show in \cref{sec:comparison-with-no-iteration-pseudoalgebras} that this definition is a faithful generalisation of the bicategories of pseudoalgebras for non-relative pseudomonads by establishing an equivalence between pseudoalgebras for pseudomonads relative to the identity, and pseudoalgebras for no-iteration pseudomonads in the sense of \textcite{marmolejo2013no}. In \cref{sec:free-forgetful-pseudoadjunction}, we show that the construction of free $T$-pseudoalgebras exhibits a free--forgetful relative pseudoadjunction. Finally, in \cref{sec:bicategory-of-pseudoalgebras-as-a-limit}, we give a three-dimensional perspective on the bicategories of pseudoalgebras by showing that they admit strict pseudofunctors into certain three-dimensional limits.

\subsection{Bicategories of pseudoalgebras}
\label{sec:bicategories-of-pseudoalgebras}

Algebras for relative monads were first studied by \citeauthor{walters1969alternative} as the algebras for \emph{devices}~\cite{walters1969alternative,walters1970categorical}. They were introduced in full generality in \cite{altenkirch2010monads}. Our definition is a joint generalisation of the definition of an algebra for a relative 2-monad and of an algebra for a no-iteration pseudomonad~\cite{arkor2024formal,marmolejo2013no}.

\begin{definition}\label{pseudoalgebra}
    Let $\J$ be a pseudofunctor between bicategories.
    A \emph{pseudoalgebra} (or simply \emph{$T$-pseudoalgebra}) for a $J$-relative pseudomonad $T$ as in \cref{relative-pseudomonad} comprises
    \begin{enumerate}
        \item an object $A \in \E$;
        \item a functor $\ph^a_X \colon \E[JX, A] \to \E[TX, A]$ for each $X \in \A$, the \emph{extension operator} (we allow ourselves to elide the subscript when it can be inferred from context);
        \item a natural family of invertible 2-cells $\hat a_{g, f} \colon (g^a f)^a \tto g^a f^*$ for each $f \colon JX \to TY$ and $g \colon JY \to A$ in $\E$;
        \[\begin{tikzcd}[column sep=large]
            TX && A \\
            & TY
            \arrow["{f^*}"', curve={height=6pt}, from=1-1, to=2-2]
            \arrow["{g^a}"', curve={height=6pt}, from=2-2, to=1-3]
            \arrow[""{name=0, anchor=center, inner sep=0}, "{(g^a f)^a}", from=1-1, to=1-3]
            \arrow["{\hat a_{g, f}}"', shorten <=3pt, shorten >=3pt, Rightarrow, from=0, to=2-2]
        \end{tikzcd}\]
        \item a natural family of invertible 2-cells $\tilde a_f \colon f \tto f^a i_X$ for each $f \colon JX \to A$ in $\E$,
        \[\begin{tikzcd}[column sep=large]
            JX && A \\
            & TX
            \arrow["{i_X}"', curve={height=6pt}, from=1-1, to=2-2]
            \arrow["{f^a}"', curve={height=6pt}, from=2-2, to=1-3]
            \arrow[""{name=0, anchor=center, inner sep=0}, "f", from=1-1, to=1-3]
            \arrow["{\tilde a_f}"', shorten <=3pt, shorten >=3pt, Rightarrow, from=0, to=2-2]
        \end{tikzcd}\]
    \end{enumerate}
    such that
    \begin{enumerate}[resume]
        \item \label{psalg-mu} for each $f \colon JX \to TY$, $g \colon JY \to TZ$, and $h \colon JZ \to A$ in $\E$, the following associativity law holds;
        \[\begin{tikzcd}[column sep=large]
            {((h^ag)^af)^a} && {(h^ag)^af^*} \\
            {((h^ag^*)f)^a} && {(h^ag^*)f^*} \\
            {(h^a(g^*f))^a} & {h^a(g^*f)^*} & {h^a(g^*f^*)}
            \arrow["{\hat a_{h^ag,f}}", from=1-1, to=1-3]
            \arrow["{\hat a_{h,g}f^*}", from=1-3, to=2-3]
            \arrow["\iso", from=2-3, to=3-3]
            \arrow["{(\hat a_{h,g}f)^a}"', from=1-1, to=2-1]
            \arrow["\iso"', from=2-1, to=3-1]
            \arrow["{\hat a_{h,g^*f}}"', from=3-1, to=3-2]
            \arrow["{h^a\mu_{g,f}}"', from=3-2, to=3-3]
        \end{tikzcd}\]
        \item \label{psalg-theta} for each $f \colon JX \to A$ in $\E$, the following diagram unit law holds.
        \[\begin{tikzcd}
            {f^a} & {(f^ai_X)^a} & {f^a{(i_X)}^*} \\
            && {f^a1_{TX}}
            \arrow["{(\tilde a_f)^a}", from=1-1, to=1-2]
            \arrow["\iso"', shift left, curve={height=12pt}, from=1-1, to=2-3]
            \arrow["{\hat a_{f,i_X}}", from=1-2, to=1-3]
            \arrow["{f^a\theta_X}", from=1-3, to=2-3]
        \end{tikzcd}\]
    \end{enumerate}
    A pseudoalgebra is \emph{strict} (or is a \emph{strict algebra}) if each $\tilde a_f$ and each $\hat a_{f, g}$ is an identity 2-cell.
\end{definition}

For $J$ a pseudofunctor between 2-categories, pseudoalgebras for $J$-relative pseudomonads were introduced by \textcite[Definition~4.34]{lewicki2020categories}; our definition coincides in that setting. (\Citeauthor{lewicki2020categories} does not define morphisms or transformations of pseudoalgebras.)

\begin{example}
    \label{cocomplete-categories-are-P-algebras}
    Let $A$ be a locally small category with a choice of small colimits (in the sense of \cref{COC}). Then $A$ may be equipped with the structure of a $P$-pseudoalgebra. For each functor $f \colon JX \to A$, one defines $f^a$ to be the pointwise left extension $(y_X \lx f) \colon PX \to A$ of $f$ along $y_A$, which sends each $p \in PX$ to the colimit of $(f \c \pi) \colon \El(p) \to X \to A$, where $\El(p)$ is the category of elements associated to $p$. The natural families of invertible 2-cells are both induced by the universal properties of the left extensions. It is straightforward to verify directly that this structure does indeed satisfy the axioms of a pseudoalgebra, though we will give a more abstract proof in \cref{presheaves-and-cocompleteness}.
\end{example}

\begin{example}
    When $\J$ is a functor between categories and $T$ is a $J$-relative monad (\cref{relative-monad}), a $T$-pseudoalgebra is the same as an \emph{EM-algebra} in the sense of \cite[Definition~2.11]{altenkirch2015monads}.
\end{example}

\begin{remark}
    \label{lax-algebras}
    \Cref{pseudoalgebra} introduces the notions of both strict algebras and pseudoalgebras. One could also define the notions of \emph{lax} and \emph{colax algebras} for a relative pseudomonad. In fact, historically, two-dimensional notions of algebras were first considered by \textcite{bunge1974coherent}, who introduced lax algebras for \emph{lax monads} (a generalisation of pseudomonads in which the structural natural isomorphisms are weakened to non-invertible natural transformations). In this paper, we shall only consider laxity for morphisms, not for pseudomonads or pseudoalgebras. There are several reasons for this choice. First, the theory in \cite{fiore2018relative} was developed only in the context of relative pseudomonads, and it would therefore be necessary first to revisit this development in the more general context of relative (co)lax monads to obtain the appropriate generalisations of the theorems herein (necessitating lax notions of two-dimensional relative adjunctions, Kleisli lax bicategories, and so on). Second, definitions in the fully lax setting are more complicated, requiring additional coherence axioms (see \cref{psalg-eta} and the discussion preceding it). Third, the existing literature on two-dimensional monad theory focuses primarily on pseudo and strict algebras. Fourth, and most importantly, we currently have no motivating examples of (co)lax algebras in the relative setting. Nonetheless, we do expect our proofs to generalise to the lax setting.
\end{remark}

As in the non-relative setting, there are several notions of morphism between pseudoalgebras.

\begin{definition}\label{pseudomorphism}
    Let $\J$ be a pseudofunctor.
    A \emph{lax morphism} from a $T$-pseudoalgebra $\algA$ to a $T$-pseudoalgebra $\algB$ comprises
    \begin{enumerate}
        \item a morphism $h \colon A \to B$ in $\E$;
        \item a natural family of 2-cells $\overline h_f \colon (hf)^b \tto hf^a$ for each $f \colon JX \to A$ in $\E$,
    \end{enumerate}
    such that
    \begin{enumerate}[resume]
        \item \label{lax-morph-hats} for each $f \colon JX \to TY$ and $g \colon JY \to A$ in $\E$, the following associativity law holds;
        \[\begin{tikzcd}[column sep=large]
        {((hg)^bf)^b} && {(hg)^bf^*} \\
        {((hg^a)f)^b} && {(hg^a)f^*} \\
        {(h(g^af))^b} & {h(g^af)^a} & {h(g^af^*)}
        \arrow["{\hat b_{hg,f}}", from=1-1, to=1-3]
        \arrow["{\overline h_gf^*}", from=1-3, to=2-3]
        \arrow["\iso", from=2-3, to=3-3]
        \arrow["{(\overline h_gf)^b}"', from=1-1, to=2-1]
        \arrow["\iso"', from=2-1, to=3-1]
        \arrow["{\overline h_{g^af}}"', from=3-1, to=3-2]
        \arrow["{h\hat a_{g,f}}"', from=3-2, to=3-3]
        \end{tikzcd}\]
        \item \label{lax-morph-tildes} for each $f \colon JX \to A$ in $\E$, the following unit law holds.
        \[\begin{tikzcd}
            hf & {(hf)^bi_X} & {(hf^a)i_X} \\
            && {h(f^ai_X)}
            \arrow["{\tilde b_{hf}}", from=1-1, to=1-2]
            \arrow["{\overline h_f i_X}", from=1-2, to=1-3]
            \arrow["\iso", from=1-3, to=2-3]
            \arrow["{h\tilde a_f}"', curve={height=12pt}, from=1-1, to=2-3]
        \end{tikzcd}\]
    \end{enumerate}
    $(h,\overline h)$ is a \emph{pseudomorphism} if each $\overline h_f$ is invertible; and is a \emph{strict morphism} if each $\overline h_f$ is an identity 2-cell. A \emph{colax morphism} is defined analogously, except that the direction of each $\overline h$ is reversed, hence involves a natural family of 2-cells $\overline h_f \colon h f^a \tto (hf)^b$ for each $f \colon JX \to A$ in $\E$.
\end{definition}

\begin{example}
    \label{cocontinuous-functors-are-P-morphisms}
    In the context of \cref{cocomplete-categories-are-P-algebras}, every functor $f \colon A \to B$ between locally small categories may be equipped with the structure of a lax morphism of $P$-pseudoalgebras, the natural family of 2-cells being induced by the universal property of the left extensions. (This is a special property of the relative pseudomonad $P$, which we examine further in \cref{lax-idempotence,presheaves-and-cocompleteness}; in general, it is not true that, for a relative pseudomonad $T$, every 1-cell in $\E$ admits the structure of a lax morphism.) If $f$ preserves small colimits, the lax morphism is furthermore a pseudomorphism.
\end{example}

There is a single notion of 2-cell between morphisms of pseudoalgebras (modulo the direction of the data involved in a lax/colax morphism).

\begin{definition}
    \label{transformation}
    Let $(h, \overline h)$ and $(h', \overline{h'})$ be lax morphisms from $\algA$ to $\algB$. A \emph{transformation} from $(h, \overline h)$ to $(h', \overline{h'})$ is a 2-cell $\alpha \colon h \tto h'$ rendering the following diagram commutative for each $f \colon JX \to A$ in $\E$.
    \[\begin{tikzcd}
        {(hf)^b} & {(h'f)^b} \\
        {hf^a} & {h'f^a}
        \arrow["{(\alpha f)^b}", from=1-1, to=1-2]
        \arrow["{\overline{h'}_f}", from=1-2, to=2-2]
        \arrow["{\overline h_f}"', from=1-1, to=2-1]
        \arrow["{\alpha f^a}"', from=2-1, to=2-2]
    \end{tikzcd}\]
    Transformations between colax morphisms are defined in the same way, modulo reversing the horizontal arrows above.
\end{definition}

\begin{example}
    \label{natural-transformations-are-P-transformations}
    In the context of \cref{cocomplete-categories-are-P-algebras,cocontinuous-functors-are-P-morphisms}, every natural transformation between (cocontinuous) functors is a transformation of $P$-pseudoalgebras. Intuitively, this is because every natural transformation automatically coheres with colimit structure.
\end{example}

\begin{theorem}
    \label{AlgT}
    Let $\J$ be a pseudofunctor and let $T$ be a $J$-relative pseudomonad. The \mbox{$T$-pseudoalgebras} (\cref{pseudoalgebra}), lax morphisms (\cref{pseudomorphism}), and transformations (\cref{transformation}) form a bicategory $\AlgTl$, the \emph{bicategory of $T$-pseudoalgebras and lax morphisms}.
\end{theorem}

\begin{proof}
    We first show that, for any two pseudoalgebras $\algA$ and $\algB$, the lax morphisms from $\algA$ to $\algB$, together with their transformations, form a category. Since transformations are 2-cells in the underlying bicategory $\E$ satisfying a compatibility property, it suffices to show that this property is preserved by identities and composition in the hom-category $\E[A, B]$. This is indeed true: for lax morphisms $h, h', h'' \colon \algA \to \algB$, and transformations $\alpha \colon h \tto h'$ and $\alpha' \colon h' \tto h''$, the following diagrams commute by functoriality of $\ph^b$, exhibiting $1_h \colon h \tto h$ and $\alpha'\alpha \colon h \tto h''$ as transformations.
    \[
    \begin{tikzcd}
        {(hf)^b} & {(hf)^b} \\
        {hf^a} & {hf^a}
        \arrow["{(1_hf)^b}", curve={height=-12pt}, from=1-1, to=1-2]
        \arrow[Rightarrow, no head, from=1-1, to=1-2]
        \arrow["{\overline h_f}"', from=1-1, to=2-1]
        \arrow["{\overline h_f}", from=1-2, to=2-2]
        \arrow["{1_hf^a}"', curve={height=12pt}, from=2-1, to=2-2]
        \arrow[Rightarrow, no head, from=2-1, to=2-2]
    \end{tikzcd}
    \hspace{4em}
    \begin{tikzcd}[column sep=large, row sep = 2.3em]
        {(hf)^b} & {(h'f)^b} & {(h''f)^b} \\
        {hf^a} & {h'f^a} & {h''f^a}
        \arrow["{(\alpha f)^b}"{description}, from=1-1, to=1-2]
        \arrow["{\overline{h'}_f}"{description}, from=1-2, to=2-2]
        \arrow["{\overline h_f}"', from=1-1, to=2-1]
        \arrow["{\alpha f^a}"{description}, from=2-1, to=2-2]
        \arrow["{(\alpha' f)^b}"{description}, from=1-2, to=1-3]
        \arrow["{\overline{h''}_f}", from=1-3, to=2-3]
        \arrow["{\alpha' f^a}"{description}, from=2-2, to=2-3]
        \arrow["{\alpha'\alpha f^a}"', curve={height=18pt}, from=2-1, to=2-3]
        \arrow["{(\alpha'\alpha f)^b}", curve={height=-18pt}, from=1-1, to=1-3]
    \end{tikzcd}
    \]

    Next, we observe that, for each $T$-pseudoalgebra $\algA$, the family of 2-cells,
    \begin{equation}
        \label{identity-lax-morphism}
        (\overline{1_A})_f \defeq (1_Af)^a \xto\iso f^a \xto\iso 1_Af^a
    \end{equation}
    for each 1-cell $f \colon JX \to A$ in $\E$, equips the identity 1-cell $1_A \colon A \to A$ in $\E$ with the structure of a pseudomorphism $\algA \to \algA$.
    There are two coherence conditions to check -- (\ref{lax-morph-hats} \& \ref{lax-morph-tildes}) -- and these hold by commutativity of the following diagrams respectively.
    \[\begin{tikzcd}
        {((1_Ag)^af)^a} &&& {(1_{A} g)^af^*} \\
        {(g^af)^a} &&& {g^af^*} \\
        {((1_Ag^a)f)^a} &&& {(1_A g^a)f^*} \\
        {(1_A(g^af))^a} & {(g^af)^a} & {1_A(g^af)^a} & {1_A(g^af^*)}
        \arrow[""{name=0, anchor=center, inner sep=0}, "{\hat a_{1_A g, f}}", from=1-1, to=1-4]
        \arrow["\iso"', from=1-1, to=2-1]
        \arrow["\iso", from=1-4, to=2-4]
        \arrow["\iso"', from=2-1, to=3-1]
        \arrow[curve={height=-12pt}, Rightarrow, no head, from=2-1, to=4-2]
        \arrow["\iso", from=2-4, to=3-4]
        \arrow["\iso"{description}, curve={height=36pt}, from=2-4, to=4-4]
        \arrow["\iso"', from=3-1, to=4-1]
        \arrow["\iso", from=3-4, to=4-4]
        \arrow["\iso"', from=4-1, to=4-2]
        \arrow[""{name=1, anchor=center, inner sep=0}, "{\hat a_{g,f}}"{description}, curve={height=-12pt}, from=4-2, to=2-4]
        \arrow["\iso"', from=4-2, to=4-3]
        \arrow["{1_A\hat a_{g,h}}"', from=4-3, to=4-4]
        \arrow["{\natof\hat{a}}"'{pos=0.4}, draw=none, from=0, to=1]
    \end{tikzcd}\]
    \[\begin{tikzcd}
    	& {(1_{A}f)^ai_X} \\
    	{1_{A}f} & f & {f^ai_X} \\
    	{1_{A}(f^ai_X)} && {(1_{A}f^a)i_X}
    	\arrow["{\natof\tilde a}"{description}, draw=none, from=1-2, to=2-2]
    	\arrow["\iso", from=1-2, to=2-3]
    	\arrow["{\tilde a_{1_{A}f}}", from=2-1, to=1-2]
    	\arrow["\iso"{description}, from=2-1, to=2-2]
    	\arrow["{1_{A}\tilde a_f}"', from=2-1, to=3-1]
    	\arrow["{\tilde a_f}"{description}, from=2-2, to=2-3]
    	\arrow["\iso"{description}, from=2-3, to=3-1]
    	\arrow["\iso", from=2-3, to=3-3]
    	\arrow["\iso", from=3-3, to=3-1]
    \end{tikzcd}\]

    For each pair of lax morphisms $(h, \overline h) \colon \algA \to \algB$ and $(h', \overline{h'}) \colon \algB \to \algC$, and each $f \colon JX \to A$ in $\E$, the family of 2-cells
    \begin{equation}
        \label{composite-lax-morphism}
        \overline{h'h}_f \defeq ((h'h)f)^c \xto\iso (h'(hf))^c \xto{\overline{h'}_{hf}} h'(hf)^b \xto{h'\overline h_f} h'(hf^a) \xto\iso (h'h)f^a
    \end{equation}
    equips the composite 1-cell $h'h \colon A \to C$ in $\E$ with the structure of a lax morphism $\algA \to \algC$. There are two coherence conditions to check. To show the associativity law \eqref{lax-morph-hats} we use the corresponding law for each of $h$ and $h'$, as shown in the following diagram.
    \[\begin{tikzcd}
        {(((h'h)g)^cf)^c} &&&& {((h'h)g)^cf^*} \\
        {((h'(hg))^cf)^c} &&&& {(h'(hg))^cf^*} \\
        {((h'(hg)^b)f)^c} & {(h'((hg)^bf))^c} & {h'((hg)^bf)^b} & {h'((hg)^bf^*)} & {(h'(hg)^b)f^*} \\
        {((h'(hg^a))f)^c} & {(h'((hg^a)f))^c} & {h'((hg^a)f)^b} & {h'((hg^a)f^*)} & {(h'(hg^a))f^*} \\
        {(((h'h)g^a)f)^c} &&&& {((h'h)g^a)f^*} \\
        {((h'h)(g^af))^c} &&& {h'(h(g^af^*))} \\
        {(h'(h(g^af)))^c} \\
        {h'(h(g^af))^b} \\
        {h'(h(g^af)^a)} && {(h'h)(g^af)^a} && {(h'h)(g^af^*)}
        \arrow[""{name=0, anchor=center, inner sep=0}, "{\hat c_{(h'h)g,f}}", from=1-1, to=1-5]
        \arrow["\iso"', from=1-1, to=2-1]
        \arrow["\iso", from=1-5, to=2-5]
        \arrow[""{name=1, anchor=center, inner sep=0}, "{\hat c_{h'(hg),f}}"{description}, from=2-1, to=2-5]
        \arrow[""{name=2, anchor=center, inner sep=0}, "{(\overline{h'}_{hg}f)^c}"', from=2-1, to=3-1]
        \arrow[""{name=3, anchor=center, inner sep=0}, "{\overline{h'}_{hg}f^*}", from=2-5, to=3-5]
        \arrow["\iso"{description}, from=3-1, to=3-2]
        \arrow["{((h'\overline h_g)f)^c}"', from=3-1, to=4-1]
        \arrow["{\overline{h'}_{(hg)^bf}}", from=3-2, to=3-3]
        \arrow[""{name=4, anchor=center, inner sep=0}, "{(h'(\overline h_gf))^c}"{description}, from=3-2, to=4-2]
        \arrow["{h'\hat b_{hg,f}}", from=3-3, to=3-4]
        \arrow[""{name=5, anchor=center, inner sep=0}, "{h'(\overline h_gf)^b}"{description}, from=3-3, to=4-3]
        \arrow["\iso"{description}, from=3-4, to=3-5]
        \arrow["{h'(\overline h_gf^*)}"{description}, from=3-4, to=4-4]
        \arrow["{(h'\overline h_g)f^*}", from=3-5, to=4-5]
        \arrow["\iso"{description}, from=4-1, to=4-2]
        \arrow["\iso"', from=4-1, to=5-1]
        \arrow["{\overline{h'}_{(hg^a)f}}"', from=4-2, to=4-3]
        \arrow[""{name=6, anchor=center, inner sep=0}, "\iso"{description}, curve={height=-12pt}, from=4-2, to=7-1]
        \arrow[""{name=7, anchor=center, inner sep=0}, "\iso"{description}, curve={height=-18pt}, from=4-3, to=8-1]
        \arrow["\iso"{description}, from=4-4, to=4-5]
        \arrow["\iso"{description}, from=4-4, to=6-4]
        \arrow["\iso", from=4-5, to=5-5]
        \arrow["\iso"', from=5-1, to=6-1]
        \arrow["\iso", from=5-5, to=9-5]
        \arrow["\iso"', from=6-1, to=7-1]
        \arrow[""{name=8, anchor=center, inner sep=0}, "\iso"{description}, curve={height=6pt}, from=6-4, to=9-5]
        \arrow["{\overline{h'}_{h(g^af)}}"', from=7-1, to=8-1]
        \arrow["{h'\overline h_{g^ah}}"', from=8-1, to=9-1]
        \arrow[""{name=9, anchor=center, inner sep=0}, "{h'(h\hat a_{g,f})}"{description}, curve={height=12pt}, from=9-1, to=6-4]
        \arrow["\iso"', from=9-1, to=9-3]
        \arrow["{(h'h)\hat a_{g,f}}"', from=9-3, to=9-5]
        \arrow["{\natof\hat{c}}"{description}, draw=none, from=0, to=1]
        \arrow["{\algmu{h'}}"{description}, draw=none, from=2, to=3]
        \arrow["{\natof\overline {h'}}"{description}, draw=none, from=4, to=5]
        \arrow["{\natof\overline{h'}}"{description}, draw=none, from=6, to=7]
        \arrow["{\algmu h}"{description}, draw=none, from=7, to=6-4]
        \arrow[draw=none, from=9, to=8]
    \end{tikzcd}\]
    Similarly, for the unit law \eqref{lax-morph-tildes} we use the corresponding unit laws for $h$ and $h'$.
    \[\begin{tikzcd}
        {(h'h)f} &&& {((h'h)f)^ci_X} \\
        & {h'(hf)} && {(h'(hf))^ci_X} \\
        && {h'((hf)^bi_X)} & {(h'(hf)^b)i_X} \\
        && {h'((hf^a)i_X)} & {(h'(hf^a))i_X} \\
        & {h'(h(f^ai_X))} \\
        {(h'h)(f^ai_X)} &&& {((h'h)f^a)i_X}
        \arrow["{\tilde c_{(h'h)f}}", from=1-1, to=1-4]
        \arrow[""{name=0, anchor=center, inner sep=0}, "\iso"{description}, from=1-1, to=2-2]
        \arrow["{(h'h)\tilde a_f}"', from=1-1, to=6-1]
        \arrow[""{name=1, anchor=center, inner sep=0}, "\iso"{description}, from=1-4, to=2-4]
        \arrow[""{name=2, anchor=center, inner sep=0}, "{\tilde c_{h'(hf)}}"{description}, from=2-2, to=2-4]
        \arrow["{h' \tilde b_{hf}}"{description}, from=2-2, to=3-3]
        \arrow[""{name=3, anchor=center, inner sep=0}, "{h(h'\tilde a_f)}"{description}, from=2-2, to=5-2]
        \arrow["{\overline{h'}_{hf} i_X}"{description}, from=2-4, to=3-4]
        \arrow[""{name=4, anchor=center, inner sep=0}, "\iso"{description}, from=3-3, to=3-4]
        \arrow[""{name=5, anchor=center, inner sep=0}, "{h'(\overline{h}_f i_X)}"{description}, from=3-3, to=4-3]
        \arrow["{(h' \overline{h}_f)i_X}"{description}, from=3-4, to=4-4]
        \arrow["\iso"{description}, from=4-3, to=4-4]
        \arrow["\iso"{description}, from=4-3, to=5-2]
        \arrow["\iso"{description}, from=4-4, to=6-4]
        \arrow["\iso"{description}, from=6-1, to=5-2]
        \arrow["\iso", from=6-4, to=6-1]
        \arrow["{\natof\tilde c}"{description}, draw=none, from=0, to=1]
        \arrow["{\morunit{h'}}"{description}, draw=none, from=2, to=4]
        \arrow["{\morunit{h}}"{description}, draw=none, from=3, to=5]
        %
		%
        \arrow[
                "{\overline{h'h}_f i_X}",
                rounded corners,
                to path=
                {
                -- ([xshift=1.8cm]\tikztostart.center)
                -- ([xshift=1.8cm]\tikztotarget.center)
                \tikztonodes
                -- (\tikztotarget.east)
                },
                from=1-4, to=6-4
            ]
    \end{tikzcd}
    \]

    Now let $\alpha \colon (h, \overline h) \tto (k, \overline k)$ and $\alpha' \colon (h', \overline{h'}) \tto (k', \overline{k'})$ be transformations. We wish to show the horizontal composite $\alpha' \cdot \alpha$ is a transformation $(h'h, \overline{h'h}) \tto (k'k, \overline{k'k})$. This follows from commutativity of the following diagram.
    \[\begin{tikzcd}[column sep=5em]
    	{((h'h)f)^c} & {((k'h)f)^c} & {((k'k)f)^c} \\
    	{(h'(hf))^c} & {(k'(hf))^c} & {(k'(kf))^c} \\
    	{h'(hf)^b} & {k'(hf)^b} & {k'(kf)^b} \\
    	{h'(hf^a)} & {k'(h f^a)} & {k'(k f^a)} \\
    	{(h'h)f^a} & {(k'h)f^a} & {(k'k)f^a}
    	\arrow["{((\alpha' h)f)^c}"{description}, from=1-1, to=1-2]
    	\arrow["{((\alpha'\cdot\alpha)f)^c}", curve={height=-30pt}, from=1-1, to=1-3]
    	\arrow["\iso"{description}, from=1-1, to=2-1]
    	\arrow["{\overline{h'h}_f}"', curve={height=50pt, pos=.2}, from=1-1, to=5-1]
    	\arrow["{((k'\alpha)f)^c}"{description}, from=1-2, to=1-3]
    	\arrow["\iso"{description}, from=1-2, to=2-2]
    	\arrow["\iso"{description}, from=1-3, to=2-3]
    	\arrow["{\overline{k'k}_f}", curve={height=-50pt, pos=.2}, from=1-3, to=5-3]
    	\arrow[""{name=0, anchor=center, inner sep=0}, "{(\alpha' (hf))^c}"{description}, from=2-1, to=2-2]
    	\arrow["{\overline{h'}_{hf}}"{description}, from=2-1, to=3-1]
    	\arrow[""{name=1, anchor=center, inner sep=0}, "{(k'(\alpha f))^c}"{description}, from=2-2, to=2-3]
    	\arrow["{\overline{k'}_{hf}}"{description}, from=2-2, to=3-2]
    	\arrow["{\overline{k'}_{kf}}"{description}, from=2-3, to=3-3]
    	\arrow[""{name=2, anchor=center, inner sep=0}, "{\alpha' (hf)^b}"{description}, from=3-1, to=3-2]
    	\arrow["{h'\overline h_f}"{description}, from=3-1, to=4-1]
    	\arrow[""{name=3, anchor=center, inner sep=0}, "{k'(\alpha f)^b}"{description}, from=3-2, to=3-3]
    	\arrow["{k'\overline h_f}"{description}, from=3-2, to=4-2]
    	\arrow["{k'\overline k_f}"{description}, from=3-3, to=4-3]
    	\arrow["{\alpha' (h f^a)}"{description}, from=4-1, to=4-2]
    	\arrow["\iso"{description}, from=4-1, to=5-1]
    	\arrow[""{name=4, anchor=center, inner sep=0}, "{k'(\alpha f^a)}"{description}, from=4-2, to=4-3]
    	\arrow["\iso"{description}, from=4-2, to=5-2]
    	\arrow["\iso"{description}, from=4-3, to=5-3]
    	\arrow["{(\alpha' h)f^a}"{description}, from=5-1, to=5-2]
    	\arrow["{(\alpha'\cdot\alpha)f^a}"', curve={height=30pt}, from=5-1, to=5-3]
    	\arrow["{(k'\alpha)f^a}"{description}, from=5-2, to=5-3]
    	\arrow["{\trans{\alpha'}}"{description}, draw=none, from=0, to=2]
    	\arrow["{\natof\overline{k'}}"{description}, draw=none, from=1, to=3]
    	\arrow["{\trans{\alpha}}"{description}, draw=none, from=3, to=4]
    \end{tikzcd}\]
    Consequently, the composition functor ${\c} \colon \E[B, C] \times \E[A, B] \to \E[A, C]$ lifts to a functor
    \begin{align*}
        {\c} \colon \AlgTl[\algB, & \algC] \times \AlgTl[\algA, \algB] \\ &\hspace{3.5cm} \longrightarrow \AlgTl[\algA, \algC]
    \end{align*}
    on hom-categories of lax morphisms and transformations.

    The associators and unitors for $\AlgTl$ are inherited from $\E$: it thus remains to show that these are transformations componentwise. This follows from the commutativity of the following three diagrams, for the unitors and associator respectively.
    \[
    \begin{tikzcd}
        {((h1_A)f)^b} & {(hf)^b} \\
        {(h(1_A f))^b} \\
        {h(1_A f)^a} \\
        {hf^a} \\
        {h(1_A f^a)} \\
        {(h1_A)f^a} & {hf^a}
        \arrow["\iso", from=1-1, to=1-2]
        \arrow["\iso"{description}, from=1-1, to=2-1]
        \arrow["{\overline{h 1_A}_f}"', curve={height=50pt,pos=.2}, from=1-1, to=6-1]
        \arrow[""{name=0, anchor=center, inner sep=0}, "{\overline h_f}", from=1-2, to=6-2]
        \arrow["\iso"{description}, from=2-1, to=1-2]
        \arrow["{\overline h_{1f}}"{description}, from=2-1, to=3-1]
        \arrow[""{name=1, anchor=center, inner sep=0}, "\iso"{description}, from=3-1, to=4-1]
        \arrow["\iso"{description}, from=4-1, to=5-1]
        \arrow[Rightarrow, no head, from=4-1, to=6-2]
        \arrow["\iso"{description}, from=5-1, to=6-1]
        \arrow["\iso"', from=6-1, to=6-2]
        \arrow["{\natof\overline h}"{description}, draw=none, from=1, to=0]
    \end{tikzcd}
    \hspace{6em}
    \begin{tikzcd}
    	{(hf)^b} & {((1_B h)f)^b} \\
    	& {(1_B(hf))^b} \\
    	& {(hf)^b} \\
    	& {1_B(hf)^b} \\
    	& {1_B(hf^a)} \\
    	{hf^a} & {(1_B h)f^a}
    	\arrow["{\overline h_f}"', from=1-1, to=6-1]
    	\arrow["\iso"', from=1-2, to=1-1]
    	\arrow["\iso"{description}, from=1-2, to=2-2]
    	\arrow["{\overline{1_B h}_f}", curve={height=-50pt, pos = .2}, from=1-2, to=6-2]
    	\arrow["\iso"{description}, from=2-2, to=3-2]
    	\arrow[Rightarrow, no head, from=3-2, to=1-1]
    	\arrow["\iso"{description}, from=3-2, to=4-2]
    	\arrow["{1_B\overline h_f}"{description}, from=4-2, to=5-2]
    	\arrow["\iso"{description}, from=5-2, to=6-1]
    	\arrow["\iso"{description}, from=5-2, to=6-2]
    	\arrow["\iso", from=6-2, to=6-1]
    \end{tikzcd}\]
    \[\begin{tikzcd}
        {\left( ((fg)h)k \right)^d} & {\left( (f(gh)) k \right)^d} \\
        {\left( (fg)(hk) \right)^d} & {\left( f ((gh) k) \right)^d} \\
        {\left( f (g (hk)) \right)^d} & {f \left(((gh) k) \right)^c} \\
        {f (g (hk))^c} & {f \left( g (h k) \right)^c} \\
        {f \left( g (hk)^b \right)} & {f \left( g (h k)^b \right)} \\
        {(fg)(hk)^b} & {f \left( g (h k^a) \right)} \\
        {(fg)(hk^a)} & {f \left((gh) k^a \right)} \\
        {\left( (fg)h \right) k^a} & {\left( f(gh) \right) k^a}
        \arrow["\iso", from=1-1, to=1-2]
        \arrow["\iso"', from=1-1, to=2-1]
        \arrow["\iso", from=1-2, to=2-2]
        \arrow["\iso"{description}, from=2-1, to=3-1]
        \arrow["{\overline{fg}_{hk}}"', curve={height=50pt, pos=.2}, from=2-1, to=6-1]
        \arrow["{\overline{f}_{(gh)k}}", from=2-2, to=3-2]
        \arrow[""{name=0, anchor=center, inner sep=0}, "\iso"{description}, from=3-1, to=2-2]
        \arrow["{\overline{f}_{g(hk)}}"{description}, from=3-1, to=4-1]
        \arrow["\iso"{description}, from=3-2, to=4-2]
        \arrow["{f \overline{gh}_k}", curve={height=-50pt, pos=.2}, from=3-2, to=7-2]
        \arrow[""{name=1, anchor=center, inner sep=0}, Rightarrow, no head, from=4-1, to=4-2]
        \arrow["{f \overline{g}_{hk}}"{description}, from=4-1, to=5-1]
        \arrow["{f \overline{g}_{hk}}"{description}, from=4-2, to=5-2]
        \arrow[Rightarrow, no head, from=5-1, to=5-2]
        \arrow["\iso"{description}, from=5-1, to=6-1]
        \arrow["{f (g \overline{h}_k)}"{description}, from=5-2, to=6-2]
        \arrow["{(fg)\overline{h}_k}"', from=6-1, to=7-1]
        \arrow["\iso"{description}, from=6-2, to=7-2]
        \arrow["\iso"{description}, from=7-1, to=7-2]
        \arrow["\iso"', from=7-1, to=8-1]
        \arrow["\iso", from=7-2, to=8-2]
        \arrow["\iso"', from=8-1, to=8-2]
        \arrow["{\natof\overline{f}}"{description, pos=0.6}, draw=none, from=0, to=1]
    \end{tikzcd}\qedshift\]
\end{proof}

\begin{proposition}
    \label{AlgT-variants}
    By restricting the 1-cells of $\AlgTl$, we obtain:
    \begin{enumerate}
        \item a bicategory $\AlgT$ of \emph{$T$-pseudoalgebras and pseudomorphisms}, equipped with a locally \ff{} strict pseudofunctor $\AlgT \to \AlgTl$;
        \item assuming $\E$ is a 2-category, a 2-category $\AlgTs$ of \emph{$T$-pseudoalgebras and strict morphisms}, equipped with a locally \ff{} strict 2-functor $\AlgTs \to \AlgT$.
    \end{enumerate}
    Furthermore, by considering instead \emph{colax} morphisms in \cref{AlgT}, we obtain
    \begin{enumerate}[resume]
        \item a bicategory $\AlgTc$ of \emph{$T$-pseudoalgebras and colax morphisms}, together with a locally \ff{} strict pseudofunctor $\AlgT \to \AlgTc$.
    \end{enumerate}
    By restricting each of the above to strict algebras, we obtain full sub-bicategories $\StrAlgTl \ffto \AlgTl$, $\StrAlgT \ffto \AlgT$, $\StrAlgTc \ffto \AlgTc$; and, assuming $\E$ is a 2-category, a 2-category $\StrAlgTs \ffto \AlgTs$.
\end{proposition}

\begin{proof}
    For (1), we only need check that pseudomorphisms are closed under composition, since identities are pseudomorphisms by definition. However, this is trivial, since each of the components in \eqref{composite-lax-morphism} is then invertible. (2) is similar, except that for identities and composites to be strict, we furthermore require the structural isomorphisms in $\E$ to be identities, so that \eqref{identity-lax-morphism} and \eqref{composite-lax-morphism} are identities. (3) follows by duality.
\end{proof}

\begin{remark}
    \label{EM-2-category}
    When $\E$ is a 2-category, each of the bicategories of pseudoalgebras in \cref{AlgT,AlgT-variants} is a 2-category, since the structural isomorphisms are inherited from $\E$. This is in contrast to $\KlT$, which is typically a bicategory when $T$ is not a 2-monad, even when $\E$ is a 2-category. This observation will play an important role in \cref{resolutions-and-coherence}.
\end{remark}

\begin{example}
    Following \cref{EM-2-category}, when $J$ is a 2-functor between 2-categories and $T$ is a relative 2-monad, $\StrAlgTs$ is the 2-category of $\Cat$-enriched algebras in the sense of \cite{arkor2024formal}.
\end{example}

For the next proposition, recall that an amnestic isofibration is a functor $u \colon C \to D$ with the property that, for every isomorphism $\phi \colon d \to u(c)$ in $D$ there exists a unique isomorphism $\phi' \colon d' \to c$ in $C$ such that $u(\phi') = \phi$~\cite[Lemma~2.4.7]{avery2017structure}.

\begin{proposition}
    \label{forgetful-pseudofunctor}
    The function sending each pseudoalgebra to its underlying object admits the structure of a strict pseudofunctor $U_T \colon \AlgTl \to \E$, which is locally faithful, locally conservative, and locally an amnestic isofibration.
\end{proposition}

\begin{proof}
    For each pair of $T$-pseudoalgebras $\algA$ and $\algB$, we define $(U_T)_{A, B} \colon \AlgTl[\algA, \algB] \to \E[A, B]$ to be the faithful functor sending a lax morphism $(h, \overline h)$ to the 1-cell $h$.
    For conservativity, let $\alpha \colon (h, \overline h) \tto (h', \overline{h'})$ be a transformation of lax algebra morphisms $\algA \to \algB$, and suppose that the 2-cell $\alpha$ is invertible in $\E$. Then $\alpha\inv$ is also a transformation: indeed, the transformation axiom \eqref{transformation} below follows immediately from the corresponding axiom for $\alpha$.
    \[\begin{tikzcd}[column sep=4em]
    	{(h'f)^b} & {(h f)^b} \\
    	{h'f^a} & {h f^a}
    	\arrow["{(\alpha\inv f)^b}", from=1-1, to=1-2]
    	\arrow["{\overline{h'}_f}"', from=1-1, to=2-1]
    	\arrow["{\overline{h}_f}", from=1-2, to=2-2]
    	\arrow["{\alpha\inv f^a}"', from=2-1, to=2-2]
    \end{tikzcd}\]

    To see that $U_T$ locally an amnestic isofibration, let $(h', \overline{h'})$ be a lax morphism $\algA \to \algB$ and let $h \colon A \to B$ be a 1-cell in $\E$. Suppose also that there is an invertible 2-cell $\alpha \colon h \tto h'$ in $\E$.
    Define a lax morphism structure on $h$ by setting
    \begin{equation}
    \label{eq:induced-by-amnestic}
    \overline{h}_f \defeq (hf)^b \xto{(\alpha f)^b} (h'f)^b \xto{\overline{h'}_f} h'f^a \xto{\alpha\inv f^a} hf^a
    \end{equation}
    for each $f \colon JX \to A$. The unit and associativity axioms hold by the commutativity of the next two diagrams.
    \[\begin{tikzcd}[column sep=huge]
		hf & {(hf)^b i_X} \\
        {h' f} & {(h' f)^b i_X} \\
        & {h' f^a i_X} \\
        {hf } & {h f^a i_X}
        \arrow[""{name=0, anchor=center, inner sep=0}, "{\tilde{b}_{hf}}", from=1-1, to=1-2]
        \arrow["{\alpha f}"{description}, from=1-1, to=2-1]
        \arrow["{(\alpha f)^b i_X}"{description}, from=1-2, to=2-2]
        \arrow[""{name=1, anchor=center, inner sep=0}, "{\tilde{b}_{h'f}}"{description}, from=2-1, to=2-2]
        \arrow[""{name=2, anchor=center, inner sep=0}, "{h' \tilde{a}_f}"{description}, curve={height=12pt}, from=2-1, to=3-2]
        \arrow["{\alpha\inv f}"{description}, from=2-1, to=4-1]
        \arrow["{\overline{h'}_f i_X}"{description}, from=2-2, to=3-2]
        \arrow["{\alpha\inv f^a i_X}"{description}, from=3-2, to=4-2]
        \arrow["{h \tilde{a}_f}"', from=4-1, to=4-2]
        \arrow["{\natof\tilde{b}}"{description}, draw=none, from=0, to=1]
        \arrow["{\morunit{h'}}"{description, pos=0.4}, draw=none, from=2-2, to=2]
		%
	    %
	    \arrow[
	            "{\overline{h}_f i_X }",
	            rounded corners,
	            to path=
	            { -- ([xshift=1cm]\tikztostart.center)
	            -- ([xshift=1cm]\tikztotarget.center)
	   \tikztonodes
	            -- (\tikztotarget.east)
	            },
	            from=1-2, to=4-2
	        ]
	    \arrow[
	        "{1_{hf}}"',
	        rounded corners,
	        to path=
	        { -- ([xshift=-.8cm]\tikztostart.center)
	        -- ([xshift=-.8cm]\tikztotarget.center)
	        \tikztonodes
	        -- (\tikztotarget.west)
	        },
	        from=1-1, to=4-1
	    ]
	\end{tikzcd}
    \hspace{.8em}
    \begin{tikzcd}[row sep = 2.1em]
    	&& {(hg)^b f^*} \\
    	{((hg)^b f)^b} && {((h' g)^b f)^b} && {(h' g)^b f^*} \\
    	{(h g^a f)^b} && {(h' g^a f)^b} \\
    	&& {(h' g^a f)^b} \\
    	{h (g^a f)^a} && {h' (g^a f)^a} && {h' g^a f^*} \\
    	&& {h g^a f^*}
    	\arrow["{\natof\hat{b}}"{description}, draw=none, from=1-3, to=2-3]
    	\arrow["{(\alpha g)^b f}", from=1-3, to=2-5]
    	\arrow["{\hat{b}_{hg, f}}", from=2-1, to=1-3]
    	\arrow["{((\alpha g)^b f)^b}"{description}, from=2-1, to=2-3]
    	\arrow["{(\overline{h}_g f)^b}"', from=2-1, to=3-1]
    	\arrow[""{name=0, anchor=center, inner sep=0}, "{\hat{b}_{h'g, f}}"{description}, from=2-3, to=2-5]
    	\arrow["{(\overline{h'}_g f)^b}"{description}, from=2-3, to=3-3]
    	\arrow["{\overline{h'}_g f^*}", from=2-5, to=5-5]
    	\arrow["{(\alpha g^a f)^b}"{description}, from=3-1, to=4-3]
    	\arrow["{\tilde{h}_{g^a f}}"', from=3-1, to=5-1]
    	\arrow["{(\alpha\inv g^a f)^b}"{description}, from=3-3, to=3-1]
    	\arrow[Rightarrow, no head, from=3-3, to=4-3]
    	\arrow["{\overline{h'}_{g^a f}}"{description}, from=4-3, to=5-3]
    	\arrow["{h \hat{a}_{g, f}}"', from=5-1, to=6-3]
    	\arrow["{\alpha\inv (g^a f)^a}"{description}, from=5-3, to=5-1]
    	\arrow[""{name=1, anchor=center, inner sep=0}, "{h' \hat{a}_{g, f}}"{description}, from=5-3, to=5-5]
    	\arrow["{\alpha\inv g^a f^*}", from=5-5, to=6-3]
    	\arrow["{\morassoc{h'}}"{description}, draw=none, from=0, to=1]
    \end{tikzcd}\]
    Moreover, this definition of $\overline h$ makes $\alpha$ a transformation. Since $U_T$ is locally faithful, any 2-cell $(h, \overline h) \tto (h', \overline{h'})$ lying over $\alpha$ must be equal to $\alpha$.
    Thus $(U_T)_{A, B}$ is an amnestic isofibration. That this defines a strict pseudofunctor follows by inspecting the proof of \cref{AlgT}, where the identities, composites, and structural isomorphisms in $\AlgT$ are inherited from those in $\E$.
\end{proof}

\begin{definition}
    We also denote by $U_T \colon \AlgT \to \E$ the locally faithful, locally conservative, locally amnestic isofibrant, strict pseudofunctor $\AlgT \to \AlgTl \to \E$ given by composing the pseudofunctors defined in \cref{AlgT-variants,forgetful-pseudofunctor}, and similarly for the other variants.
\end{definition}

\subsection{Comparison with no-iteration pseudoalgebras}
\label{sec:comparison-with-no-iteration-pseudoalgebras}

To justify our definitions it is necessary to show that, when $J$ is the identity pseudofunctor, we recover the usual notions of pseudoalgebra and their morphisms and transformations. We therefore compare our definitions to the corresponding notions for no-iteration pseudomonads. No-iteration pseudomonads were introduced in \cite{marmolejo2013no} as an alternative presentation of pseudomonads in terms of an extension operator $\E[{-}_1, T{-}_2] \tto \E[T{-}_1, T{-}_2]$ rather than a multiplication operator $TT \tto T$. As shown in \cite[Proposition~3.3]{fiore2018relative}, these are precisely pseudomonads relative to the identity pseudofunctor. Here, we shall show that this relationship extends to the bicategories of pseudoalgebras.

We first establish a lemma that will be useful in later calculations. Conceptually, this corresponds to the additional axiom that would need to be imposed in the definition of a lax algebra for a relative pseudomonad (see \cref{lax-algebras}). However, for pseudoalgebras it is automatic.

\begin{lemma}
    \label{psalg-eta}
    Let $\algA$ be a pseudoalgebra for a $(\J)$-relative pseudomonad $T$. For each $f \colon JX \to TY$ and $g \colon JY \to A$ in $\E$, the following diagram commutes.
    \[\begin{tikzcd}[column sep=large]
        {g^af} & {(g^af)^ai_X} & {(g^af^*)i_X} \\
        && {g^a(f^*i_X)}
        \arrow["{\tilde a_{g^af}}", from=1-1, to=1-2]
        \arrow["{g^a\eta_f}"', curve={height=12pt}, from=1-1, to=2-3]
        \arrow["{\hat a_{g,f}i_X}", from=1-2, to=1-3]
        \arrow["\iso", from=1-3, to=2-3]
    \end{tikzcd}\]
\end{lemma}

\begin{proof}
    Since $\tilde a$ is invertible, the claim holds if and only if the diagram below commutes (where, as anticipated in \cref{notation}, we drop the associators for readability).
    \[\begin{tikzcd}[column sep=large]
        {g^af} & {(g^af)^ai_X} & {g^a f^* i_X} \\
        {g^a f^* i_X} && {(g^a f^* i_X)^a i_X}
        \arrow["{\tilde a_{g^af}}", from=1-1, to=1-2]
        \arrow["{g^a\eta_f}"', from=1-1, to=2-1]
        \arrow["{\hat a_{g,f}i_X}", from=1-2, to=1-3]
        \arrow["{\tilde{a}_{g^a f^* i_X}}", from=1-3, to=2-3]
        \arrow["{\tilde{a}_{g^a f^* i_X}}"', from=2-1, to=2-3]
    \end{tikzcd}\]
    First, observe that the following diagram commutes.
    \[\begin{tikzcd}[column sep=6em]
        {g^af} && {(g^af)^ai_X} & {g^a f^* i_X} \\
        {(g^a f)^a i_X} && {((g^af)^ai_X)^a i_X} & {(g^a f^* i_X)^a i_X} \\
        {(g^a f)^a 1_{TX} i_X} && {(g^af)^a {i_X}^* i_X} \\
        {g^a f^* 1_{TX} i_X} && {g^a f^* i_X^* i_X} \\
        {g^a f^* i_X} &&& {g^a (f^* i_X)^* i_X}
        \arrow[""{name=0, anchor=center, inner sep=0}, "{\tilde a_{g^af}}", from=1-1, to=1-3]
        \arrow["{\tilde{a}_{g^a f}}"', from=1-1, to=2-1]
        \arrow[""{name=1, anchor=center, inner sep=0}, "{\hat a_{g,f}i_X}", from=1-3, to=1-4]
        \arrow["{\tilde{a}_{(g^a f)^a i_X}}"{description}, from=1-3, to=2-3]
        \arrow["{\tilde{a}_{g^a f^* i_X}}", from=1-4, to=2-4]
        \arrow[""{name=2, anchor=center, inner sep=0}, "{(\tilde{a}_{g^a f})^a i_X}"{description}, from=2-1, to=2-3]
        \arrow["\iso"', from=2-1, to=3-1]
        \arrow[""{name=3, anchor=center, inner sep=0}, "{(\hat{a}_{g,f} i_X)^a i_X}"{description}, from=2-3, to=2-4]
        \arrow["{\hat{a}_{g^a f, i_X} i_X}"{description}, from=2-3, to=3-3]
        \arrow[""{name=4, anchor=center, inner sep=0}, "{(g^a f)^a \theta^{-1}_X i_X}"{description}, from=3-1, to=3-3]
        \arrow["{\hat{a}_{g, f} 1_{TX} i_X}"', from=3-1, to=4-1]
        \arrow["{\hat{a}_{g, f} {i_X}^* i_X}"{description}, from=3-3, to=4-3]
        \arrow[""{name=5, anchor=center, inner sep=0}, "{g^a f^* \theta_X\inv i_X}"{description}, from=4-1, to=4-3]
        \arrow["\iso"', from=4-1, to=5-1]
        \arrow[""{name=6, anchor=center, inner sep=0}, "{g^a \mu_{f, i_X}^{-1} i_X}"{description}, from=4-3, to=5-4]
        \arrow[""{name=7, anchor=center, inner sep=0}, "{g^a(\eta_f)^* i_X}"', from=5-1, to=5-4]
        \arrow["{\hat{a}_{g, f^* i_X}\inv i_X}"', from=5-4, to=2-4]
        \arrow["{\natof\tilde{a}}"{description}, draw=none, from=0, to=2]
        \arrow["{\natof\tilde{a}}"{description}, draw=none, from=1, to=3]
        \arrow["{\algtheta A}"{description}, draw=none, from=2, to=4]
        \arrow["{\algmu{A}}"{description}, shift left=3, draw=none, from=3, to=6]
        \arrow["{\mndtheta T}"{description}, draw=none, from=5, to=7]
    \end{tikzcd}\]
    But, by commutativity of the following diagram, the anticlockwise path above is exactly the required composite.
    \[\begin{tikzcd}[column sep=huge]
        {g^af} && {g^a f^* i_X} \\
        {(g^a f)^a i_X} && {(g^a f^* i_X)^a i_X} \\
        {(g^a f)^a 1_{TX} i_X} \\
        {g^a f^* 1_{TX} i_X} & {g^a f^* i_X} & {g^a (f^* i_X)^* i_X}
        \arrow[""{name=0, anchor=center, inner sep=0}, "{g^a \eta_f}", from=1-1, to=1-3]
        \arrow["{\tilde{a}_{g^a f}}"', from=1-1, to=2-1]
        \arrow["{\tilde{a}_{g^a f^* i_X}}", from=1-3, to=2-3]
        \arrow[""{name=1, anchor=center, inner sep=0}, "{g^a (\eta_f)^a i_X}"{description}, from=2-1, to=2-3]
        \arrow["\iso"', from=2-1, to=3-1]
        \arrow[""{name=2, anchor=center, inner sep=0}, "{\hat{a}_{g, f} i_X}"{description}, curve={height=-12pt}, from=2-1, to=4-2]
        \arrow["{\hat{a}_{g, f} 1_{TX} i_X}"', from=3-1, to=4-1]
        \arrow["\iso"', from=4-1, to=4-2]
        \arrow["{g^a(\eta_f)^* i_X}"', from=4-2, to=4-3]
        \arrow[""{name=3, anchor=center, inner sep=0}, "{\hat{a}_{g, f^* i_X}\inv i_X}"', from=4-3, to=2-3]
        \arrow["{\natof\tilde a}"{description}, draw=none, from=1, to=0]
        \arrow["{\natof\hat a}"{description}, draw=none, from=2, to=3]
    \end{tikzcd}\qedshift\]
\end{proof}

\begin{proposition}
    \label{no-iteration-algebras}
    Under the identification of pseudomonads relative to the identity and no-iteration pseudomonads established in \cite[Proposition~3.3]{fiore2018relative}:
    \begin{enumerate}
        \item a pseudoalgebra for an identity-relative pseudomonad in the sense of \cref{pseudoalgebra} is the same as a no-iteration pseudoalgebra;
        \item a pseudomorphism between pseudoalgebras in the sense of \cref{pseudomorphism} is the same as a pseudomorphism between no-iteration pseudoalgebra;
        \item a transformation in the sense of \cref{transformation} is the same as a transformation between pseudomorphisms of no-iteration pseudoalgebras.
    \end{enumerate}
\end{proposition}

\begin{proof}
    In each instance, the two notions involve exactly the same data, except for the direction of the invertible 2-cells.
    \begin{enumerate}
        \item Using the numbering of axioms for a no-iteration pseudoalgebra in \cite[\S4]{marmolejo2013no}, the equivalences between the axioms for a relative pseudomonad pseudoalgebra and those for a no-iteration pseudoalgebra are given as follows.
        \begin{center}
        \begin{tblr}{ccc}
            \emph{Pseudoalgebras for relative pseudomonads} & & \emph{No-iteration pseudoalgebras} \\
            Naturality of $\hat a$ & $\iff$ & Axioms 4 and 5 \\
            Naturality of $\tilde a$ & $\iff$ & Axiom 1 \\
            Extension axiom & $\iff$ & Axiom 6 \\
            Unit axiom & $\iff$ & Axiom 2 \\
            \Cref{psalg-eta} & $\iff$ & Axiom 3
        \end{tblr}
        \end{center}
        Note that it follows that Axiom 3 for a no-iteration pseudoalgebras in \cite[\S4]{marmolejo2013no} is redundant, in that it can be derived from the others. This was also observed in \cite[Corollary~6.4]{lack2014monads}.
        \item Using the equation numbering for a 1-cell between no-iteration pseudoalgebras in \cite[\S4]{marmolejo2013no}, the equivalence between the axioms are given as follows.
        \begin{center}
        \begin{tblr}{ccc}
            \emph{Pseudoalgebra pseudomorphisms} & & \emph{No-iteration pseudoalgebra 1-cells} \\
            Naturality of $\overline h$ & $\iff$ & Equation (17) \\
            Compatibility with $\hat a$ & $\iff$ & Equation (16) \\
            Compatibility with $\tilde a$ & $\iff$ & Equation (15)
        \end{tblr}
        \end{center}
        \item Using the equation numbering for a 2-cell in \cite[\S4]{marmolejo2013no}, the equivalence between the axioms are given as follows.
        \begin{center}
        \begin{tblr}{ccc}
            \emph{Pseudoalgebra transformations} & & \emph{No-iteration pseudoalgebra 2-cells} \\
            Compatibility condition & $\iff$ & Equation (18)
        \end{tblr}
        \end{center}
    \end{enumerate}
\end{proof}

\begin{corollary}
    Under the identification of pseudomonads relative to the identity and no-iteration pseudomonads established in \cite[Proposition~3.3]{fiore2018relative}, the 2-category of pseudoalgebras and pseudomorphisms for an identity-relative pseudomonad on a 2-category (\cref{AlgT}) is isomorphic to the 2-category of no-iteration pseudoalgebras in \cite[\S4]{marmolejo2013no}.
\end{corollary}

\begin{proof}
    Follows directly from \cref{no-iteration-algebras}, observing that identities and composition is defined in \cite[\S4]{marmolejo2013no} in the same way as in \cref{AlgT}.
\end{proof}

Consequently, when $J$ is the identity, it follows from \cite[Theorem~5.1]{marmolejo2013no} that the bicategory of pseudoalgebras is biequivalent to the usual bicategory of pseudoalgebras for a pseudomonad in the classical sense~\cite[Definition~2.4]{bunge1974coherent}.\footnote{\cite{marmolejo2013no} consider no-iteration pseudomonads only on 2-categories, but the proof extends without fuss to bicategories.}

\begin{remark}
    \label{correspondence-between-algebra-presentations}
    Let us sketch the correspondence between pseudoalgebras for identity-relative pseudomonads, and pseudoalgebras for pseudomonads; the idea is the same as in the relationship between identity-relative pseudomonads and pseudomonads~\cite[Remarks~3.4 \& 3.5]{fiore2018relative}. Given an identity-relative pseudomonad, each pseudoalgebra $A$ is equipped with a 1-cell $(1_{A})^a \colon TA \to A$. Conversely, given a pseudomonad, each pseudoalgebra $A$ is equipped with an extension operator taking $f \colon X \to A$ to $(a \c Tf) \colon TX \to TA \to A$. The two families of invertible 2-cells are obtained in the evident way.

    Given a lax morphism of pseudoalgebras for an identity-relative pseudomonad, $\overline h_{1_A} \colon h^b \tto h (1_{A})^a$ exhibits a 2-cell as follows.
    \[\begin{tikzcd}
        TA && TB \\
        \\
        A && B
        \arrow["{(h i_B)^*}", from=1-1, to=1-3]
        \arrow["{(1_A)^a}"', from=1-1, to=3-1]
        \arrow[""{name=0, anchor=center, inner sep=0}, "{h^b}"{description}, from=1-1, to=3-3]
        \arrow["{(1_A)^b}", from=1-3, to=3-3]
        \arrow["h"', from=3-1, to=3-3]
        \arrow["{\overline h_{1_A}}"', shorten <=8pt, shorten >=8pt, Rightarrow, from=0, to=3-1]
        \arrow["\iso"{description}, draw=none, from=1-3, to=0]
    \end{tikzcd}\]
    This defines a lax morphism of pseudoalgebras for a pseudomonad in the usual sense.

    Conversely, given a lax morphism $(h, \upsilon)$ in the usual sense, for each 1-cell $f \colon X \to A$ the following 2-cell exhibits a natural family of 2-cells defining a lax morphism in the sense of \cref{pseudomorphism}.
    \[\begin{tikzcd}
        TX \\
        TA & TB \\
        A & B
        \arrow["Tf"', from=1-1, to=2-1]
        \arrow[""{name=0, anchor=center, inner sep=0}, "{T(hf)}", curve={height=-12pt}, from=1-1, to=2-2]
        \arrow["Th"{description}, from=2-1, to=2-2]
        \arrow["a"', from=2-1, to=3-1]
        \arrow["\upsilon"', shorten <=6pt, shorten >=6pt, Rightarrow, from=2-2, to=3-1]
        \arrow["b", from=2-2, to=3-2]
        \arrow["h"', from=3-1, to=3-2]
        \arrow["\iso"{description}, draw=none, from=0, to=2-1]
    \end{tikzcd}\]
    We note that, while \cite{marmolejo2013no} only consider \emph{pseudo} morphisms, the above assignments extend to a correspondence between lax morphisms for no-iteration pseudoalgebras and lax morphisms for pseudoalgebras. The correspondence for pseudo, strict, and colax morphisms is obtained analogously.
\end{remark}

\subsection{The free--forgetful relative pseudoadjunction}
\label{sec:free-forgetful-pseudoadjunction}

We now study the construction of free pseudoalgebras, and show that this construction exhibits a left relative pseudoadjoint to the forgetful functor $U_T \colon \AlgT \to \E$ of \cref{forgetful-pseudofunctor}. In \cref{resolutions-and-coherence}, we will show that this relative pseudoadjunction is universal in an appropriate sense.

\begin{proposition}
    \label{free-pseudoalgebra}
    Let $\J$ be a pseudofunctor, let $T$ be a $J$-relative pseudomonad, and let $X \in \A$. The object $TX$ admits the structure of a $T$-pseudoalgebra, the \emph{free $T$-pseudoalgebra on $X$}.
\end{proposition}

\begin{proof}
    The extension operator is given by the family $\ph^*_{{-}, X} \colon \E[J{-}, TX] \to \E[T{-}, TX]$; the natural families are given by $\mu$ and $\eta$ respectively. The pseudoalgebra axioms are then exactly the axioms of a relative pseudomonad.
\end{proof}

It is a recurring pattern in the theory of relative pseudomonads that properties of a relative pseudomonad will follow from corresponding properties of their free pseudoalgebras, since the data of a relative pseudomonad is essentially an axiomatisation of the notion of free pseudoalgebra. Moreover, the proofs establishing that properties hold for arbitrary pseudoalgebras often have the same structure as proofs of properties for relative pseudomonads qua free pseudoalgebras. In these cases, we will generally defer proofs to \cite{fiore2018relative}, as it is unenlightening to reproduce the details here.

Given a 1-cell $f \colon JX \to A$ into a $T$-pseudoalgebra $\algA$, the extension $f^a \colon TX \to A$ is a 1-cell between $T$-pseudoalgebras. As expected, it is automatically structure-preserving.

\begin{corollary}
    \label{f-a has ps-morphism structure}
    Let $T$ be a $J$-relative pseudomonad and let $\algA$ be a $T$-pseudoalgebra. For each object $X \in \A$ and 1-cell $h \colon JX \to A$, the 1-cell $h^a \colon TX \to A$ admits the structure of a pseudomorphism from the free $T$-pseudoalgebra on $X$. Furthermore, given any 2-cell $\alpha \colon h \tto h' \colon JX \to A$, the 2-cell $\alpha^a \colon h^a \tto (h')^a$ is a transformation.
\end{corollary}

\begin{proof}
    The pseudomorphism structure on $h^a$ is defined by
    \[\overline{h^a}_g \defeq \hat a_{h,g} \colon (h^ag)^a \to h^ag^*\]
    and the two coherence conditions are precisely \eqref{psalg-mu} and \cref{psalg-eta}. The algebra transformation axiom for $\alpha^a$ follows immediately from the naturality of $\hat{b}$.
\end{proof}

We are now in a position to justify our claim that $TX$ is the free $T$-pseudoalgebra on $X$.

\begin{theorem}
    \label{Alg-resolution}
    Let $T$ be a $J$-relative pseudomonad. The forgetful pseudofunctor ${U_T \colon \AlgT \to \E}$ admits a left $J$-relative pseudoadjoint $F_T \colon \A \to \AlgT$. Furthermore, this relative pseudoadjunction induces $T$.
    \[\begin{tikzcd}
        & \AlgT \\
        \A && \E
        \arrow[""{name=0, anchor=center, inner sep=0}, "{U_T}", from=1-2, to=2-3]
        \arrow[""{name=1, anchor=center, inner sep=0}, "{F_T}", from=2-1, to=1-2]
        \arrow["J"', from=2-1, to=2-3]
        \arrow["\dashv"{anchor=center}, shift right=2, draw=none, from=1, to=0]
    \end{tikzcd}\]
\end{theorem}

\begin{proof}
    For each object $X \in \A$, we set $F_T X$ to be the free $T$-pseudoalgebra defined in \cref{free-pseudoalgebra}. Note that, for each $X \in \A$, we have $U_T F_T X = TX$, and so the unit of the relative pseudomonad defines a family of 1-cells $i_X \colon JX \to U_T F_T X$.

    For each $X \in \A$ and $\algB \in \AlgT$, \cref{f-a has ps-morphism structure} defines a functor $\sharp_{X, \algB} \colon \E[JX, B] \to \AlgT[F_T X, \algB]$, which we must exhibit as being a left adjoint equivalence to $U_T\ph \c i_X$.
    The action of the two composites is given on objects as follows, for $f \colon JX \to B$ and $(h, \overline h) \colon F_T X \to \algB$.
    \begin{align*}
        U_T(\sharp_{X, \algB} f) \c i_X & = f^b \c i_X &
        \sharp_{X, B}{\big( U_T( (h, \overline h) ) \c i_X \big)} & = \big(  (h \c i_X)^b, \hat{b}_{h i_X, (-)} \big)
    \end{align*}
    We define the unit and counit respectively as follows.
    \begin{align*}
        \nu_f & \defeq f \xto{\tilde{b}_f} f^b i_X \\
        \varepsilon_{(h, \overline h)} & \defeq (h i_X)^b \xto{\overline{h}_{i_X}} h (i_X)^* \xto{h \theta_X} h1_{TX} \xto{\iso} h
    \end{align*}
    That $\varepsilon$ indeed defines an algebra transformation follows from commutativity of the following diagram on the left, where ($*$) is \cite[Lemma~3.2(ii)]{fiore2018relative}. Naturality of $\varepsilon$ in $(h, \overline h)$ follows from commutativity of the following diagram on the right, where $\alpha \colon (h, \overline h) \tto (k, \overline k)$ is a transformation.
    \[
    \begin{tikzcd}[row sep=2.1em]
    	{\big( (h i_X)^b f\big)^b} & {\big( h {(i_X)}^* f \big)^b} & {(h1f)^b} & {(hf)^b} \\
        & {h ((i_X)^* f)^*} \\
        {(hi_X)^b f^*} & {h (i_X)^* f^*} & {h(1f)^*} & {h f^*}
        \arrow[""{name=0, anchor=center, inner sep=0}, "{(\overline{h}_{i_X} f)^b}", from=1-1, to=1-2]
        \arrow["{\hat{b}_{h i_X, f}}"', from=1-1, to=3-1]
        \arrow[""{name=1, anchor=center, inner sep=0}, "{(h \theta_X f)^b}", from=1-2, to=1-3]
        \arrow["{\overline{h}_{(i_X)^* f}}"{description}, from=1-2, to=2-2]
        \arrow[""{name=2, anchor=center, inner sep=0}, "\iso", from=1-3, to=1-4]
        \arrow["{\overline{h}_{1f}}"{description}, from=1-3, to=3-3]
        \arrow["{\overline{h}_f}", from=1-4, to=3-4]
        \arrow["{h \mu_{i_X, f}}"{description}, from=2-2, to=3-2]
        \arrow[""{name=3, anchor=center, inner sep=0}, "{h(\theta_X f)^*}"{description}, curve={height=-6pt}, from=2-2, to=3-3]
        \arrow[""{name=4, anchor=center, inner sep=0}, "{\overline{h}_{i_X} f^*}"', from=3-1, to=3-2]
        \arrow["{h \theta_X f}"', from=3-2, to=3-3]
        \arrow[""{name=5, anchor=center, inner sep=0}, "\iso"', from=3-3, to=3-4]
        \arrow["{\morassoc{h}}"{description}, shift right=2, draw=none, from=0, to=4]
        \arrow["{\natof{\overline h}}"{description}, draw=none, from=1, to=3]
        \arrow["{\natof{\overline h}}"{description}, draw=none, from=2, to=5]
        \arrow["{\text{$*$}}"{description, pos=0.7}, draw=none, from=3, to=3-2]
        %
        %
        \arrow[
                "{ (\varepsilon_{(h, \overline h)} f)^b}",
                rounded corners,
                to path=
                { -- ([yshift=.7cm]\tikztostart.center)
                -- ([yshift=.7cm]\tikztotarget.center)
          \tikztonodes
                -- (\tikztotarget.north)
                },
                from=1-1, to=1-4
            ]
        \arrow[
                "{\varepsilon_{(h, \overline h)} f^*}"',
                rounded corners,
                to path=
                { -- ([yshift=-.7cm]\tikztostart.center)
                -- ([yshift=-.7cm]\tikztotarget.center)
                \tikztonodes
                -- (\tikztotarget.south)
                },
                from=3-1, to=3-4
            ]
    \end{tikzcd}
    \hspace{1mm}
    %
    \begin{tikzcd}[column sep = 4em]
    	{(hi_X)^a} & {(k i_X)^a} \\
        {h (i_X)^*} & {k (i_X)^*} \\
        {h1_{TX}} & {k1_{TX}} \\
        h & k
        \arrow["{(\alpha i_X)^a}", from=1-1, to=1-2]
        \arrow[""{name=0, anchor=center, inner sep=0}, "{\overline{h}_{i_X}}"{description}, from=1-1, to=2-1]
        \arrow[""{name=1, anchor=center, inner sep=0}, "{\overline{k}_{i_X}}"{description}, from=1-2, to=2-2]
        \arrow["{\alpha (i_X)^*}"{description}, from=2-1, to=2-2]
        \arrow["{h \theta_X}"{description}, from=2-1, to=3-1]
        \arrow["{k \theta_X}"{description}, from=2-2, to=3-2]
        \arrow["{\alpha1_{TX}}"{description}, from=3-1, to=3-2]
        \arrow["\iso"{description}, from=3-1, to=4-1]
        \arrow["\iso"{description}, from=3-2, to=4-2]
        \arrow["\alpha"', from=4-1, to=4-2]
        \arrow["{\trans{\alpha}}"{description}, draw=none, from=0, to=1]
    %
    %
    \arrow[
            "{ \varepsilon_{(k, \overline k)} }",
            rounded corners,
            to path=
            { -- ([xshift=.85cm]\tikztostart.center)
            -- ([xshift=.85cm]\tikztotarget.center)
   \tikztonodes
            -- (\tikztotarget.east)
            },
            from=1-2, to=4-2
        ]
    \arrow[
        "{\varepsilon_{(h, \overline h)}}"',
        rounded corners,
        to path=
        { -- ([xshift=-.85cm]\tikztostart.center)
        -- ([xshift=-.85cm]\tikztotarget.center)
        \tikztonodes
        -- (\tikztotarget.west)
        },
        from=1-1, to=4-1
    ]
    \end{tikzcd}
    \]
    It remains to check the two triangle laws, which correspond to the next two diagrams. On the right-hand side, the cell marked ($*$) is~\cite[Lemma~3.2(iii)]{fiore2018relative}.
    \[
    \begin{tikzcd}
    	{f^b} && {f^b} \\
        {(f^b i_X)^b} & {f^b (i_X)^*} & {f^b1_{TX}}
        \arrow[Rightarrow, no head, from=1-1, to=1-3]
        \arrow[""{name=0, anchor=center, inner sep=0}, "{(\tilde{b}_f)^b}"{description}, from=1-1, to=2-1]
        \arrow["{\hat{b}_{f, i_X}}"', from=2-1, to=2-2]
        \arrow["{f^b\theta_X}"', from=2-2, to=2-3]
        \arrow[""{name=1, anchor=center, inner sep=0}, "\iso"', from=2-3, to=1-3]
        \arrow["{\algtheta{A}}"{description}, draw=none, from=0, to=1]
        \arrow[
            "{\sharp(\nu_f)}"',
            rounded corners,
            to path=
            { -- ([xshift=-1.1cm]\tikztostart.center)
            -- ([xshift=-1.1cm]\tikztotarget.center)
            \tikztonodes
            -- (\tikztotarget.west)
            },
            from=1-1, to=2-1
        ]
    \end{tikzcd}
    \hspace{2mm}
    %
    \begin{tikzcd}
    	{hi_X} && {hi_X} \\
        {(hi_X)^bi_X} & {h (i_X)^* i_X} & {h1_{TX} i_X}
        \arrow[""{name=0, anchor=center, inner sep=0}, Rightarrow, no head, from=1-1, to=1-3]
        \arrow["{\tilde{b}_{h i_X}}"{description}, from=1-1, to=2-1]
        \arrow[""{name=1, anchor=center, inner sep=0}, "{h \eta_{i_X}}"{description}, curve={height=-6pt}, from=1-1, to=2-2]
        \arrow["{\overline{h}_{i_X} i_X}"', from=2-1, to=2-2]
        \arrow["{h \theta_X i_X}"', from=2-2, to=2-3]
        \arrow["\iso"', from=2-3, to=1-3]
        \arrow["{\morunit{h}}"{description, pos=0.6}, draw=none, from=1, to=2-1]
        \arrow["{\text{($*$)}}"{description}, draw=none, from=0, to=2-3]
        %
        \arrow[
            "{\nu_{U_T(h, \overline h)i_X}}"{description},
            rounded corners,
            to path=
            { -- ([xshift=-1.1cm]\tikztostart.center)
            -- ([xshift=-1.1cm]\tikztotarget.center)
            \tikztonodes
            -- (\tikztotarget.west)
            },
            from=1-1, to=2-1
        ]
        \end{tikzcd}\]
    So we do indeed have a relative pseudoadjunction. Following \cite[Theorem 3.8]{fiore2018relative}, this induces a relative pseudomonad $\big(T', \ph^{*'}, i', \eta', \mu', \theta' \big)$ with action on objects $T'X = U_TF_T X = TX$ and extension operator given by the composite
    \[
        \ph^{*'}_X \defeq \E[JX, U_TF_T Y] \xto{\sharp_{X, F_T Y}} \AlgT[F_T X, F_T Y] \xto{(U_T)_{F_T X, F_T X}} \E[TX, TY]
    \]
    which is exactly the extension operator $\ph^*_X$ of $T$. The unit $i' = i$ by definition.
    It remains to show the 2-cell data of $T'$ coincides with that for $T$. First, for each 1-cell $f \colon JX \to TY$, the 2-cell $\eta'_f$ is defined to be the unit $\nu_f$;  using the free algebra structure given by \cref{free-pseudoalgebra}, this is exactly $\eta_f$. Next, using the fact that $U_T$ is a strict pseudofunctor, $\theta'_X$ is defined to be the composite below.\footnote{Note that there is a mistake in the stated definition of $\theta$ in the proof of \cite[Theorem~3.8]{fiore2018relative}, which should be defined analogously to \eqref{Kl-!-unitor}. However, the proof of the coherence laws \ibid{} invokes the correct definition.}
    \[
        (U_T\sharp_{X, F_T Y})(i_X) \xto{\iso} U_T\sharp_{X, F_T Y}{\big( U_T(1_{F_TX}) i_X  \big)} \xto{U_T\varepsilon_{1_{TX}}} U_T(1_{F_T X}) = 1_{U_TF_T X}
    \]
    where, because $U_T$ is strict, the first isomorphism is just the structural isomorphism in $\E$. Using the definition of the identity in $\PsAlg(T)$, we get that $\theta'_X$ is the composite along the top below.
    \[\begin{tikzcd}[column sep=huge]
        {(i_X)^*} & {(1_{TX} i_X)^*} & {1_{TX}(i_X)^*} & {1_{TX}1_{TX}} & {1_{TX}} \\
        && {(i_X)^*}
        \arrow["\cong", from=1-1, to=1-2]
        \arrow[Rightarrow, no head, from=1-1, to=2-3]
        \arrow["\cong"{description}, from=1-2, to=1-3]
        \arrow["{1_{TX}\theta_X}"{description}, from=1-3, to=1-4]
        \arrow["\cong"{description}, from=1-3, to=2-3]
        \arrow["\cong"{description}, from=1-4, to=1-5]
        \arrow["{\theta_X}"', from=2-3, to=1-5]
        \arrow[
            "{\varepsilon_{1_{TX}}}",
            rounded corners,
            to path=
            { -- ([yshift=.6cm]\tikztostart.center)
            -- ([yshift=.6cm]\tikztotarget.center)
            \tikztonodes
            -- (\tikztotarget.north)
            },
            from=1-2, to=1-5
        ]
    \end{tikzcd}\]
    So $\theta'_X = \theta_X$ as required. Finally, instantiating the definition of $\mu$ in \cite[Theorem~3.8]{fiore2018relative} for $f \colon JX \to TY$ and $g \colon JY \to TX$, and using the fact $U_T$ is a strict pseudofunctor, we see the 2-cell $\mu'$ is defined to be the composite below. We suppress the subscripts for reasons of space.
    \begin{equation}\label{def of muprime}
        U\sharp{\big(U(g) f \big)} \xto{U\sharp(U\sharp(g)\nu_f)} U\sharp{\left( U{\big( \sharp(g)\sharp(f) \big)} i \right)} = U\sharp{\left(U{\big( \sharp(g)\sharp(f) \big)} i \right)} \xto{U\varepsilon_{\sharp(g)\sharp(f)}} U{\big( \sharp(g)\sharp(f) \big)} = U\sharp(g)U\sharp(f)
    \end{equation}
    To unfold this, note that $\sharp(g)\sharp(f)$ is, by the definition of composition in $\PsAlg(T)$, the pseudomorphism with carrier $g^* f^*$ and 2-cell
    $(g^* f^* k)^* \xto{\mu_{g, f^* k}} g^* (f^* k)^* \xto{g \mu_{f, k}} g^* f^* k^*$
    for each 1-cell $k \colon JW \to TX$ in $\E$. The diagram \eqref{def of muprime} therefore unfolds to the clockwise path in the following diagram, verifying that $\mu' = \mu$.
    \[\begin{tikzcd}[column sep=large]
        {(g^* f)^*} & {(g^* f^* i_X)^*} & {g^* (f^* i_X)^*} & {g^* f^* (i_X)^*} \\
        &&& {g^* f^* 1} \\
        {(g^* f)^*} &&& {g^* f^*}
        \arrow["{(g^* \eta_f)^*}", from=1-1, to=1-2]
        \arrow[Rightarrow, no head, from=1-1, to=3-1]
        \arrow["{\mu_{g, f^* i_X}}", from=1-2, to=1-3]
        \arrow[""{name=0, anchor=center, inner sep=0}, "{(g^* \eta_f\inv)^*}"{description}, curve={height=-6pt}, from=1-2, to=3-1]
        \arrow["{g^* \mu_{f, i}}", from=1-3, to=1-4]
        \arrow[""{name=1, anchor=center, inner sep=0}, "{g^* (\eta_f\inv)^*}"{description}, curve={height=6pt}, from=1-3, to=3-4]
        \arrow["{g^* f^* \theta_X}", from=1-4, to=2-4]
        \arrow["\iso", from=2-4, to=3-4]
        \arrow["{\mu_{g,f}}"', from=3-1, to=3-4]
        \arrow["{\natof{\mu}}"{description}, draw=none, from=0, to=1]
        \arrow["{\mndtheta T}"{description}, draw=none, from=1-4, to=1]
    \end{tikzcd}\qedshift\]
\end{proof}

\subsection{A three-dimensional perspective on the bicategory of pseudoalgebras}
\label{sec:bicategory-of-pseudoalgebras-as-a-limit}

As recalled in \cref{relative-pseudomonads}, while the definition of relative pseudomonad (\cref{relative-pseudomonad}) is presented solely in terms of an assignment $X \mapsto TX$ on objects, this assignment extends to a pseudofunctor. Furthermore, the families $i_X \colon JX \to TX$ and $\ph^*_{X, Y} \colon \E[JX, TY] \to \E[TX, TY]$ extend to pseudonatural transformations. In this subsection, we show something similar holds for pseudoalgebras. First, we show in \cref{algebras give pseudonat trans and modification} that the extension operator for each pseudoalgebra is automatically pseudonatural, and that the unitors form modifications. Next, we examine the global structure formed by the pseudoalgebras for a fixed relative pseudomonad. In particular, we show in \cref{algebras-into-inserter} that each of the bicategories of pseudoalgebras defined in \cref{AlgT,AlgT-variants} admit a canonical locally faithful strict pseudofunctor into a bicategory whose objects are pairs of an object $A \in \E$ and a pseudonatural transformation $\E[J{-}, A] \tto \E[T{-}, A]$. This shows that the pseudonaturality structure of pseudoalgebras is coherent in an appropriate two-dimensional sense, and explains why the composition of lax morphisms of pseudoalgebras is reminiscent of the composition of lax natural transformations. More abstractly, it lays the groundwork for presenting the bicategories of pseudoalgebras as certain three-dimensional limits. Though we do not pursue this here, it would be a natural starting point for a formal theory of relative pseudomonads (\cf~\cite{lack2000coherent,arkor2024formal}).

The following extends \cite[Proposition~4.7(iii)]{fiore2018relative} from free pseudoalgebras to arbitrary pseudoalgebras, and confirms that the unitor $\eta$ (\cref{relative-pseudomonad}) is indeed a modification.

\begin{lemma}
    \label{algebras give pseudonat trans and modification}
    Let $\J$ be a pseudofunctor,
    $T$ be a $J$-relative pseudomonad, and $\algA$ be a $T$-pseudoalgebra (\cref{pseudoalgebra}).
    \begin{enumerate}
        \item \label{algebras as pseudonat trans} The family of functors
            $\big\{ \ph^a_X \colon \E[JX, A] \to \E[TX, A] \big\}_{X \in \A}$
            admits the structure of a pseudonatural transformation $\ph^a \colon \E[J{-}, A] \tto \E[T{-}, A] \colon \A\op \to \Cat$.
        \item \label{a tilde is a modification} The family of 2-cells $\tilde{a}$ defines an invertible modification filling the triangle below in $\Hom(\A\op, \Cat)$.
        \[\begin{tikzcd}
        	{\E[J{-}, A]} & {\E[T{-}, A]} \\
        	& {\E[J{-}, A]}
        	\arrow["{\ph^a}", from=1-1, to=1-2]
        	\arrow[""{name=0, anchor=center, inner sep=0}, curve={height=6pt}, Rightarrow, no head, from=1-1, to=2-2]
        	\arrow["{\ph \c i}", from=1-2, to=2-2]
        	\arrow["{\tilde{a}}", shorten <=4pt, Rightarrow, from=0, to=1-2]
        \end{tikzcd}\]
    \end{enumerate}
\end{lemma}

\begin{proof}
    \begin{enumerate}
        \item The 1-cell components of the pseudonatural transformation are the functors $(-)^a$. The 2-cell components must be natural isomorphisms $\ph^a_f$ filling the square below for each $f \colon Y \to X$ in $\A$.
        \[\begin{tikzcd}[column sep=large]
            {\E[JX, A]} & {\E[JY, A]} \\
            {\E[TX, A]} & {\E[TY, A]}
            \arrow["{\E[Jf, A]}", from=1-1, to=1-2]
            \arrow["{\ph^a_X}"', from=1-1, to=2-1]
            \arrow["{\ph^a_Y}", from=1-2, to=2-2]
            \arrow["{\E[Tf, h]}"', from=2-1, to=2-2]
        \end{tikzcd}\]
        For each $k \colon JX \to A$, we define $k^a_f$ to be the following composite.
        \begin{equation} \label{pseudonat trans structure on algebras}
        \begin{tikzcd}[column sep=large]
            {(k \c Jf)^a} & {\big( (k^a i_X) Jf \big)^a} & {\big( k^a (i_X Jf) \big)^a} & {k^a (i_X Jf)^* = k^a \circ Tf}
            \arrow["{(\tilde{a}_k Jf)^a}", from=1-1, to=1-2]
            \arrow["\cong", from=1-2, to=1-3]
            \arrow["{\hat{a}_{k, i_X Jf}}", from=1-3, to=1-4]
        \end{tikzcd}
        \end{equation}
        As in the proof of \cite[Proposition~4.7(iii)]{fiore2018relative}, the coherence conditions for a pseudonatural transformation are straightforward to check.
        \item The modification law amounts to verifying that the following two pasting diagrams are equal for each $f \colon Y \to X$ in $\A$.
        \[
        \begin{tikzcd}
            {\E[JX, A]} && {\E[JY, A]} & {\E[JY, A]} \\
            &&& {\E[TY, A]} \\
            {\E[JX, A]} && {E[JY, A]} & {\E[JY, A]}
            \arrow["{\E[Jf, A]}", from=1-1, to=1-3]
            \arrow[""{name=0, anchor=center, inner sep=0}, Rightarrow, no head, from=1-1, to=3-1]
            \arrow[Rightarrow, no head, from=1-3, to=1-4]
            \arrow[""{name=1, anchor=center, inner sep=0}, Rightarrow, no head, from=1-3, to=3-3]
            \arrow["{(-)^a_Y}", from=1-4, to=2-4]
            \arrow["{\ph \c i_Y}", from=2-4, to=3-4]
            \arrow["{\E[Jf, A]}"', from=3-1, to=3-3]
            \arrow[Rightarrow, no head, from=3-3, to=3-4]
            \arrow["{=}"{description}, draw=none, from=0, to=1]
            \arrow["{\tilde{a}}"'{pos=0.4}, shorten <=9pt, shorten >=13pt, Rightarrow, from=2-4, to=1]
        \end{tikzcd}
        \begin{tikzcd}
            {\E[JX, A]} & {\E[JX, A]} && {\E[JY, A]} \\
            & {\E[TX, A]} && {\E[TY, A]} \\
            {\E[JX, A]} & {\E[JX, A]} && {\E[JY, A]}
            \arrow[Rightarrow, no head, from=1-1, to=1-2]
            \arrow[""{name=0, anchor=center, inner sep=0}, Rightarrow, no head, from=1-1, to=3-1]
            \arrow["{\E[Jf, A]}", from=1-2, to=1-4]
            \arrow[""{name=1, anchor=center, inner sep=0}, "{(-)^a_X}"{description}, from=1-2, to=2-2]
            \arrow[""{name=2, anchor=center, inner sep=0}, "{(-)^a_Y}", from=1-4, to=2-4]
            \arrow["{\E[Tf, A]}"{description}, from=2-2, to=2-4]
            \arrow[""{name=3, anchor=center, inner sep=0}, "{\ph \c i_X}"{description}, from=2-2, to=3-2]
            \arrow[""{name=4, anchor=center, inner sep=0}, "{\ph \c i_Y}", from=2-4, to=3-4]
            \arrow[Rightarrow, no head, from=3-1, to=3-2]
            \arrow["{\E[Jf, A]}"', from=3-2, to=3-4]
            \arrow["{\ph^a_f}"', shorten <=33pt, shorten >=33pt, Rightarrow, from=2, to=1]
            \arrow["{\tilde{a}}"'{pos=0.4}, shorten <=6pt, shorten >=13pt, Rightarrow, from=2-2, to=0]
            \arrow["{\ph \c i_f}"', shorten <=33pt, shorten >=33pt, Rightarrow, from=4, to=3]
        \end{tikzcd}
        \]
        Evaluating at $k \colon JX \to A$, this equation becomes the commutativity of the following diagram.
        \[\begin{tikzcd}
            {k \c Jf} && {k \c Jf} \\
            {(k \c Jf)^a \c i_Y} \\
            {(k^a \c i_X \c Jf)^a \c i_Y} \\
            {k^a \c (i_X \c Jf)^* \c i_Y} && {k^a \c i_X \c Jf}
            \arrow[Rightarrow, no head, from=1-1, to=1-3]
            \arrow["{\tilde{a}_{k \circ Jf}}"', from=1-1, to=2-1]
            \arrow[""{name=0, anchor=center, inner sep=0}, "{\tilde{a}_k \c Jf}", from=1-3, to=4-3]
            \arrow["{\tilde{a}_{k \c Jf}^{-1}}"{description}, curve={height=12pt}, from=2-1, to=1-3]
            \arrow[""{name=1, anchor=center, inner sep=0}, "{(\tilde{a}_k \c Jf)^a \c i_Y}"', from=2-1, to=3-1]
            \arrow["{\hat{a}_{k, i_X \c Jf} \c i_Y}"', from=3-1, to=4-1]
            \arrow[""{name=2, anchor=center, inner sep=0}, "{\tilde{a}_{k^a \c i_X \c Jf}^{-1}}"{description}, curve={height=-12pt}, from=3-1, to=4-3]
            \arrow["{k^a \c \eta_{i_X \c Jf}^{-1}}"', from=4-1, to=4-3]
            \arrow["{\natof\tilde a}"{description}, draw=none, from=1, to=0]
            \arrow["{\algeta{A}}"{description, pos=0.7}, shift right=2, draw=none, from=2, to=4-1]
        \end{tikzcd}\qedshift\]
    \end{enumerate}
\end{proof}

Next, we study the interaction between the pseudonatural transformations induced by each pseudoalgebra. It will be useful to introduce a preliminary definition.

\begin{definition}
    Let $F, G \colon \A \to \B$ be pseudofunctors. Their \emph{lax-inserter} $\Ins_l(F, G)$ is the bicategory defined as follows.
    \begin{itemize}
        \item Objects are pairs of an object $X \in \A$ and a 1-cell $x \colon FX \to GX$ in $\B$.
        \item 1-cells from $(X, x)$ to $(Y, y)$ are pairs of a morphism $f \colon X \to Y$ in $\A$ and a 2-cell $\phi$ in $\B$ as follows.
        \[\begin{tikzcd}
            FX & FY \\
            GX & GY
            \arrow["Ff", from=1-1, to=1-2]
            \arrow["x"', from=1-1, to=2-1]
            \arrow["\phi"', shorten <=4pt, shorten >=4pt, Rightarrow, from=1-2, to=2-1]
            \arrow["y", from=1-2, to=2-2]
            \arrow["Gf"', from=2-1, to=2-2]
        \end{tikzcd}\]
        \item 2-cells from $(f, \phi)$ to $(g, \gamma)$ are 2-cells $\alpha \colon f \tto g$ in $\A$ satisfying the following compatibility condition.
        \begin{equation}\label{lax-inserter-2-cell-law}
        \begin{tikzcd}[column sep=large]
            FX & FY \\
            GX & GY
            \arrow[""{name=0, anchor=center, inner sep=0}, "Ff", curve={height=-12pt}, from=1-1, to=1-2]
            \arrow["x"', from=1-1, to=2-1]
            \arrow["y", from=1-2, to=2-2]
            \arrow[""{name=1, anchor=center, inner sep=0}, "Gf", curve={height=-12pt}, from=2-1, to=2-2]
            \arrow[""{name=2, anchor=center, inner sep=0}, "Gg"', curve={height=12pt}, from=2-1, to=2-2]
            \arrow["{\phi\:}"'{pos=0.3}, shorten <=4pt, shorten >=11pt, Rightarrow, from=0, to=1]
            \arrow["{G\alpha}"', shorten <=3pt, shorten >=3pt, Rightarrow, from=1, to=2]
        \end{tikzcd}
        \quad
        =
        \quad
        \begin{tikzcd}[column sep=large]
            FX & FY \\
            GX & GY
            \arrow[""{name=0, anchor=center, inner sep=0}, "Ff", curve={height=-12pt}, from=1-1, to=1-2]
            \arrow[""{name=1, anchor=center, inner sep=0}, "Gf"', curve={height=12pt}, from=1-1, to=1-2]
            \arrow["x"', from=1-1, to=2-1]
            \arrow["y", from=1-2, to=2-2]
            \arrow[""{name=2, anchor=center, inner sep=0}, "Gg"', curve={height=12pt}, from=2-1, to=2-2]
            \arrow["{F\alpha}", shorten <=3pt, shorten >=3pt, Rightarrow, from=0, to=1]
            \arrow["{\:\gamma}"{pos=0.7}, shorten <=11pt, shorten >=4pt, Rightarrow, from=1, to=2]
        \end{tikzcd}
		\end{equation}
        \item The composite $(X, x) \xto{(f, \phi)} (Y, y) \xto{(g, \gamma)} (Z, z)$ is defined to be $gf$ equipped with the following 2-cell.
        \begin{equation}
            z \c F(gf) \xto\iso (z \c Fg) \c Ff \xto{\gamma \c Ff} G(g \c y) \c Ff \xto\iso Gg \c (y \c Ff) \xto{Gg \c \phi} Gg \c (Gf \c x) \xto\iso G(gf) \c x
        \end{equation}
        The identity on $(X, x)$ is defined to be $1_X$ equipped with the 2-cell $x \c F(1_X) \xto\iso G(1_X) \c x$.
        \item The associators and unitors are inherited from $\A$.
    \end{itemize}
    By requiring each $\phi$ in the definition of the morphisms to be invertible, we obtain the \emph{pseudo-inserter} $\Ins_p(F, G)$. By requiring it to be the identity, we obtain the \emph{strict-inserter} $\Ins_s(F, G)$. By reversing the direction of $\phi$, we obtain the \emph{colax-inserter} $\Ins_c(F, G)$.

    For each $w \in \{ s, p, l, c \}$ we  write $\pi$ for each of the evident strict pseudofunctor $\Ins_w(F, G) \to \A$.
\end{definition}

\begin{remark}
    The lax-inserter of two pseudofunctors is a three-dimensional notion of limit living in the tricategory of bicategories, which generalises a two-dimensional notion of limit known as an \emph{inserter}~\cite[\S6.5]{lack2010companion}. Note in particular that the pseudo-inserter of two pseudofunctors is distinct from the similarly-named \emph{iso-inserter} (as studied in the context of pseudofunctors in \cite[Definition~2.14, Lemma~2.16 \& Remark~2.16]{gurski2013coherence}, for instance).
\end{remark}

The \emph{nerve} of a pseudofunctor $F \colon \A \to \B$ is the pseudofunctor $N_F \colon \B \to \Hom(\A\op, \Cat)$ defined by $N_F \defeq \B[F{-}_2, {-}_1]$. In particular, for a $J$-relative pseudomonad $T$ we have pseudofunctors $N_J, N_T \colon \E \to \Hom(\A\op, \Cat)$. The following result shows that the bicategory structure of $\AlgTl$ (\ie{} the composition, identities, and structural transformations) coincides with that of $\Ins_l(N_J, N_T)$. This demonstrates, for instance, that the composition of lax morphisms may be seen as precisely the composition of lax natural transformations.

\begin{proposition}
    \label{algebras-into-inserter}
    Let $\J$ be a pseudofunctor and let $T$ be a $J$-relative pseudomonad. For each $w \in \{ s, p, l, c \}$, there is a canonical locally faithful strict pseudofunctor $K \colon \PsAlg_w(T) \to \Ins_w(N_J, N_T)$ rendering the following diagram commutative.
    \[\begin{tikzcd}
        {\PsAlg_w(T)} & {\Ins_w(N_J, N_T)} \\
        & \E
        \arrow["K", from=1-1, to=1-2]
        \arrow["{U_T}"', from=1-1, to=2-2]
        \arrow["\pi", from=1-2, to=2-2]
    \end{tikzcd}\]
\end{proposition}

\begin{proof}
	We show the result for $w = l$; the other cases are straightforward adjustments.
    For the action on objects, we use \cref{algebras as pseudonat trans} and define $K\big(\algA\big) \defeq (A, \ph^a)$.
    For the action on 1-cells, consider a lax morphism
    $(h, \overline h) \colon \algA \to \algB$.
    We shall show that the 2-cells $\overline h_f$ form a modification filling the square on the left below in $\Hom(\A\op, \Cat)$, so that we may define $K\big((h, \overline h)\big) \defeq (h, \overline h)$.
    \[\begin{tikzcd}
    	{\E[J{-}, A]} & {\E[J{-}, B]} \\
    	{\E[T{-}, A]} & {\E[T{-}, B]}
    	\arrow["{\E[J{-}, h]}", from=1-1, to=1-2]
    	\arrow[""{name=0, anchor=center, inner sep=0}, "{\ph^a}"', from=1-1, to=2-1]
    	\arrow[""{name=1, anchor=center, inner sep=0}, "{\ph^b}", from=1-2, to=2-2]
    	\arrow["{\E[T{-}, h]}"', from=2-1, to=2-2]
    	\arrow["{\overline{h}}"', shorten <=20pt, shorten >=20pt, Rightarrow, from=1, to=0]
    \end{tikzcd}\]
    The modification axiom states that, for each $f \colon Y \to X$, the following 2-cells are equal.
    \[\begin{tikzcd}[column sep=4.2em,row sep=2.5em]
        {\E[JX, A]} & {\E[JY, A]} \\
        {\E[TX, A]} & {\E[JX, B]} & {\E[JY, B]} \\
        {\E[TX, B]} & {\E[TY, B]}
        \arrow["{{\E[Jf, A]}}", from=1-1, to=1-2]
        \arrow["{{(-)^a}}"', from=1-1, to=2-1]
        \arrow[""{name=0, anchor=center, inner sep=0}, "{{\E[JX, h]}}"{description}, from=1-1, to=2-2]
        \arrow[""{name=1, anchor=center, inner sep=0}, "{{\E[JY, h]}}", from=1-2, to=2-3]
        \arrow["{{\E[TX, h]}}"', from=2-1, to=3-1]
        \arrow["{{\overline h}}"{description}, shorten <=2pt, shorten >=2pt, Rightarrow, from=2-2, to=2-1]
        \arrow["{{\E[Jf, B]}}"{description}, from=2-2, to=2-3]
        \arrow[""{name=2, anchor=center, inner sep=0}, "{{(-)^b}}"{description}, from=2-2, to=3-1]
        \arrow[""{name=3, anchor=center, inner sep=0}, "{{(-)^b}}", from=2-3, to=3-2]
        \arrow["{{\E[Tf, B]}}"', from=3-1, to=3-2]
        \arrow["\cong"{description}, shorten <=21pt, shorten >=21pt, Rightarrow, from=1, to=0]
        \arrow["{{(-)^b_f}}"{description}, shorten <=21pt, shorten >=21pt, Rightarrow, from=3, to=2]
    \end{tikzcd}\]
    \[\begin{tikzcd}[column sep=4.2em,row sep=2.5em]
        {\E[JX, A]} & {\E[JY, A]} \\
        {\E[TX, A]} & {\E[TY, A]} & {\E[JY, B]} \\
        {\E[TX, B]} & {\E[TY, B]}
        \arrow["{{\E[Jf, A]}}", from=1-1, to=1-2]
        \arrow[""{name=0, anchor=center, inner sep=0}, "{{(-)^a}}"', from=1-1, to=2-1]
        \arrow[""{name=1, anchor=center, inner sep=0}, "{(-)^a}"{description}, from=1-2, to=2-2]
        \arrow["{{\E[JY, h]}}", from=1-2, to=2-3]
        \arrow["{\E[Tf, A]}"{description}, from=2-1, to=2-2]
        \arrow[""{name=2, anchor=center, inner sep=0}, "{{\E[TX, h]}}"', from=2-1, to=3-1]
        \arrow[""{name=3, anchor=center, inner sep=0}, "{\E[TY, h]}"{description}, from=2-2, to=3-2]
        \arrow["{\overline{h}}"{description}, Rightarrow, from=2-3, to=2-2]
        \arrow["{{(-)^b}}", from=2-3, to=3-2]
        \arrow["{{\E[Tf, B]}}"', from=3-1, to=3-2]
        \arrow["{\ph^a_f}"{description}, shorten <=21pt, shorten >=21pt, Rightarrow, from=1, to=0]
        \arrow["\iso"{description}, shorten <=21pt, shorten >=21pt, Rightarrow, from=3, to=2]
    \end{tikzcd}\]
    Using the definition of the pseudonatural transformation $\ph^a$ in \eqref{pseudonat trans structure on algebras}, for each 1-cell $k \colon JX \to A$, this equation becomes the commutativity of the following diagram.
    \[\begin{tikzcd}[column sep=large]
        {(h  k  Jf)^b} && {\big( (hk)^b  i_X  Jf \big)^b} \\
        {h  (k  Jf)^a} & {(h  k^a  i_X  Jf)^b} & {(hk)^b  (i_X  Jf)^*} \\
        {h  \big( k^a  i_X  Jf \big)^a} && {h  k^a  (i_X  Jf)^*}
        \arrow[""{name=0, anchor=center, inner sep=0}, "{{{(\tilde{b}_{hk}  Jf)^b}}}", from=1-1, to=1-3]
        \arrow["{{{\overline{h}_{k  Jf}}}}"', from=1-1, to=2-1]
        \arrow["{{{(h   \tilde{a}_k  Jf)^b}}}"{description}, from=1-1, to=2-2]
        \arrow[""{name=1, anchor=center, inner sep=0}, "{{(\overline{h}_k i_X Jf)^b}}"{description}, from=1-3, to=2-2]
        \arrow["{{\hat{b}_{hk,i_X Jf}}}", from=1-3, to=2-3]
        \arrow["{{\natof{\overline{h}}}}"{description}, draw=none, from=2-1, to=2-2]
        \arrow["{{{h  (\tilde{a}_k  Jf)^a}}}"', from=2-1, to=3-1]
        \arrow["{{\overline{h}_{k^a i_X Jf}}}"{description}, from=2-2, to=3-1]
        \arrow["{{\overline{h}_k  Tf}}", from=2-3, to=3-3]
        \arrow[""{name=2, anchor=center, inner sep=0}, "{{h \hat{a}_{k,i_X Jf}}}"', from=3-1, to=3-3]
        \arrow["{{\morunit{h}}}"{description}, draw=none, from=0, to=2-2]
        \arrow["{\morassoc{h}}"{description, pos=0.6}, shift left=3, draw=none, from=1, to=2]
    \end{tikzcd}\]
    Turning to the transformations, if $\alpha \colon (h, \overline h) \tto (h', \overline{h'})$ is an algebra transformation, then $\alpha$ is a 2-cell in $\Ins_w(N_J, N_T)$: indeed, the transformation axiom in \cref{transformation} is exactly the condition \eqref{lax-inserter-2-cell-law}. So we may define $K(\alpha) \defeq \alpha$. This is clearly functorial with respect to vertical composition, and locally faithful. Strict functoriality of $K$ follows by unwinding the definitions of composition and identities in both bicategories.
\end{proof}

\section{Doctrinal adjunction and transport of structure}
\label{sec:doctrinal-adjunction}

In this section, we explore how pseudoalgebras interact with adjunctions. It was observed by \textcite{kelly1974doctrinal} that, given an adjunction between monoidal categories, colax monoidal structures on the left adjoint are in bijection with lax monoidal structures on the right adjoint. This situation in fact holds, much more generally, for adjunctions between the pseudoalgebras for a 2-monad~\cite{kelly1974doctrinal}. This phenomenon was dubbed \emph{doctrinal adjunction} (a \emph{doctrine} being another term for a 2-monad), and proves to be a remarkably useful tool. We show in this section that doctrinal adjunction holds also for pseudoalgebras for relative pseudomonads; this generalises both \citeauthor{kelly1974doctrinal}'s original setting, and doctrinal adjunction for pseudomonads as studied by \textcite[Theorem~1.4.14]{nunes2017pseudomonads} and \textcite[Corollary~78]{walker2019distributive}. Doctrinal adjunction will be useful in \cref{lax-idempotence} when we study lax-idempotence for pseudoalgebras, and subsequently to characterise the pseudoalgebras for the presheaf construction in \cref{presheaves-and-cocompleteness}. Below, \cref{doctrinal-adjunction} is the analogue of \cite[Theorem~1.2]{kelly1974doctrinal}, while \cref{adjunctions-lift-to-algebras} is the analogue of \cite[Theorems~1.4 \& 1.5]{kelly1974doctrinal}.

\begin{theorem}
    \label{doctrinal-adjunction}
    Let $\J$ be a pseudofunctor, let $T$ be a $J$-relative pseudomonad, and let $\algA$ and $\algB$ be $T$-pseudoalgebras. For any adjunction
    $\begin{tikzcd}
        A & B
        \arrow[""{name=0, anchor=center, inner sep=0}, "\ell", shift left=2, from=1-1, to=1-2]
        \arrow[""{name=1, anchor=center, inner sep=0}, "r", shift left=2, from=1-2, to=1-1]
        \arrow["\dashv"{anchor=center, rotate=-90}, draw=none, from=0, to=1]
    \end{tikzcd}$
    between their carriers in $\E$, the following are in bijection:
    \begin{enumerate}
        \item colax morphism structures on $\ell$;
        \item lax morphism structures on $r$.
    \end{enumerate}
\end{theorem}

\begin{proof}
    We write $\nu \colon 1_A \tto r \ell$ and $\varepsilon \colon \ell r \tto 1_B$ for the unit and counit of the adjunction respectively. Suppose that $r$ has a lax morphism structure. We define a putative colax morphism structure on $\ell$ by setting $\overline{\ell}_f$ to be the mate of $\overline{r}_f$ for each $f \colon JX \to A$ in $\E$.
    \begin{equation} \label{induced-morphism-structure-on-ell}
    \overline{\ell}_f \defeq \Big( \ell f^a \xto{\ell(\nu f)^a} \ell(r\ell f)^a \xto{\ell\overline r_{\ell f}} \ell r(\ell f)^b \xto{\varepsilon(\ell f)^b} (\ell f)^b \Big) \end{equation}
    For the associativity law, we need to fill the following diagram.
    \[\begin{tikzcd}
        {\ell \big( g^a f)^a} & {(\ell g^a f)^b} \\
        {\ell g^a f^*} \\
        {(\ell g)^a f^*} & {\big( (\ell g)^b f \big)^b}
        \arrow["{\overline{\ell}_{g^a f}}", from=1-1, to=1-2]
        \arrow["{\ell \hat{a}_{g, f}}"', from=1-1, to=2-1]
        \arrow["{(\overline{\ell}_g f)^b}", from=1-2, to=3-2]
        \arrow["{\overline{\ell}_g f^*}"', from=2-1, to=3-1]
        \arrow["{\hat{b}_{\ell g, f}}", from=3-2, to=3-1]
    \end{tikzcd}\]
    We do this as follows; the unlabelled cells use the functoriality of vertical composition and ($*$) is the right triangle law for the adjunction. Note that the boundary of the following diagram is by definition the boundary of the preceding diagram.
    \[\begin{tikzcd}[column sep=3.3em]
        {\ell \big( g^a f)^a} && {\ell ( r \ell g^a f)^a} & {\ell r (\ell g^a f)^b} & {(\ell g^a f)^b} \\
        {\ell g^a f^*} & {\ell \big( (r \ell g)^a f \big)^a} & {\ell \big( r \ell (r \ell g)^a f \big)^a} & {\ell r \big( \ell (r \ell g)^a f \big)^b} & {\big( \ell (r \ell g)^a f \big)^b} \\
        {\ell (r \ell g)^a f^*} & {\ell \big( r (\ell g)^b f \big)^a} & {\ell \big( r \ell r (\ell g)^b f \big)^a} & {\ell r \big( \ell r (\ell g)^b f \big)^b} & {\big( \ell r (\ell g)^b f \big)^b} \\
        && {\ell \big( r (\ell g)^b f \big)^a} \\
        {\ell r (\ell g)^a f^*} &&& {\ell r \big( (\ell g)^b f \big)^b} \\
        {(\ell g)^a f^*} &&&& {\big( (\ell g)^b f \big)^b}
        \arrow["{\ell (\nu g^a f)^a}", from=1-1, to=1-3]
        \arrow["{\ell \hat{a}_{g, f}}"{description}, from=1-1, to=2-1]
        \arrow["{\ell ((\nu g)^a f)^a}"{description}, from=1-1, to=2-2]
        \arrow["{\ell \overline{r}_{\ell g^a f}}", from=1-3, to=1-4]
        \arrow[""{name=0, anchor=center, inner sep=0}, "{\ell (r \ell (\nu g)^a f)^a}"{description}, from=1-3, to=2-3]
        \arrow["{\varepsilon (\ell g^a f)^b}", from=1-4, to=1-5]
        \arrow[""{name=1, anchor=center, inner sep=0}, "{\ell r( \ell (\nu g)^a f )^b}"{description}, from=1-4, to=2-4]
        \arrow["{( \ell (\nu g)^a f )^b}"{description}, from=1-5, to=2-5]
        \arrow["{\natof\hat{a}}"{description}, draw=none, from=2-1, to=2-2]
        \arrow["{\ell (\nu g)^a f^*}"{description}, from=2-1, to=3-1]
        \arrow["{\ell (\nu (r \ell g)^a f )^a}"', from=2-2, to=2-3]
        \arrow["{\ell \hat{a}_{r \ell g, f}}"{description}, from=2-2, to=3-1]
        \arrow["{\ell (\overline{r}_{\ell g} f )^b}"{description}, from=2-2, to=3-2]
        \arrow["{\ell \overline{r}_{\ell (r \ell g)^a f}}"', from=2-3, to=2-4]
        \arrow[""{name=2, anchor=center, inner sep=0}, "{\ell (r \ell \overline{r}_{\ell g} f )^b}"{description}, from=2-3, to=3-3]
        \arrow["{\varepsilon ( \ell (r \ell g)^a f )^b}"', from=2-4, to=2-5]
        \arrow[""{name=3, anchor=center, inner sep=0}, "{\ell r ( \ell \overline{r}_{\ell g} f )^b}"{description}, from=2-4, to=3-4]
        \arrow["{( \ell \overline{r}_{\ell g} f )^b}"{description}, from=2-5, to=3-5]
        \arrow[""{name=4, anchor=center, inner sep=0}, "{\ell r_{\ell g} f^*}"{description}, from=3-1, to=5-1]
        \arrow["{\ell ( \nu r (\ell g)^b f )^a}"', from=3-2, to=3-3]
        \arrow[""{name=5, anchor=center, inner sep=0}, curve={height=12pt}, Rightarrow, no head, from=3-2, to=4-3]
        \arrow["{\ell \overline{r}_{\ell r (\ell g)^b f}}"', from=3-3, to=3-4]
        \arrow[""{name=6, anchor=center, inner sep=0}, "{\ell ( r \varepsilon (\ell g)^b f )^a}"{description}, from=3-3, to=4-3]
        \arrow["{\varepsilon ( \ell r (\ell g)^b f )^b}"', from=3-4, to=3-5]
        \arrow[""{name=7, anchor=center, inner sep=0}, "{\ell r ( \varepsilon (\ell g)^b f )^b}"{description}, from=3-4, to=5-4]
        \arrow["{( \varepsilon (\ell g)^b f )^b}"{description}, from=3-5, to=6-5]
        \arrow["{\ell \overline{r}_{(\ell g)^b f}}"{description}, from=4-3, to=5-4]
        \arrow["{\varepsilon (\ell g)^a f^*}"{description}, from=5-1, to=6-1]
        \arrow["{\ell r \hat{b}_{\ell g, f}}"{description}, from=5-4, to=5-1]
        \arrow["{\varepsilon ( (\ell g)^b f )^b}"{description}, from=5-4, to=6-5]
        \arrow["{\hat{b}_{\ell g, f}}", from=6-5, to=6-1]
        \arrow["{\natof\overline{r}}"{description}, draw=none, from=0, to=1]
        \arrow["{\natof\overline{r}}"{description}, draw=none, from=2, to=3]
        \arrow["{\morassoc r}"{description}, draw=none, from=4, to=4-3]
        \arrow["{\text{($*$)}}"{description}, draw=none, from=5, to=3-3]
        \arrow["{\natof\overline{r}}"{description}, draw=none, from=6, to=7]
    \end{tikzcd}\]
    The proof of the unit law is similar, as shown below, where ($*$) is the left triangle law.
    \[\begin{tikzcd}[column sep=large]
    	{\ell f} &&& {\ell f} \\
    	& {\ell r \ell f} \\
    	{\ell f^a i_X} & {\ell ( r \ell f)^a i_X} & {\ell r (\ell f)^b i_X} & {(\ell f)^b i_X}
    	\arrow[""{name=0, anchor=center, inner sep=0}, Rightarrow, no head, from=1-1, to=1-4]
    	\arrow["{\ell \nu f}"{description}, from=1-1, to=2-2]
    	\arrow[""{name=1, anchor=center, inner sep=0}, "{\ell \tilde{a}_f}"', from=1-1, to=3-1]
    	\arrow["{\tilde{b}_{\ell f}}", from=1-4, to=3-4]
    	\arrow["{\varepsilon  \ell f}"{description}, from=2-2, to=1-4]
    	\arrow[""{name=2, anchor=center, inner sep=0}, "{\ell  \tilde{a}_{r \ell f}}"{description}, from=2-2, to=3-2]
    	\arrow[""{name=3, anchor=center, inner sep=0}, "{\ell r \tilde{b}_{\ell f}}"{description}, curve={height=-12pt}, from=2-2, to=3-3]
    	\arrow["{\ell (\nu f)^a i_X}"', from=3-1, to=3-2]
    	\arrow["{\ell \overline{r}_{\ell f} i_X}"', from=3-2, to=3-3]
    	\arrow["{\varepsilon (\ell f)^b i_X}"', from=3-3, to=3-4]
    	\arrow["{\natof\tilde{a}}"{description}, draw=none, from=1, to=2]
    	\arrow["{\text{($*$)}}"{description}, draw=none, from=2-2, to=0]
    	\arrow["{\morunit r}"{description}, draw=none, from=3, to=3-2]
        %
        %
        \arrow[
            "{\overline{\ell}_{f} i}",
            swap,
            rounded corners,
            to path=
            { -- ([yshift=-.8cm]\tikztostart.center)
            -- ([yshift=-.8cm]\tikztotarget.center)
            \tikztonodes
            -- (\tikztotarget.south) },
            from=3-1, to=3-4
        ]
    \end{tikzcd}\]
    Thus $(\ell, \overline\ell)$ is a colax morphism.
    A dual argument shows that every colax morphism structure on $\ell$ induces a lax morphism structure on $r$, by defining $\overline r_g \colon (rg)^a \to rg^b$ to be the mate of $\overline{l}_f$ for each 1-cell $g \colon JX \to B$ in $\E$.
    \begin{equation}
        \label{mate-of-r}
        \overline r_g \defeq \Big( (rg)^a \xto{\nu(rg)^a} r\ell(rg)^a \xto{r\overline\ell_{rg}} r(\ell rg)^b \xto{r(\varepsilon g)^b} rg^b \Big)
    \end{equation}
    Finally, these assignments form a bijection by the usual calculus of mates.
\end{proof}

\begin{proposition}
    \label{adjunctions-lift-to-algebras}
    Under the same assumptions as \cref{doctrinal-adjunction}:
    \begin{enumerate}
    \item \label{adjunction-lifts-to-AlgTl}
        the adjunction $\ell \adj r$ lifts to $\AlgTl$ if and only if $\overline\ell$ is invertible, \ie{} $(\ell, \overline \ell)$ is a pseudomorphism from $\algA$ to $\algB$;
    \item \label{adjunction-lifts-to-AlgTc}
        the adjunction $\ell \adj r$ lifts to $\AlgTc$ if and only if $\overline r$ is invertible, \ie{} $(r, \overline r)$ is a pseudomorphism.
    \end{enumerate}
    Moreover, in each case the inverse is constructed by taking mates.
\end{proposition}

\begin{proof}
We prove (1); the argument for (2) is dual. Suppose $(\ell, \overline \ell)$ is a pseudomorphism; we need to show that the unit $\nu \colon 1_A \tto r \ell$ and counit $\varepsilon \colon \ell r \tto 1_B$ each satisfy the transformation axiom \eqref{transformation}, so that $(\ell, \overline{\ell}^{-1}) \colon \algA \rightleftarrows \algB \cocolon (r, \overline r)$ is an adjunction in $\AlgTl$. This is shown by the two diagrams below, in which ($*$) indicates the left triangle law in each case.
\begin{center}
\begin{tikzcd}[column sep = 3.7em, row sep = 2.2em]
	{f^a} && {(r \ell f)^a} \\
	& {r \ell f^a} & {r \ell (r \ell f)^a} \\
	& {r (\ell f)^b} & {r (\ell r \ell f)^b} \\
	&& {r (\ell f)^b} \\
	{f^a} && {r \ell f^a}
	\arrow["{(\nu f)^a}", from=1-1, to=1-3]
	\arrow["{\nu f^a}"{description}, from=1-1, to=2-2]
	\arrow[Rightarrow, no head, from=1-1, to=5-1]
	\arrow["{\nu (r \ell f)^a}", from=1-3, to=2-3]
	\arrow["{r \ell (\nu f)^a}"{description}, from=2-2, to=2-3]
	\arrow["{r (\overline{\ell}_f)^{-1}}"', from=2-2, to=3-2]
	\arrow["{\natof\overline{\ell}}"{description}, draw=none, from=2-2, to=3-3]
	\arrow["{r (\overline{\ell}_{r \ell f})\inv}", from=2-3, to=3-3]
	\arrow["{r  (\ell \nu f)^b}"{description}, from=3-2, to=3-3]
	\arrow[""{name=0, anchor=center, inner sep=0}, curve={height=12pt}, Rightarrow, no head, from=3-2, to=4-3]
	\arrow["{r (\varepsilon \ell f)^b}", from=3-3, to=4-3]
	\arrow["{r \ell_f}", from=4-3, to=5-3]
	\arrow["{\nu f^a}"', from=5-1, to=5-3]
	\arrow["{(*)}"{description}, draw=none, from=3-3, to=0]
    \arrow[
        "{\overline{r}_{\ell f}}",
        rounded corners,
        to path=
        { -- ([xshift=1.4cm]\tikztostart.center)
        -- ([xshift=1.4cm]\tikztotarget.center)
        \tikztonodes
        -- (\tikztotarget.east)
        },
        from=1-3, to=4-3
    ]
\end{tikzcd}
\hspace{1mm}
\begin{tikzcd}[column sep = 3.7em, row sep = 2.2em]
	{(\ell r g)^b} && {g^b} \\
	{\ell (r g)^a} \\
	{\ell r \ell (r g)^a} & {\ell (r g)^a} \\
	{\ell r (\ell r g)^b} & {(\ell r g)^b} \\
	{\ell r g^b} && {g^b}
	\arrow["{(\varepsilon g)^b}", from=1-1, to=1-3]
	\arrow["{\overline{\ell}_{r g}}"', from=1-1, to=2-1]
	\arrow[Rightarrow, no head, from=1-3, to=5-3]
	\arrow["{\ell \nu (r g)^a}"', from=2-1, to=3-1]
	\arrow[""{name=0, anchor=center, inner sep=0}, curve={height=-12pt}, Rightarrow, no head, from=2-1, to=3-2]
	\arrow["{\varepsilon \ell (rg)^a}"{description}, from=3-1, to=3-2]
	\arrow["{\ell r (\overline{\ell}_{r g})\inv}"', from=3-1, to=4-1]
	\arrow["{(\overline{\ell}_{rg})\inv}"{description}, from=3-2, to=4-2]
	\arrow["{\varepsilon (\ell r g)^b}"{description}, from=4-1, to=4-2]
	\arrow["{\ell r (\varepsilon g)^b}"', from=4-1, to=5-1]
	\arrow["{(\varepsilon g)^b}"{description}, from=4-2, to=5-3]
	\arrow["{\varepsilon g^b}"', from=5-1, to=5-3]
	\arrow["{\text{($*$)}}"{description}, draw=none, from=0, to=3-1]
    %
    %
    \arrow[
            "{\ell \overline{r}_g}"',
            rounded corners,
            to path=
            { -- ([xshift=-1.4cm]\tikztostart.center)
            -- ([xshift=-1.4cm]\tikztotarget.center)
            \tikztonodes
            -- (\tikztotarget.west)
            },
            from=2-1, to=5-1
        ]
\end{tikzcd}
\end{center}

For the converse, suppose that $\ell$ acquires a lax algebra morphism structure $(\ell, \overline \ell)$ such that $(\nu, \varepsilon) \colon (\ell, \overline \ell) \adj (r, \overline r)$ is an adjunction in $\AlgTl$. Then the 2-cells $\nu$ and $\varepsilon$ must lift to $\AlgTl$, so the cells labelled ($*$), and hence both diagrams below, must commute.
\begin{center}
\begin{tikzcd}
	{\ell f^a} & {\ell (r \ell f)^a} & {\ell r (\ell f)^b} & {(\ell f)^b} \\
	\\
	{\ell f^a} && {\ell r \ell f^a} & {\ell f^a}
	\arrow["{\ell (\nu f)^a}", from=1-1, to=1-2]
	\arrow[""{name=0, anchor=center, inner sep=0}, Rightarrow, no head, from=1-1, to=3-1]
	\arrow["{\ell \overline{r}_{\ell f}}", from=1-2, to=1-3]
	\arrow["{\varepsilon (\ell f)^b}", from=1-3, to=1-4]
	\arrow[""{name=1, anchor=center, inner sep=0}, "{\ell r \overline{\ell}_f}"{description}, from=1-3, to=3-3]
	\arrow["{\overline{\ell}_f}", from=1-4, to=3-4]
	\arrow["{\ell \nu f^a}"', from=3-1, to=3-3]
	\arrow["{\varepsilon \ell f^a}"', from=3-3, to=3-4]
	\arrow["{\text{($*$)}}"{description}, draw=none, from=0, to=1]
\end{tikzcd}
\quad
\begin{tikzcd}
	{(\ell f)^b} & {(\ell r \ell g )^a} && { (\ell g)^b} \\
	\\
	{\ell f^a} & {\ell (r \ell g)^a} & {\ell r g^b} & {(\ell g)^b}
	\arrow["{(\ell \nu f)^b}", from=1-1, to=1-2]
	\arrow[""{name=0, anchor=center, inner sep=0}, "{\overline{\ell}_f}"', from=1-1, to=3-1]
	\arrow["{(\varepsilon \ell g)^b}", from=1-2, to=1-4]
	\arrow[""{name=1, anchor=center, inner sep=0}, "{\overline{\ell}_{r \ell g}}"{description}, from=1-2, to=3-2]
	\arrow[""{name=2, anchor=center, inner sep=0}, Rightarrow, no head, from=1-4, to=3-4]
	\arrow["{\ell (\nu f)^a}"', from=3-1, to=3-2]
	\arrow["{\ell \overline{r}_g}"', from=3-2, to=3-3]
	\arrow["{\varepsilon g^b}"', from=3-3, to=3-4]
	\arrow["{\natof\overline{l}}"{description}, draw=none, from=0, to=1]
	\arrow["{\text{($*$)}}"{description}, draw=none, from=1, to=2]
\end{tikzcd}
\end{center}
Now, the bottom composite on the left and the top composite on the right are, by the left triangle law, both equal to the identity. So $\overline{\ell}_f$ is invertible, with inverse the mate of $\overline{r}_f$ defined in \eqref{mate-of-r}.
\end{proof}

\begin{remark}
    In the preceding theorem, the induced colax morphism structure on $\overline l_f$ is invertible if and only if the composite \eqref{induced-morphism-structure-on-ell} is invertible. In the non-relative setting, Nunes calls this the \emph{Beck--Chevalley condition} on the lax morphism $r$~\cite[Definition 1.4.13]{nunes2017pseudomonads}.
\end{remark}

A variant of doctrinal adjunction concerns when an object of $\E$ is reflective (or coreflective) in a pseudoalgebra, giving sufficient conditions for pseudoalgebra structure to be inherited via the reflection. In particular, these conditions hold when the reflection is an equivalence. In this way, we obtain a \emph{transport of structure} theorem for pseudoalgebras, which shows that pseudoalgebra structure may be transferred across an equivalence in $\E$.

\begin{corollary}[Transport of structure]
    \label{transport-of-structure}
    Let $(\nu, \varepsilon) \colon \ell \colon A \leftrightarrows B \colon r$ be a coreflective adjunction in $\E$ (so that $\nu$ is invertible) and let $\algB$ be a $T$-pseudoalgebra. If the 2-cell
    \begin{equation}
        \label{transport-2-cell}
        r (\ell r f)^b \xto{r(\varepsilon f)^b} r f^b
    \end{equation}
    is invertible for every $f \colon JX \to B$, then:
    \begin{enumerate}
        \item $A$ acquires a $T$-pseudoalgebra structure;
        \item $r$ acquires the structure of an algebra pseudomorphism and $\ell$ acquires the structure of a colax algebra morphism;
        \item the unit $\nu$ and counit $\varepsilon$ are algebra transformations.
    \end{enumerate}
    Hence, the adjunction $(\nu, \varepsilon) \colon \ell \dashv r$ lifts to an adjunction in $\AlgTc$. In particular, if $(\nu, \varepsilon) \colon \ell \dashv r$ is an adjoint equivalence then it lifts to $\AlgT$. Conversely, if
    	$(\nu, \varepsilon) \colon \ell \dashv r$
    lifts to an adjunction in $\AlgTc$ then the 2-cell \eqref{transport-2-cell}
    is invertible for every $f \colon JX \to B$ in $\E$.
\end{corollary}

 \begin{proof}
 	We start with the forward direction.
     For the algebra structure on $A$, we define an extension operator as follows:
     \begin{equation}
        \label{reflected-algebra-structure}
        (-)^a_X \defeq \E[JX, A]\xto{\ell \c -} \E[JX, B]\xto{(-)^b_X} \E[TX, B] \xto{r \c -} \E[TX, A]
     \end{equation}
     We define $\tilde a$ by $\tilde{a}_f \defeq \big(  f\xto{\nu f} r \ell f \xto{f \tilde{b}_{\ell f}} r (\ell f)^b i_X = f^a i_X \big)$ for each $f \colon JX \to A$; this is invertible since $\nu$ is. For $\hat a$, we note that $(g^ a f)^a = r \big( \ell r (\ell g)^b f  \big)^b$, so that we may define $\hat a$ as follows:
     \[
         \hat{a}_{g, f} \defeq \Big( (g^ a f)^a = r \big( \ell r (\ell g)^b f  \big)^b \xto{r(\varepsilon(\ell g)^b f)^b} r \big( (\ell g)^b f \big)^b \xto{ r\hat{b}_{\ell g, f}} r (\ell g)^b f^* = g^a f^* \Big)
     \]
    The invertibility condition \eqref{transport-2-cell} ensures this is invertible. Next, we need to verify the associativity and unit laws in \cref{pseudoalgebra}. For the associativity law, observe first that the following equations hold:
     \[
         \big( (h^a g)^a f \big)^a =  r\big[ \ell r \big( \ell r (\ell h)^b g \big)^b f \big]^b
         \qquad\qquad
         (h^a g^* f)^a = r \big( \ell r (\ell h)^b q^* f \big)^b
     \]
    The associativity law then unwinds to the next diagram, in which the unlabelled shapes are instances of the functoriality of composition.
    \[\begin{tikzcd}[column sep=5.8em, row sep = 2.3em]
    	{r \big[ \ell r  \big(  \ell r (\ell h)^b g \big)^b f \big]^b} & {r \big[ \big( \ell r (\ell h)^b g \big)^b f \big]^b} & {r \big( \ell r (\ell h)^b g \big)^b f^*} \\
    	{r \big[ \ell r  \big(  (\ell h)^b g \big)^b f \big]^b} \\
    	{r  \big[  \ell r (\ell h)^b g^* f \big]^b} & {r \big[ \big( (\ell h)^b g \big)^b f \big]^b} \\
    	{r  \big[ (\ell h)^b g^* f \big]^b} && {r  \big( (\ell h)^b g \big)^b f^*} \\
    	{r (\ell h)^b \big( g^* f \big)^*} && {r (\ell h)^b g^* f^* }
    	\arrow["{r [ \varepsilon \big(  \ell r (\ell h)^b g )^b f ]^b}", from=1-1, to=1-2]
    	\arrow["{r [  \ell r ( \varepsilon (\ell h)^b g )^b f ]^b}"{description}, from=1-1, to=2-1]
    	\arrow["{r \hat{b}_{\ell r (\ell h)^b g, f}}", from=1-2, to=1-3]
    	\arrow[""{name=0, anchor=center, inner sep=0}, "{r \big[ \big( \varepsilon (\ell h)^b g \big)^b f \big]^b}"{description}, from=1-2, to=3-2]
    	\arrow[""{name=1, anchor=center, inner sep=0}, "{r \big( \varepsilon (\ell h)^b g \big)^b f^*}"{description}, from=1-3, to=4-3]
    	\arrow["{r [  \ell r \hat{b}_{\ell h, g} f ]^b}"{description}, from=2-1, to=3-1]
    	\arrow["{r [ \varepsilon ((\ell h)^b g )^b f ]^b}"{description}, from=2-1, to=3-2]
    	\arrow["{r [ \varepsilon (\ell h)^b g^* f \big]^b}"{description}, from=3-1, to=4-1]
    	\arrow["{r \big[ \hat{b}_{\ell h, g}f \big]^b}"{description}, from=3-2, to=4-1]
    	\arrow["{r \hat{b}_{(\ell h)^b g, f}}"{description}, from=3-2, to=4-3]
    	\arrow["{r \hat{b}_{\ell h, g^*f}}"{description}, from=4-1, to=5-1]
    	\arrow["{\ell \hat{b}_{rh, g} f^*}"{description}, from=4-3, to=5-3]
    	\arrow[""{name=2, anchor=center, inner sep=0}, "{r (\ell h)^b \mu_{g, f}}"', from=5-1, to=5-3]
    	\arrow["{\natof\hat{b}}"{description}, draw=none, from=0, to=1]
    	\arrow["{\algmu{B}}"{description}, draw=none, from=3-2, to=2]
        %
        %
        \arrow[
            "{\hat{a}_{h^a g, f}}",
            rounded corners,
            to path=
            { -- ([yshift=.9cm]\tikztostart.center)
            -- ([yshift=.9cm]\tikztotarget.center)
            \tikztonodes
            -- (\tikztotarget.north)
            },
            from=1-1, to=1-3
        ]
        %
        \arrow[
                "{\hat{a}_{h,g} f^*}",
                rounded corners,
                to path=
                { -- ([xshift=1.6cm]\tikztostart.center)
                -- ([xshift=1.6cm]\tikztotarget.center)
                \tikztonodes
                -- (\tikztotarget.east)
                },
                from=1-3, to=5-3
            ]
        %
        \arrow[
                "{(\hat{a}_{h,g} f)^a}"',
                rounded corners,
                to path=
                { -- ([xshift=-1.7cm]\tikztostart.center)
                -- ([xshift=-1.7cm]\tikztotarget.center)
                \tikztonodes
                -- (\tikztotarget.west)
                },
                from=1-1, to=3-1
            ]
        %
        \arrow[
                "{\hat{a}_{h,g^* f}}"',
                rounded corners,
                to path=
                { -- ([xshift=-1.7cm]\tikztostart.center)
                -- ([xshift=-1.7cm]\tikztotarget.center)
                \tikztonodes
                -- (\tikztotarget.west)
                },
                from=3-1, to=5-1
            ]
    \end{tikzcd}\]
    For the unit law we make use of the left triangle law for the adjunction $\ell \dashv r$:
    \[\begin{tikzcd}[column sep=huge]
        	{r (\ell f)^b} & {r (\ell f)^b} & {r (\ell f)^b 1_{TX}} \\
        	{r (\ell r \ell f)^b } \\
        	{r \big( \ell r (\ell f)^b i_X \big)^b} & {r\big((\ell f)^b i_X \big)^b} & {r (\ell f)^b (i_X)^*}
        	\arrow[Rightarrow, no head, from=1-1, to=1-2]
        	\arrow["{r(\ell \nu f)^b}"', from=1-1, to=2-1]
        	\arrow[""{name=0, anchor=center, inner sep=0}, "\iso", from=1-2, to=1-3]
        	\arrow["{r (\tilde{b}_{\ell f})^b}"{description}, from=1-2, to=3-2]
        	\arrow["{r (\varepsilon \ell f)^b}"{description}, from=2-1, to=1-2]
        	\arrow["{r ( \ell r \tilde{b}_{\ell f})^b}"', from=2-1, to=3-1]
        	\arrow["{r(\varepsilon (\ell f)^b i_X)^b}"', from=3-1, to=3-2]
        	\arrow[""{name=1, anchor=center, inner sep=0}, "{r \tilde{b}_{\ell f, i_X}}"', from=3-2, to=3-3]
        	\arrow["{r (\ell f)^b \theta_X}"', from=3-3, to=1-3]
        	\arrow["{\algtheta B}"{description}, draw=none, from=1, to=0]
        %
        %
        \arrow[
                "{(\tilde{a}_f)^a}"',
                rounded corners,
                to path=
                { -- ([xshift=-1.7cm]\tikztostart.center)
                -- ([xshift=-1.7cm]\tikztotarget.center)
                \tikztonodes
                -- (\tikztotarget.west)
                },
                from=1-1, to=3-1
            ]
        \arrow[
            "{\hat{a}_{f, i}}"',
            rounded corners,
            to path=
            { -- ([yshift=-.9cm]\tikztostart.center)
            -- ([yshift=-.9cm]\tikztotarget.center)
            \tikztonodes
            -- (\tikztotarget.south)
            },
            from=3-1, to=3-3
        ]
    \end{tikzcd}\]
    So $\algA$ is indeed a $T$-pseudoalgebra. Next, we construct a pseudomorphism structure on $r$ and apply doctrinal adjunction (\cref{doctrinal-adjunction}). To this end, define $\overline r$ as follows, which is invertible since \eqref{transport-2-cell} is:
    \[
        \overline{r}_f \defeq \Big( (r f)^a = r (\ell r f)^b \xto{ r (\varepsilon f)^b } r f^b  \Big)
    \]
    For the associativity law for $(r, \overline r)$, observe first that
    $\big( (r g)^a f\big)^a = r \big[  \ell r (\ell r g)^b f \big]^b$. The required diagram is then as shown below.
    \[\begin{tikzcd}[column sep = 4.2em]
        {r \big[\ell r (\ell r g)^b f \big]^b} & {r \big[  (\ell r g)^b f \big]^b} & {r  (\ell r g)^a f^*} \\
        {r \big[\ell r g^b f \big]^b} & {r[  g^b f ]^b} & {r g^a f^*}
        \arrow["{r [ \ell r (\varepsilon g)^b f ]^b}"', from=1-1, to=2-1]
        \arrow["{r [  \varepsilon  g^b f ]^b}"', from=2-1, to=2-2]
        \arrow["{r [ \varepsilon (\ell r g)^b f ]^b}", from=1-1, to=1-2]
        \arrow["{r \hat{b}_{\ell r g, f}}", from=1-2, to=1-3]
        \arrow["{r \hat{b}_{g,f}}"', from=2-2, to=2-3]
        \arrow[""{name=0, anchor=center, inner sep=0}, "{r (\varepsilon g)^a f^*}", from=1-3, to=2-3]
        \arrow[""{name=1, anchor=center, inner sep=0}, "{r [  (\varepsilon g)^b f ]^b}"{description}, from=1-2, to=2-2]
        \arrow["{\natof \hat{b}}"{description}, draw=none, from=1, to=0]
    %
    %
    \arrow[
            "{\hat{a}_{rg, f}}",
            rounded corners,
            to path=
            { -- ([yshift=.8cm]\tikztostart.center)
            -- ([yshift=.8cm]\tikztotarget.center)
            \tikztonodes
            -- (\tikztotarget.north)
            },
            from=1-1, to=1-3
        ]
    \arrow[
            "{\overline{r}_g f^*}",
            rounded corners,
            to path=
            { -- ([xshift=1.5cm]\tikztostart.center)
            -- ([xshift=1.5cm]\tikztotarget.center)
            \tikztonodes
            -- (\tikztotarget.east)
            },
            from=1-3, to=2-3
        ]
    \arrow[
            "{(\overline{r}_f g)^a }"',
            rounded corners,
            to path=
            { -- ([xshift=-1.8cm]\tikztostart.center)
            -- ([xshift=-1.8cm]\tikztotarget.center)
            \tikztonodes
            -- (\tikztotarget.west)
            },
            from=1-1, to=2-1
        ]
    %
    \arrow[
            "{\overline{r}_{g^a f}}"',
            rounded corners,
            to path=
            { -- ([yshift=-.7cm]\tikztostart.center)
            -- ([yshift=-.7cm]\tikztotarget.center)
            \tikztonodes
            -- (\tikztotarget.south)
            },
            from=2-1, to=2-2
        ]
    \end{tikzcd}\]
    Finally, for the unit axiom we apply the right triangle law.
    \[
    \begin{tikzcd}[column sep = 4em]
		rf & rf \\
		{r \ell r f} \\
		{r(\ell r f)^a i_X} & {r f^a i_X}
		\arrow[Rightarrow, no head, from=1-1, to=1-2]
		\arrow["{\nu rf}"', from=1-1, to=2-1]
		\arrow["{r \tilde{a}_f}", from=1-2, to=3-2]
		\arrow[""{name=0, anchor=center, inner sep=0}, "{r \varepsilon f}"{description}, from=2-1, to=1-2]
		\arrow["{r\tilde{a}_{\ell r f}}"', from=2-1, to=3-1]
		\arrow[""{name=1, anchor=center, inner sep=0}, "{r(\varepsilon f)^a i_X}"', from=3-1, to=3-2]
		\arrow["{\natof{\tilde{a}}}"{description}, draw=none, from=0, to=1]
        %
        %
        \arrow[
                "{\tilde{a}_{rf}}"',
                rounded corners,
                to path=
                { -- ([xshift=-1.3cm]\tikztostart.center)
                -- ([xshift=-1.3cm]\tikztotarget.center)
                \tikztonodes
                -- (\tikztotarget.west)
                },
                from=1-1, to=3-1
            ]
        \arrow[
                "{\overline{r}_f i}"',
                rounded corners,
                to path=
                { -- ([yshift=-.9cm]\tikztostart.center)
                -- ([yshift=-.9cm]\tikztotarget.center)
                \tikztonodes
                -- (\tikztotarget.south)
                },
                from=3-1, to=3-2
            ]
    \end{tikzcd}\]
    Thus, $(r, \overline r)$ is a $T$-algebra pseudomorphism $\algB \to \algA$.

	The rest of the claim now follows. Indeed, by \cref{doctrinal-adjunction} we see that $\ell$ acquires a colax morphism structure as in \eqref{induced-morphism-structure-on-ell}. Next, applying \cref{adjunctions-lift-to-algebras}, the adjunction $(\nu, \varepsilon) \colon \ell \adj r$ lifts to $\AlgTc$. In particular, if we start with an adjoint equivalence then $\nu$, $\varepsilon$ and $\overline{r}_f$ are all invertible, so $\ell$ becomes a pseudomorphism and the adjunction lifts to $\AlgT$. Moreover, since the forgetful pseudofunctor $U_T$ is locally conservative (\cref{forgetful-pseudofunctor}), this too is an adjoint equivalence.

	Turning now to the converse, suppose that the adjunction
		$(\nu, \varepsilon) \colon \ell \dashv r$
	lifts to an adjunction
		$(\nu, \varepsilon) \colon (\ell, \overline \ell) \dashv (r, \overline r)$
	in $\AlgTc$. We want to show that $r(\varepsilon f)^b$ is invertible for every
	$f \colon JX \to B$. Observe first that $r(\varepsilon f)^b i_X$ is invertible. This follows from the commutativity of the following diagram, in which every vertical arrow is, by assumption, an isomorphism, and the unlabelled region is the right triangle law.
    \[\begin{tikzcd}
        {r f} & rf \\
        {r \ell r f} \\
        {r(\ell rf)^b i_X} & {r f^b i_X}
        \arrow[Rightarrow, no head, from=1-1, to=1-2]
        \arrow["{\nu rf}"', from=1-1, to=2-1]
        \arrow["{r \tilde{b}_f}", from=1-2, to=3-2]
        \arrow[""{name=0, anchor=center, inner sep=0}, "{r \varepsilon f}"{description}, curve={height=6pt}, from=2-1, to=1-2]
        \arrow["{r \ell r \tilde{b}_f}"', from=2-1, to=3-1]
        \arrow[""{name=1, anchor=center, inner sep=0}, "{r(\varepsilon f)^b i_X}"', from=3-1, to=3-2]
        \arrow["{\natof{\tilde{b}}}"{description}, draw=none, from=0, to=1]
    \end{tikzcd}\]
	From \cref{adjunctions-lift-to-algebras}, since the adjunction $(\ell, \overline \ell) \dashv (r, \overline r)$ lifts to $\AlgTc$, the 2-cell $\overline{r}$ is invertible. Since isomorphisms satisfy the 2-out-of-3 property, it therefore suffices to show that the composite
		$r (\ell r f)^b \xto{r (\varepsilon f)^b} r f^b \xto{\overline{r}_f} (rf)^a$
	is invertible.
	This follows from the commutativity of the following diagram, since every arrow in the clockwise path is an isomorphism.
	\[\begin{tikzcd}[column sep = 4em]
		{r(\ell rf)^b} & {r (\ell rf)^b (i_X)^*} & {r \big( (\ell r f)^b i_X \big)^b} & {\big(r (\ell r f)^b i_X \big)^a} \\
		\\
		&& {r ( f^b i_X )^b} & {(r f^b i_X )^a} \\
		{r (\ell rf)^b} && {r f^b} & {(rf)^a}
		\arrow["{r (\ell rf)^b \theta_X\inv}", from=1-1, to=1-2]
		\arrow[""{name=0, anchor=center, inner sep=0}, Rightarrow, no head, from=1-1, to=4-1]
		\arrow[""{name=1, anchor=center, inner sep=0}, "{r \hat{b}_{\ell r f, i_X}\inv}", from=1-2, to=1-3]
		\arrow[""{name=2, anchor=center, inner sep=0}, "{\overline{r}_{(\ell r f)^b i_X}}", from=1-3, to=1-4]
		\arrow["{r ( (\varepsilon f)^b i_X)^b}"{description}, from=1-3, to=3-3]
		\arrow[""{name=3, anchor=center, inner sep=0}, "{r (\tilde{b}_{\ell r f}\inv)^b}"{description}, from=1-3, to=4-1]
		\arrow["{(r (\varepsilon f)^b i_X )^a}", from=1-4, to=3-4]
		\arrow[""{name=4, anchor=center, inner sep=0}, "{\overline{r}_{f^b i_X}}"{description}, from=3-3, to=3-4]
		\arrow["{r (\tilde{b}_{f}\inv)^b}"{description}, from=3-3, to=4-3]
		\arrow["{(r \tilde{b}_f\inv)^a}", from=3-4, to=4-4]
		\arrow[""{name=5, anchor=center, inner sep=0}, "{r (\varepsilon f)^b}"', from=4-1, to=4-3]
		\arrow[""{name=6, anchor=center, inner sep=0}, "{\overline{r}_f}"', from=4-3, to=4-4]
		\arrow["{\algtheta{A}}"{description}, draw=none, from=1, to=0]
		\arrow["{\natof{\overline{r}}}"{description}, draw=none, from=2, to=4]
		\arrow["{\natof{\tilde{b}}}"{description}, draw=none, from=3, to=5]
		\arrow["{\natof{\overline{r}}}"{description}, draw=none, from=4, to=6]
	\end{tikzcd}\qedshift\]
 \end{proof}

 \begin{remark}
    The reader may observe that the proof of \cref{transport-of-structure} actually show how to construct a lax algebra structure and lax algebra morphism from an adjunction. However, since (as discussed in \cref{lax-algebras}) in this paper we are concerned with pseudoalgebras, we have imposed additional assumptions to ensure invertibility of the algebra structure.
 \end{remark}

 For completeness, we also include the dual result, which is the analogue of \cite[Theorem~3.1]{kelly1974doctrinal}.

 \begin{corollary}
 	\label{transport-of-structure-counit-invertible}
    Let $(\nu, \varepsilon) \colon \ell \colon A \rightleftarrows B \colon r$ be a reflective adjunction in $\E$ (so that $\varepsilon$ is invertible) and let $\algA$ be a $T$-pseudoalgebra. If the 2-cell
    \begin{equation}
        \label{transport-2-cell-dual}
        \ell f^a \xto{\ell (\nu f)^a} \ell (r \ell f)^a
    \end{equation}
    is invertible for every $f \colon JX \to A$, then:
    \begin{enumerate}
        \item $B$ acquires a $T$-pseudoalgebra structure;
        \item $\ell$ acquires the structure of an algebra pseudomorphism and $r$ acquires the structure of a lax algebra morphism;
        \item the unit $\nu$ and counit $\varepsilon$ are algebra transformations.
    \end{enumerate}
    Hence, the adjunction $(\nu, \varepsilon) \colon \ell \dashv r$ lifts to an adjunction in $\AlgTl$. Conversely, if $(\nu, \varepsilon) \colon \ell \dashv r$ lifts to an adjunction in $\AlgTl$ then the 2-cell \eqref{transport-2-cell-dual}
    is invertible for every $f \colon JX \to A$ in $\E$.
 \end{corollary}

 \begin{proof}
    Recall that an adjunction $\ell \colon A \rightleftarrows B \colon r$ in $\E$ with unit
        $\nu \colon 1_A \to r \ell$
    and counit
        $\varepsilon \colon \ell r \to 1_B$
    is equivalently an adjunction in $\E\co$ with left adjoint $r$, right adjoint $\ell$, unit $\varepsilon$, and counit $\nu$.
    In this light, the conditions given are precisely the assumptions of \cref{transport-of-structure} with respect to the adjunction
        $r \dashv \ell$
    in $\E\co$. The claim then follows once one observes that a pseudoalgebra/pseudomorphism in $\E\co$ is equivalently a pseudoalgebra/pseudomorphism in $\E$, while a colax algebra morphism in $\E\co$ is equivalently a lax algebra morphism in $\E$.
 \end{proof}

\begin{corollary}
    $U_T \colon \AlgT \to \E$ is an equifibration in the sense of \cite[Definition~4.1]{campbell2020homotopy}.
\end{corollary}

\begin{proof}
    The adjoint equivalence lifting property follows from \cref{transport-of-structure}, while the local isofibration property follows from \cref{forgetful-pseudofunctor}.
\end{proof}

\section{Lax-idempotent pseudoalgebras}
\label{lax-idempotence}

As hinted at in \cref{cocontinuous-functors-are-P-morphisms,natural-transformations-are-P-transformations}, the presheaf construction is a particularly well-behaved relative pseudomonad. As shown in \cite[\S5]{fiore2018relative}, it is an example of a \emph{lax-idempotent relative pseudomonad}. In the non-relative setting, the lax-idempotent pseudomonads are, intuitively, those that correspond to cocompletions under a class of colimits~\cite{kock1995monads,zoberlein1976doctrines,power2000representation}; the notion of lax-idempotence for relative pseudomonads was introduced to capture analogous behaviour in the relative setting. To characterise the pseudoalgebras for the presheaf construction (as well as a family of related constructions exhibiting free cocompletions under other classes of colimits), it will be useful to develop the theory of lax-idempotence for relative pseudomonads further. In particular, we will observe that there is an unexpected bifurcation in the move from pseudomonads to relative pseudomonads that renders the notion of lax-idempotence introduced by \cite{fiore2018relative} too weak to establish several theorems of interest. The following notion of lax-idempotence for pseudoalgebras is intended to rectify this problem, as we will explain shortly.

\begin{definition}
    \label{lax-idempotent-pseudoalgebra}
    Let $\J$ be a pseudofunctor and let $T$ be a $J$-relative pseudomonad. A $T$-pseudoalgebra $\algA$ is \emph{lax-idempotent}\footnotemark{} if $\tilde a$ exhibits the unit of an adjunction
    \begin{equation}
    	\label{eq:lax-idempotent-adjunction}
    	\E[JX, A] \colon \ph^a \adj \ph \c i_X \cocolon \E[TX, A]
    \end{equation}
    for each $X \in \A$.%
    \footnotetext{As a compound adjective, it is appropriate always to hyphenate \emph{lax-idempotent}, and hence also \emph{lax-idempotence}. In particular, \emph{lax} is not acting as an adverb in this context: lax-idempotence should be viewed as a property reminiscent of, but not a particular case of, idempotence.}
\end{definition}

Let us immediately note that the definition of lax-idempotence for pseudoalgebras may be rephrased in terms of left extensions, as this perspective will be helpful later.

\begin{lemma}
    \label{lax-idempotent-iff-left-ext}
    A $T$-pseudoalgebra $\algA$ is lax-idempotent if and only if $\tilde{a}_f \colon f \tto f^a i_X$ exhibits $f^a$ as the left extension of $f$ along $i_X$ for each $f \colon JX \to A$ in $\E$.
\end{lemma}

\begin{proof}
    By definition, $\E$ admits left extensions along $i_X$ if and only if the precomposition functor $\ph \c i_X$ admits a left adjoint~\cite[\S2.2]{lack2010companion}.
\end{proof}

\begin{remark}
    The pseudoalgebra is \emph{pseudo-idempotent} if \eqref{eq:lax-idempotent-adjunction} is furthermore an adjoint equivalence. Dually, a $T$-pseudoalgebra is \emph{colax-idempotent} if $\tilde a\inv$ exhibits the counit of an adjunction
    \[\E[TX, A] \colon \ph \c i_X \adj \ph^a \cocolon \E[JX, A]\]
    for each $X \in \A$. Just as lax-idempotence is intended to capture colimit-like properties, colax-idempotence is intended to limit-like properties, while pseudo-idempotence captures properties that are both limit- and colimit-like, such as free (co)completions under absolute (co)limits. While we will focus on lax-idempotence in what follows, the results of this section carry through with appropriate modifications for pseudo- and colax-idempotent pseudoalgebras.
\end{remark}

The notion of \emph{lax-idempotent relative pseudomonad} was introduced in \cite[\S5]{fiore2018relative}. We shall take the following characterisation as the definition.\footnote{In particular, this characterisation justifies the simpler definition of lax-idempotence utilised in \cite[Definition~5.1]{slattery2023pseudocommutativity} (\cf{}~\cite[(5.1)]{fiore2018relative}).}

\begin{proposition}
    \label{lax-idempotent-iff-free-algebras-are}
    A relative pseudomonad is lax-idempotent in the sense of \cite[Definition~5.1]{fiore2018relative} if and only if every free pseudoalgebra is lax-idempotent, \ie $\eta$ exhibits the unit of an adjunction
    \begin{equation*}
        \label{lax-idempotence-condition}
        \E[JX, TY] \colon \ph^*_{X, Y} \adj \ph \c i_X \cocolon \E[TX, TY]
    \end{equation*}
    for each $X, Y \in \A$.
\end{proposition}

\begin{proof}
    \cite[Definition~5.1]{fiore2018relative} asks for the stated condition subject to two compatibility conditions. However, the two compatibility conditions are in fact redundant, following from \cite[Lemma~3.2(i \& iii)]{fiore2018relative}.
\end{proof}

Consequently, the free pseudoalgebras for a relative pseudomonad that is lax-idempotent in the sense of \cite{fiore2018relative} are themselves lax-idempotent. However, crucially, this is \emph{not} generally true for arbitary pseudoalgebras: that is, there exist lax-idempotent relative pseudomonads whose pseudoalgebras are not lax-idempotent. The following provides a particularly simple counterexample in one dimension.

\begin{example}
    \label{idempotence-counterexample}
    Let $\E$ be the free category on the following graph,
    \[\begin{tikzcd}
        J & T & E
        \arrow["i", from=1-1, to=1-2]
        \arrow["f", curve={height=-6pt}, from=1-2, to=1-3]
        \arrow["{f'}"', curve={height=6pt}, from=1-2, to=1-3]
    \end{tikzcd}\]
    subject to $f i = f' i$. We can view $J$ and $T$ as functors $1 \to \E$ from the terminal category. $T$ may then be equipped with a relative monad structure $T$: the unit is $i$, and the extension operator sends the morphism $i$ to the identity on $T$: this is trivially a bijection because $T$ has no nontrivial endomorphisms. Thus $T$ is lax-idempotent (in fact, it is strictly idempotent). However, $E$ has two distinct algebra structures: one sends the unique map $J \to E$ to $f$, and the other sends the unique map to $f'$. Thus $E$ admits an algebra structure that is not lax-idempotent.
\end{example}

The existence of \cref{idempotence-counterexample} may be surprising, because in the \emph{non-relative} setting, no counterexamples exist; this explains why the definition of \cite{fiore2018relative} seems reasonable when one does not consider non-free pseudoalgebras.

\begin{proposition}
    \label{li-implies-ali-for-non-relative}
    A (non-relative) pseudomonad $T$ is lax-idempotent if and only if every $T$\nbh{}pseudoalgebra is lax-idempotent.
\end{proposition}

\begin{proof}
    It follows immediately from \cref{lax-idempotent-iff-free-algebras-are} that if every $T$-pseudoalgebra is lax-idempotent, then so is $T$. For the converse, note that a lax-idempotent pseudomonad is equivalently a left extension pseudomonad \cites[\S4]{marmolejo2012kan}[Theorem~5.3]{fiore2018relative}. Every pseudoalgebra for such is equivalently a left extension pseudoalgebra by \cite[\S5]{marmolejo2012kan}, hence lax-idempotent by \cref{lax-idempotent-iff-left-ext}.
\end{proof}

The above considerations suggest that a more appropriate definition of lax-idempotence for relative pseudomonads would be to ask that \emph{every} pseudoalgebra is lax-idempotent. We will see in the remainder of this section that this stronger notion satisfies various expected properties known to hold for lax-idempotent (non-relative) pseudomonads. However, to avoid confusion with \cite{fiore2018relative}, we will continue to use \emph{lax-idempotent relative pseudomonad} in this paper to refer to the weaker notion appearing in \cref{lax-idempotent-iff-free-algebras-are}.

The following characterisation of lax-idempotent pseudoalgebras in terms of lax morphisms, together with \cref{2-cells-between-lax-morphisms-of-lax-idempotent-pseudoalgebras-are-transformations}, motivates the terminology \emph{lax-idempotent} (in the terminology of \cite{kelly1997property}, this is a strengthening of \emph{property (AEL)}).

\begin{theorem}
    \label{lax-idempotent-iff-unique-lax-morphism-structure}
    Let $\J$ be a pseudofunctor and let $T$ be a $J$-relative pseudomonad. A $T$-pseudoalgebra $\algB$ is lax-idempotent if and only if, for every $T$-pseudoalgebra $\algA$, every 1-cell $f \colon A \to B$ in $\E$ admits a unique lax morphism structure $\overline f$. In this case,
    	$\overline f_g \colon (fg)^b \tto f g^a$
    is the unique 2-cell corresponding to
    	$f \tilde{a}_g \colon fg \tto f g^a i_X$
    under the adjunction \eqref{eq:lax-idempotent-adjunction}.
\end{theorem}

\begin{proof}
    ($\implies$)
    Suppose $\algB$ is a lax-idempotent $T$-pseudoalgebra, so that we have adjunctions $(\tilde b,\beta) \colon (-)^b \dashv (-) \c i_X$ for each $X \in \A$. Let $\algA$ be a $T$-pseudoalgebra and let $f \colon A \to B$ be a 1-cell in $\E$.
    If $f$ were to have lax morphism structure then each component $\overline f_g \colon (fg)^b \tto fg^a$ would correspond under the adjunction
    \eqref{eq:lax-idempotent-adjunction} to the following composite:
    \begin{equation}
        \label{eq:candidate-lax-morphism}
        fg \xto{\tilde b_{fg}} (fg)^b i_X \xto{\overline f_g i_X} fg^a i_X
    \end{equation}
    By the unit law for lax morphisms \eqref{lax-morph-tildes}, this composite is equal to $f \tilde a_g \colon fg \tto fg^a i_X$. Hence supposing lax morphism structure exists, it is necessarily unique.

    It remains to show that \eqref{eq:candidate-lax-morphism} indeed gives $f$ the structure of a lax morphism. The unit law \eqref{lax-morph-tildes} holds by definition; it remains to show the associativity law \eqref{lax-morph-hats}.
    We do so by taking transposes, and filling in the resulting diagram as shown below.
    \[\begin{tikzcd}
    	{(hg)^bf} && {((hg)^bf)^b i_X} \\
    	\\
    	{((hg)^bf)^bi_X} & {hg^af} & {(hg)^b f^* i_X} \\
    	\\
    	{(hg^af)^b i_X} & {h(g^af)^a i_X} & {hg^af^* i_X}
    	\arrow["{\tilde{b}_{(hg)^b f}}", from=1-1, to=1-3]
    	\arrow["{\tilde{b}_{(hg)^b f}}"', from=1-1, to=3-1]
    	\arrow[""{name=0, anchor=center, inner sep=0}, "{\overline{h}_g f}"{description}, from=1-1, to=3-2]
    	\arrow[""{name=1, anchor=center, inner sep=0}, "{(hg)^b\eta_f}"{description}, from=1-1, to=3-3]
    	\arrow["{\hat{b}_{hg, f} i_X}", from=1-3, to=3-3]
    	\arrow["{(\overline{h}_g f)^b i_X}"', from=3-1, to=5-1]
    	\arrow[""{name=2, anchor=center, inner sep=0}, "{\tilde{b}_{h g^a f}}"{description}, curve={height=6pt}, from=3-2, to=5-1]
    	\arrow["{h \tilde{a}_{g^a f}}"{description}, from=3-2, to=5-2]
    	\arrow[""{name=3, anchor=center, inner sep=0}, "{h g^a \eta_f}"{description}, curve={height=-6pt}, from=3-2, to=5-3]
    	\arrow["{\overline{h}_g f^* i_X}", from=3-3, to=5-3]
    	\arrow["{\overline{h}_{g^a f} i_X}"', from=5-1, to=5-2]
    	\arrow["{h \hat{a}_{g, f} i_X}"', from=5-2, to=5-3]
    	\arrow["{\algeta{B}}"{description}, draw=none, from=1-3, to=1]
    	\arrow["{\natof\tilde{b}}"{description}, shift right=4, draw=none, from=2, to=0]
    	\arrow["{\algeta{A}}"{description, pos=0.6}, draw=none, from=3, to=5-2]
    	\arrow["{\morunit{h}}"{description, pos=0.6}, draw=none, from=2, to=5-2]
    \end{tikzcd}\]
    Thus, defining $\overline f_g$ to be the transpose of $f\tilde a_g$ gives $f$ the structure of a lax morphism.

    ($\impliedby$) Let $\algB$ be a $T$-pseudoalgebra such that, for every $T$-pseudoalgebra $\algA$ and 1-cell $f \colon A \to B$ in $\E$, there is a unique choice of 2-cells $\overline f$ equipping $f$ with the structure of a lax morphism $(f, \overline f)$. Fix $X \in \A$. We must find a natural transformation
    $\beta \colon ({-} \c i_X)^b \tto 1$ forming the counit of an adjunction
    $(\tilde b,\beta) \colon (-)^b \dashv (-) \c i_X$. Taking $\algA$ to be the free $T$-pseudoalgebra on $X$ (\cref{free-pseudoalgebra}), we have that, for each $f \colon TX \to B$, there exists a unique lax morphism structure on $f$ with components
    $\overline f_g \colon (fg)^b \tto fg^*$.
    So we may define the components of $\beta$ to be as follows:
    \[\beta_f \defeq \Big( (f i_X)^b \xto{\overline f_i} f i_X^* \xto{f\theta_X} f \Big)\]

    We now need to check the triangle laws, which are the following two diagrams.
    \[
    \begin{tikzcd}
        {fi_X} && {(fi_X)^b i_X} \\
        && {f (i_X)^*i_X} \\
        && {fi_X}
        \arrow["{\tilde b_{f i_X}}", from=1-1, to=1-3]
        \arrow[""{name=0, anchor=center, inner sep=0}, "{f\eta_{i_X}}"{description}, curve={height=6pt}, from=1-1, to=2-3]
        \arrow[curve={height=12pt}, Rightarrow, no head, from=1-1, to=3-3]
        \arrow["{\overline{f}_{i_X} i_X}", from=1-3, to=2-3]
        \arrow["{f\theta_Xi_X}", from=2-3, to=3-3]
        \arrow["{\morunit{f}}"{description, pos=0.4}, draw=none, from=1-3, to=0]
    \end{tikzcd}
    \hspace{4em}
    \begin{tikzcd}
        {g^b} && {(g^bi_X)^b} \\
        && {g^b (i_X)^*} \\
        && {g^b}
        \arrow[""{name=0, anchor=center, inner sep=0}, "{(\tilde b_g)^b}", from=1-1, to=1-3]
        \arrow[""{name=1, anchor=center, inner sep=0}, curve={height=12pt}, Rightarrow, no head, from=1-1, to=3-3]
        \arrow["{\overline{g^b} = \hat b_{g,i_X}}", from=1-3, to=2-3]
        \arrow["{g^b \theta_X}", from=2-3, to=3-3]
        \arrow["{\algtheta{B}}"{description}, shift left=5, draw=none, from=0, to=1]
    \end{tikzcd}
    \]
	The unlabelled triangle in the left-hand diagram is \cite[Lemma~3.2(iii)]{fiore2018relative}.
    It remains to show that the family $\beta$ is natural. This follows from commutativity of the following diagram, for each $f, f' \colon TX \to B$ and $\rho \colon f \tto f'$.
    \[
    \begin{tikzcd}[column sep = 3.2em]
    	{(fi_X)^b} & {f (i_X)^*} & f \\
    	{(f'i_X)^b} & {f' (i_X)^*} & {f'}
    	\arrow[""{name=0, anchor=center, inner sep=0}, "{\overline{f}_{i_X}}", from=1-1, to=1-2]
    	\arrow["{(\rho i_X)^b}"', from=1-1, to=2-1]
    	\arrow["{f \theta_X}", from=1-2, to=1-3]
    	\arrow["{\rho (i_X)^*}"{description}, from=1-2, to=2-2]
    	\arrow["\rho", from=1-3, to=2-3]
    	\arrow[""{name=1, anchor=center, inner sep=0}, "{\overline{f'}_{i_X}}"', from=2-1, to=2-2]
    	\arrow["{f' \theta_X}"', from=2-2, to=2-3]
    	\arrow["{\trans{\rho}}"{description}, draw=none, from=0, to=1]
        \arrow[
            "\beta_f",
            rounded corners,
            to path=
            { -- ([yshift=.7cm]\tikztostart.center)
            -- ([yshift=.7cm]\tikztotarget.center)
            \tikztonodes
            -- (\tikztotarget.north) },
            from=1-1, to=1-3
        ]
        \arrow[
            "\beta_{f'}"',
            rounded corners,
            to path=
            { -- ([yshift=-.7cm]\tikztostart.center)
            -- ([yshift=-.7cm]\tikztotarget.center)
            \tikztonodes
            -- (\tikztotarget.south) },
            from=2-1, to=2-3
        ]
    \end{tikzcd}
    \]
    Hence $\beta$ is indeed the counit of an adjunction
    	$(\tilde b,\beta) \colon (-)^b \dashv (-) \c i_X$
    witnessing $\algB$ as lax-idempotent.
\end{proof}

As in the non-relative setting (see \cite[\S6]{kelly1997property}), the observation that 1-cells have unique lax morphism structure has a number of useful further consequences. In the terminology of \cite{kelly1997property}, the next lemma proves an analogue of \emph{property (2L)}: every 2-cell between lax morphisms with lax-idempotent codomain is a transformation.

\begin{corollary}
    \label{2-cells-between-lax-morphisms-of-lax-idempotent-pseudoalgebras-are-transformations}
    Let $\J$ be a pseudofunctor, $T$ be a $J$-relative pseudomonad, $\algA$ be an $T$-pseudoalgebra, and $\algB$ be a lax-idempotent $T$-pseudoalgebra. The forgetful functor
    \[
    	U_T \colon \AlgTl{\big[\algA, \algB\big]} \to \E[A, B]
    \]
    given by the action of the strict pseudofunctor $U_T \colon \AlgTl \to \E$ is fully faithful.
\end{corollary}

\begin{proof}
	Faithfulness follows from \cref{forgetful-pseudofunctor}.
	For fullness, let $\phi \colon f \tto f'$ be any 2-cell in $\E$. By \cref{lax-idempotent-iff-unique-lax-morphism-structure}, $f$ and $f'$ have unique lax algebra structures $\overline f$ and $\overline{f'}$, respectively. To show that $\phi$ is a transformation we need to show the square below commutes for all $g \colon JX \to A$:
    \[\begin{tikzcd}
        {(fg)^b} & {(f'g)^b} \\
        {fg^a} & {f'g^a}
        \arrow["{(\phi g)^b}", from=1-1, to=1-2]
        \arrow["{\overline f_g}"', from=1-1, to=2-1]
        \arrow["{\overline{f'}_g}", from=1-2, to=2-2]
        \arrow["{\phi g^a}"', from=2-1, to=2-2]
    \end{tikzcd}\]
    By \cref{lax-idempotent-iff-unique-lax-morphism-structure}, $(fg)^b \xto{\overline f_g} f g^a$
    is equal to the composite
    \begin{equation}
        (fg)^b \xto{(f \tilde{a}_g)^b} (f g^a i_X)^b \xto{\beta_{f g^a}} f g^a
    \end{equation}
    (and similarly for $f'$), where $\beta$ is the counit of the adjunction \eqref{eq:lax-idempotent-adjunction}.
    Hence the required naturality square decomposes as follows, and so every 2-cell is a transformation.
    \[\begin{tikzcd}
        {(fg)^b} && {(f'g)^b} \\
        {(fg^ai_X)^b} && {(f'g^ai_X)^b} \\
        {fg^a} && {f'g^a}
        \arrow["{(\phi g)^b}", from=1-1, to=1-3]
        \arrow["{(f\tilde a_g)^b}"', from=1-1, to=2-1]
        \arrow["{(f'\tilde a_g)^b}", from=1-3, to=2-3]
        \arrow[""{name=0, anchor=center, inner sep=0}, "{(\phi g^ai_X)^b}"{description}, from=2-1, to=2-3]
        \arrow["{\beta_{fg^a}}"', from=2-1, to=3-1]
        \arrow["{\beta_{f'g^a}}", from=2-3, to=3-3]
        \arrow[""{name=1, anchor=center, inner sep=0}, "{\phi g^a}"', from=3-1, to=3-3]
        \arrow["{\natof\beta}"{description}, draw=none, from=0, to=1]
    \end{tikzcd}\qedshift\]
\end{proof}

One useful consequence of the preceding results is the observation that left adjoint 1-cells between lax-idempotent algebras are always pseudomorphisms. This agrees with our intuition that lax-idempotence corresponds to colimit structure, as left adjoints always preserve colimits. The following was proven in the strict setting in \cite[Proposition~2.5]{kock1995monads} and for pseudomonads in \cite[Proposition~5.1]{marmolejo1997doctrines} (both under the stronger assumption that the pseudomonad itself is lax-idempotent). As observed in \cite{kelly1974doctrinal}, it is a consequence of doctrinal adjunction.

\begin{proposition}
    \label{left-adjoints-are-pseudomorphisms}
    Let $\J$ be a pseudofunctor, let $T$ be a $J$-relative pseudomonad, and let $\algA$ and $\algB$ be lax-idempotent $T$-pseudoalgebras. Every left adjoint 1-cell $A \to B$ in $\E$ is a pseudomorphism of $T$-pseudoalgebras.
\end{proposition}

\begin{proof}
    Let $\ell \colon A \rightleftarrows B \cocolon r$ be an adjunction in $\E$. Since $\algA$ and $\algB$ are lax-idempotent, $\ell$ and $r$ acquire unique lax morphism structures by \cref{lax-idempotent-iff-unique-lax-morphism-structure}, for which the unit and counit 2-cells are transformations by \cref{2-cells-between-lax-morphisms-of-lax-idempotent-pseudoalgebras-are-transformations}, so that the adjunction lifts to an adjunction in $\AlgTl$. Consequently, by \cref{adjunctions-lift-to-algebras}, $\ell$ is a pseudomorphism.
\end{proof}

This in turn implies that lax-idempotent pseudoalgebra structure is essentially unique.

\begin{corollary}
    \label{lax-pseudoalgebra-structures-are-essentially-unique}
    Any two lax-idempotent pseudoalgebra structures on the same object are equivalent.
\end{corollary}

\begin{proof}
    By \cref{left-adjoints-are-pseudomorphisms,2-cells-between-lax-morphisms-of-lax-idempotent-pseudoalgebras-are-transformations}, the trivial adjunction $1_A \adj 1_A$ in $\E$ extends to an adjoint equivalence between any pair of pseudoalgebras with the same carrier in $\AlgT$.
\end{proof}

\begin{remark}
    Together with \cref{lax-idempotent-iff-unique-lax-morphism-structure}, this shows that lax-idempotent pseudoalgebra structure is \emph{property-like} in the sense of \cite[\S4]{kelly1997property}. However, we do not know whether they are furthermore \emph{fully property-like} in the sense of \cite[\S5]{kelly1997property}, which would require every 2-cell between colax morphisms to have the structure of a transformation (see \cref{lax-fixed-points} for further discussion).
\end{remark}

Finally, we observe that lax-idempotent structure transports across reflections.

\begin{proposition}
    Under the assumptions of \cref{transport-of-structure-counit-invertible}, if the $T$-pseudoalgebra $\algA$ is lax-idempotent, then so is the induced $T$-pseudoalgebra structure on $B$.
\end{proposition}

\begin{proof}
    For each object $X \in \A$, we have adjunctions as follows:
    \[\begin{tikzcd}[column sep=large]
        {\E[JX, B]} & {\E[JX, A]} & {\E[TX, A]} & {\E[TX, B]}
        \arrow[""{name=0, anchor=center, inner sep=0}, "{r \c \ph}"', shift right=2, hook', from=1-1, to=1-2]
        \arrow[""{name=1, anchor=center, inner sep=0}, "{\ell \c \ph}"', shift right=2, from=1-2, to=1-1]
        \arrow[""{name=2, anchor=center, inner sep=0}, "{\ph^a_X}", shift left=2, hook, from=1-2, to=1-3]
        \arrow[""{name=3, anchor=center, inner sep=0}, "{\ph \c i_X}", shift left=2, from=1-3, to=1-2]
        \arrow[""{name=4, anchor=center, inner sep=0}, "{\ell \c \ph}", shift left=2, from=1-3, to=1-4]
        \arrow[""{name=5, anchor=center, inner sep=0}, "{r \c \ph}", shift left=2, hook, from=1-4, to=1-3]
        \arrow["\dashv"{anchor=center, rotate=-90}, draw=none, from=2, to=3]
        \arrow["\dashv"{anchor=center, rotate=-90}, draw=none, from=1, to=0]
        \arrow["\dashv"{anchor=center, rotate=-90}, draw=none, from=4, to=5]
    \end{tikzcd}\]
    Consequently, by \cite[Lemma~2.7]{arkor2024adjoint}, we have $\ell \c (r\ph)^a_X \adj \ell r \ph \c i_X \iso \ph \c i_X$, where the left adjoint is precisely the extension operator $\ph^b_X$ for the induced pseudoalgebra structure on $B$ (\cf~\eqref{reflected-algebra-structure}), and it is straightforward to verify the unit is precisely $\tilde b$.
\end{proof}

\subsection{Left extension pseudoalgebras}

In \cite[\S5]{fiore2018relative}, building upon the work of \cite{marmolejo2012kan}, the authors showed that the definition of a relative pseudomonad could be simplified under the assumption of lax-idempotence, by taking advantage of the universal properties of left extensions. We shall show the same is true for pseudoalgebras, extending \cite[\S3 \& \S5]{marmolejo2012kan}. The following notion was considered in \cite[Definition~4.36]{lewicki2020categories} by the name of \emph{relative $T$-complete object}.

\begin{definition}
    \label{left-extension-algebra}
    Let $T$ be a relative pseudomonad. A \emph{left extension $T$-pseudoalgebra} comprises
    \begin{enumerate}
        \item an object $A$ in $\E$;
        \item a 1-cell $f^a_X \colon TX \to A$ for each $f \colon JX \to A$ in $\E$;
        \item an invertible 2-cell $\tilde a_f \colon f \tto f^a i_X$ for each $f \colon JX \to A$ in $\E$,
    \end{enumerate}
    such that
    \begin{enumerate}[resume]
        \item for each $f \colon JX \to A$ in $\E$, the 2-cell $\tilde a_f$ exhibits $f^a$ as the left extension of $f$ along $i_X$;
        \item for each $f \colon JX \to TY$ and $g \colon TY \to A$ in $\E$, the 2-cell $g^a \eta_f \colon g^a f \tto g^a f^* i_X$ exhibits $g^a f^*$ as the left extension of $g^a f$ along $i_X$.
        \qedhere
    \end{enumerate}
\end{definition}

The following is an analogue of \cite[Theorem~5.3(i \& v)]{fiore2018relative} for pseudoalgebras. That left extension $T$-pseudoalgebras induce $T$-pseudoalgebras is also proven in \cite[Lemma~4.37]{lewicki2020categories}.

\begin{proposition}
    \label{lax-idempotent-algebra-is-left-extension-algebra}
    Let $T$ be a relative pseudomonad. For each object $A$ in $\E$, there is a bijection between left extension $T$-pseudoalgebra structures on $A$ and lax-idempotent $T$-pseudoalgebra structures on~$A$.
\end{proposition}

\begin{proof}
    Given a left-extension $T$-pseudoalgebra $\algA$, and 1-cells $f \colon JX \to TY$ and $g \colon JY \to A$ in $\E$, the 2-cell $g^a \eta_f \colon g^a f \tto g^a f^* i_X$ induces a 2-cell $\hat a_{g, f} \colon (g^a f)^a \tto g^a f^*$ by the universal property of the left extension $(g^a f)^a$. The proof that this data satisfies the axioms of a $T$-pseudoalgebra follows as in the proof of \cite[Theorem~5.3(v ${\tto}$ i)]{fiore2018relative}. That the $T$-pseudoalgebra is lax-idempotent follows from \cref{lax-idempotent-iff-left-ext}. The converse is trivial. That this correspondence is bijective follows from uniqueness of the universal property of the left extension defining $\hat a_{g, f}$.
\end{proof}

\begin{remark}
    We may define a \emph{$T$-pseudoalgebra} for a left extension relative pseudomonad $T$ in the sense of \cite[Definition~5.2]{fiore2018relative}\footnotemark{} to simply be a left extension pseudoalgebra for the induced lax-idempotent relative pseudomonad~\cite[Theorem~5.3]{fiore2018relative}.
    However, note that the definition of left extension pseudoalgebra does not require the relative pseudomonad itself to be lax-idempotent.%
    \footnotetext{\cite{fiore2018relative} uses the term \emph{relative left Kan pseudomonad}, but we prefer a more descriptive term.}
\end{remark}

We may also formulate the notion of pseudomorphism between left extension pseudoalgebras in terms of left extensions. (By \cref{lax-idempotent-iff-unique-lax-morphism-structure}, the appropriate notion of lax morphism for left extension pseudoalgebras is uninteresting, being simply a 1-cell between the carriers.)

\begin{corollary}
    Let $\J$ be a pseudofunctor, $T$ be a $J$-relative pseudomonad, and $\algLxA$ and $\algLxB$ be left extension $T$-pseudoalgebras. A 1-cell $h \colon A \to B$ admits the (necessarily unique) structure of a pseudomorphism between the induced $T$-pseudoalgebras (\cref{lax-idempotent-algebra-is-left-extension-algebra}) if and only if, for each $f \colon JX \to A$ in $\E$, the 2-cell $h\tilde a_f$ exhibits $hf^a \colon TX \to B$ as the left extension of $hf \colon JX \to B$ along $i_X \colon JX \to TX$.
\end{corollary}

\begin{proof}
    This follows immediately from \cref{lax-idempotent-iff-unique-lax-morphism-structure}, the family of 2-cells exhibiting the unique lax morphism structure being precisely those induced by the universal property of the left extensions.
\end{proof}

\subsection{Lax fixed points}
\label{lax-fixed-points}

We conclude this section by discussing an aspect of lax-idempotence only present for non-relative pseudomonads. We do this for two reasons. The first is that our treatment of lax-idempotence for relative pseudomonads looks quite different from the usual treatments of lax-idempotent pseudomonads (\eg~in \cite{marmolejo1997doctrines,kelly1997property}), and it is helpful to clarify the relationship. The second is that it explains why some properties of lax-idempotent pseudomonads do not appear to generalise to relative pseudomonads.

Recall from \cref{li-implies-ali-for-non-relative} that a pseudomonad is lax-idempotent if and only if each of its pseudoalgebras is lax-idempotent.
There is an alternative characterisation of lax-idempotence that is more familiar in the literature on lax-idempotent pseudomonads. We recall this in \cref{lax-idempotent-iff-lax-fixed-points} after presenting a preliminary definition. In the following, we will continue to work with pseudomonads in extension form, \ie{} as pseudomonads relative to the identity, rather than in pseudomonoid form.

Recall from \cref{correspondence-between-algebra-presentations} that, given a pseudoalgebra in extension form $\algA$, we obtain the usual pseudoalgebra structure 1-cell as $(1_A)^a \colon TA \to A$.

\begin{definition}
    Let $T$ be a pseudomonad on a bicategory $\E$. A pseudoalgebra $\algA$ is a \emph{lax fixed point} if $\tilde a_{1_A}\inv \colon (1_A)^a \c i_A \tto 1_A$ exhibits the counit of an adjunction $(1_A)^a \adj i_A \colon A \to TA$.
\end{definition}

\begin{remark}
    \label{lax-fixed-point-and-colax-algebras}
    Given a pseudomonad $T$ and an object $A$ for which $i_A \colon A \to TA$ exhibits the right adjoint of a coreflection, $A$ is automatically equipped with the structure of a colax $T$-algebra~\cite[Lemma~2.4.16]{stepan2024lax}. When $T$ is lax-idempotent, every such colax structure is automatically pseudo~\cite[Proposition~2.5.11]{stepan2024lax}.
\end{remark}

Note that the notion of lax fixed point does not make sense for relative pseudomonads, because it is not generally possibly to form $(1_A)^a$ unless $J$ is the identity. Consequently, there is no analgoue of the following characterisation of lax-idempotent pseudomonads in the relative setting.

\begin{proposition}[{\cite[Theorem~10.7]{marmolejo1997doctrines}}]
    \label{lax-idempotent-iff-lax-fixed-points}
    A (non-relative) pseudomonad is lax-idempotent if and only if every pseudoalgebra is a lax fixed point.
\end{proposition}

\begin{example}
    \Cref{li-implies-ali-for-non-relative,lax-idempotent-iff-lax-fixed-points} together say that, given a pseudomonad $T$, every $T$-pseudoalgebra is lax-idempotent if and only if every $T$-pseudoalgebra is a fixed point. However, in general, these two properties do not coincide. For instance, let $T$ be any pseudomonad on $\CAT$. Then the terminal category is necessarily a lax-idempotent $T$-pseudoalgebra (since in that case \eqref{eq:lax-idempotent-adjunction} is trivial), but is a lax fixed point if and only if $T(1)$ admits a terminal object. In particular, taking $T$ the free strict monoidal category 2-monad, we have that $T(0) = 1$ and that $T(1)$ does not admit a terminal object, so this furthermore shows that it is not even true that free pseudoalgebras are lax-idempotent if and only if they are lax fixed points.
\end{example}

Lax fixed points admit a similar, but dual, characterisation to lax-idempotent pseudoalgebras in terms of lax morphisms (\cf~\cref{lax-idempotent-iff-unique-lax-morphism-structure}); this observation appears to be new, as lax fixed points do not appear to have previously been studied for arbitrary pseudomonads, rather than lax-idempotent pseudomonads (although see \cref{lax-fixed-point-and-colax-algebras}).

\begin{theorem}
    \label{lax-fixed-point-iff-unique-lax-morphism-structure}
    Let $T$ be a (non-relative) pseudomonad on a bicategory $\E$. A $T$-pseudoalgebra $\algA$ is a lax fixed point if and only if, for every $T$-pseudoalgebra $\algB$, every 1-cell $g \colon A \to B$ in $\E$ admits a unique lax morphism structure $\overline g$. Furthermore, in this case, the forgetful functor
    \[
    	U_T \colon \AlgTl{\big[\algA, \algB\big]} \to \E[A, B]
    \]
    given by the action of the strict pseudofunctor $U_T \colon \AlgTl \to \E$ is fully faithful.
\end{theorem}

\begin{proof}
    ($\implies$) Suppose $\algA$ is a lax fixed point. That (1) $g$ admits a unique lax morphism structure, and (2) that every 2-cell between lax morphisms is a transformation both follow from the proof of \cite[Lemma~2.4.17]{stepan2024lax}, observing that the proof goes through even without the assumption that the codomain is induced by an adjunction.

    ($\impliedby$) Taking $g \defeq i_A$, we have a candidate 2-cell for the unit as follows.
    \begin{equation}
        \nu \defeq \Big( 1_{TA} \xto{\theta_A\inv} (i_A)^* \xto{\iso} (i_A 1_A)^* \xto{\overline{i_A}_{1_A}} i_A (1_A)^a \Big)
    \end{equation}
    To verify that $\nu$ and $\tilde a_{1_A}\inv$ define the unit and counit respectively of an adjunction $(1_A)^a \adj i_A$, we must show the two triangle laws. The right triangle law is straightforward once the definitions are unfolded, as shown below. The cell marked ($*$) is \cite[Lemma~3.2(iii)]{fiore2018relative}.
    \[\begin{tikzcd}[column sep=large]
        {i_A} \\
        {1_{TA}i_A} \\
        {(i_A)^* i_A} & {i_A} \\
        {(i_A 1_A)^* i_A} \\
        {i_A (1_A)^a i_A} & {i_A 1_A} & {i_A}
        \arrow["\iso"', from=1-1, to=2-1]
        \arrow[""{name=0, anchor=center, inner sep=0}, Rightarrow, no head, from=1-1, to=3-2]
        \arrow["{\theta_A^{-1} i_A}"', from=2-1, to=3-1]
        \arrow["{\eta_{i_A}\inv}"{description}, from=3-1, to=3-2]
        \arrow[""{name=1, anchor=center, inner sep=0}, "\iso"', from=3-1, to=4-1]
        \arrow[""{name=2, anchor=center, inner sep=0}, "\iso"{description}, from=3-2, to=5-2]
        \arrow[Rightarrow, no head, from=3-2, to=5-3]
        \arrow["{\overline{i}_{1_A} i_A}"', from=4-1, to=5-1]
        \arrow[""{name=3, anchor=center, inner sep=0}, "{\eta_{i_A 1_A}\inv}"{description}, curve={height=-12pt}, from=4-1, to=5-2]
        \arrow["{i_A \tilde{a}_{1_A}\inv}"', from=5-1, to=5-2]
        \arrow["{\iso }"', from=5-2, to=5-3]
        \arrow["{(\ast)}"{description}, draw=none, from=3-1, to=0]
        \arrow["{\natof{\eta}}"{description}, shift left=2, draw=none, from=1, to=2]
        \arrow["{\morunit{i}}"{description, pos=0.6}, draw=none, from=3, to=5-1]
        %
        \arrow[
        "{\nu i_A}"',
        rounded corners,
        to path=
        { -- ([xshift=-1.3cm]\tikztostart.center)
        -- ([xshift=-1.3cm]\tikztotarget.center)
        \tikztonodes
        -- (\tikztotarget.west)
        },
        from=2-1, to=5-1
        ]
    \end{tikzcd}\]
    For the left triangle law, we use the uniqueness of lax morphism structures on 1-cells out of $A$. Recall from \cref{f-a has ps-morphism structure} that $(1_A)^a$ becomes a pseudomorphism by taking
    $\hat{a}_{1_A, k} \colon (1_A k)^a \tto (1_A)^a k^*$
    for each $k \colon JX \to A$. We thus have a composite lax morphism
    	$(i_A, \overline{i_A}) \c ((1_A)^a, \hat{a}) \colon \algA \to \algA$
    whose 2-cell component is given at $k \colon JX \to A$ by the following composite:
    \begin{equation}\label{lax-fixed-point-composite-1}
    	\big( (1_A)^a i_A k \big)^a
    		\xto{\hat{a}_{1_A, i_A k}}
    		(1_A)^a (i_A k)^*
    		\xto{(1_A)^a \overline{i_A}_k}
    		(1_A)^a i_A k^a
    \end{equation}
    On the other hand, $\tilde{a}_{1_A} \colon 1_A \xto{\iso} (1_A)^a i_A$ is an isomorphism in $\E$, so by the amnestic isofibration property of $U_T$ (\cref{forgetful-pseudofunctor}), we get a lax morphism structure on $(1_A)^a i_A$ induced from that on $1_A$. Explicitly, we have that the following defines a lax morphism structure on $(1_A)^a i_A$ for each $k \colon JX \to A$ in $\E$:
    \begin{equation} \label{lax-fixed-point-composite-2}
    	\big( (1_A)^a i_A k\big)^a
    	\xto{(\tilde{a}_{1_A}\inv k)^a}
    	( 1_A k )^a
    	\xto{\iso}
    	1_A k^a
    	\xto{\tilde{a}_{1_A} k^a}
    	(1_A)^a i_A k^a
    \end{equation}
    Since we have defined two lax morphism structures on $(1_A)^a i_A$, we must have that \eqref{lax-fixed-point-composite-1} $=$ \eqref{lax-fixed-point-composite-2}. Denoting this fact by ($*$) for $k = 1_A$, the triangle law follows from commutativity of the next diagram.
    \[\begin{tikzcd}[sep=large]
        {(1_A)^a} & {(1_A)^a 1_{TA}} & {(1_A)^a(i_A)^*} & {(1_A)^a(i_A 1_A)^*} & {(1_A)^a i_A (1_A)^a} \\
        && {\big( (1_A)^a i_A \big)^a} & {\big( (1_A)^a i_A 1_A \big)^a} & {1_A(1_A)^a } \\
        {(1_A)^a} &&& {(1_A 1_A)^a} & {(1_A)^a}
        \arrow["\iso", from=1-1, to=1-2]
        \arrow[Rightarrow, no head, from=1-1, to=3-1]
        \arrow[""{name=0, anchor=center, inner sep=0}, "{(1_A)^a \theta_A\inv}", from=1-2, to=1-3]
        \arrow["\iso"{description}, from=1-2, to=3-1]
        \arrow[""{name=1, anchor=center, inner sep=0}, "\iso", from=1-3, to=1-4]
        \arrow["{\hat{a}_{1_A, i_A}}"{description}, from=1-3, to=2-3]
        \arrow["{(1_A)^a \overline{i_A}_{1_A}}", from=1-4, to=1-5]
        \arrow["{\hat{a}_{1_A, i_A 1_A}\inv}"{description}, from=1-4, to=2-4]
        \arrow["{(\ast)}"{description}, draw=none, from=1-4, to=3-5]
        \arrow["{\tilde{a}_{1_A}\inv (1_A)^a}", from=1-5, to=2-5]
        \arrow[""{name=2, anchor=center, inner sep=0}, "\iso"{description}, from=2-3, to=2-4]
        \arrow[""{name=3, anchor=center, inner sep=0}, "{(\tilde{a}_{1_A}\inv)^a}"{description}, from=2-3, to=3-1]
        \arrow["{(\tilde{a}_{1_A}\inv 1_A)^a}"{description}, from=2-4, to=3-4]
        \arrow["\iso", from=2-5, to=3-5]
        \arrow["\iso"{description}, from=3-1, to=3-4]
        \arrow[curve={height=30pt}, Rightarrow, no head, from=3-1, to=3-5]
        \arrow["\iso"{description}, from=3-4, to=3-5]
        \arrow["{\algtheta{A}}"{description}, draw=none, from=0, to=3]
        \arrow["{\natof{\hat{a}}}"{description}, draw=none, from=1, to=2]
        %
        %
        \arrow[
            "{(1_A)^a \nu}",
            rounded corners,
            to path=
            { -- ([yshift=.8cm]\tikztostart.center)
            -- ([yshift=.8cm]\tikztotarget.center)
            \tikztonodes
            -- (\tikztotarget.north)
            },
            from=1-2, to=1-5
        ]
    \end{tikzcd}\qedshift\]
\end{proof}

The above characterisation is fundamental in establishing the following useful property of lax-idempotent pseudomonads, which in particular implies they are \emph{fully property-like} in the sense of \cite[\S5]{kelly1997property}.

\begin{proposition}
    \label{colax-is-invertible}
    Let $T$ be a pseudomonad, let $\algA$ and $\algB$ be pseudoalgebras, and let $f \colon A \to B$ be a 1-cell in $\E$. Suppose that $f$ admits a colax morphism structure. If $\algA$ is a lax fixed point then the canonical lax morphism structure of \cref{lax-fixed-point-iff-unique-lax-morphism-structure} is inverse to the colax morphism structure.
\end{proposition}

\begin{proof}
    The proof follows that of \cite[Lemma~6.5]{kelly1997property} exactly, observing that the full strength of the assumptions \ibid is not necessary. The key observation is that, since $\algA$ is a lax fixed point, there is an adjunction $\E[TA, B] \colon \ph \c i_A \adj \ph \c (1_A)^a \cocolon \E[A, B]$ and so equality of 2-cells between 1-cells $TA \to B$ may be tested by precomposing $i_A$. The unit laws for the lax and colax morphisms then imply that the given lax and colax morphism structures are inverse to one another.
\end{proof}

As a consequence of the characterisation theorem above, we recover the fact that colax morphisms of pseudoalgebras for lax-idempotent non-relative pseudomonads are automatically pseudomorphisms.

\begin{corollary}[{\cite[Lemma~6.5 \& Corollary~6.6]{kelly1997property}}]
    \label{colax-morphisms-are-pseudo}
    Every colax morphism between pseudoalgebras for a lax-idempotent pseudomonad is pseudo. Consequently, every 2-cell between colax morphisms is a transformation.
\end{corollary}

\begin{proof}
    By \cref{lax-idempotent-iff-lax-fixed-points}, every pseudoalgebra for a lax-idempotent pseudomonad is a lax fixed point. By \cref{lax-fixed-point-iff-unique-lax-morphism-structure}, every 1-cell between pseudoalgebras admits a unique lax morphism structure, for which every 2-cell is a transformation. The first claim thus follows from \cref{colax-is-invertible}. The second claim then follows immediately, because a 2-cell is a transformation for a pseudomorphism qua colax morphism if and only if it is a transformation for a pseudomorphism qua lax morphism.
\end{proof}

\begin{remark}
    \label{are-colax-morphisms-pseudo}
    We do not know whether an analogue of \cref{colax-morphisms-are-pseudo} holds for relative pseudomonads. As explained above, the proof strategy in the non-relative setting crucially relies on the fact that pseudoalgebras for lax-idempotent pseudomonads are lax fixed points (rather than the fact they are lax-idempotent) and so does not extend to the relative setting. However, we also do not know of an example of a colax morphism between lax-idempotent $T$-pseudoalgebras, for $T$ a relative pseudomonad, that is not a pseudomorphism.
\end{remark}

\section{Resolutions of relative pseudomonads and coherence}
\label{resolutions-and-coherence}

In the one-dimensional setting, the category of algebras for a relative monad $T$ satisfies a universal property among resolutions of $T$ (that is, relative adjunctions inducing $T$)~\cite[Theorem~2.12]{altenkirch2015monads}, generalising that for monads~\cite[Theorem~2.2]{eilenberg1965adjoint}. In the two-dimensional setting, the 2-category of strict algebras for a relative 2-monad satisfies an analogous universal property~\cite[Corollary~6.41]{arkor2024formal}. In this section, we establish the appropriate universal property in the non-strict setting. In fact, we shall establish a universal property not just for the bicategory of pseudoalgebras, but also for the Kleisli bicategory, as the universal property of the latter was merely alluded to in \cite[Remark~4.5]{fiore2018relative}. As a consequence, we shall deduce that the Kleisli bicategory embeds essentially \ff{}ly into the bicategory of pseudoalgebras,  before discussing some consequences.

\begin{definition}
    \label{resolution}
    Let $T$ be a $J$-relative pseudomonad. We define a 2-category $\Res(T)$ as follows.
    \begin{itemize}
        \item An object is a \emph{resolution} of $T$: a $J$-relative pseudoadjunction for which the induced relative pseudomonad is exactly $T$. Where there is no risk of ambiguity, we shall typically denote a resolution simply by its apex.
        \item A 1-cell from $(L \colon \A \to \C, R \colon \C \to \E, i, \sharp, \eta, \varepsilon)$ to $(L \colon \A \to \C', R \colon \C' \to \E, i', \sharp', \eta', \varepsilon')$ is a \emph{morphism} of resolutions: a pseudofunctor $M \colon \C \to \C'$ for which $R = R'M$ and for which $MLX = L'X$ for each $X \in \A$.
        \item A 2-cell from $M$ to $M'$ is a \emph{transformation} of resolutions: a pseudonatural identity\footnotemark{} $\varpi \colon M \tto M'$ such that $R'(\varpi_f) = 1_{Rf}$.
    \end{itemize}
    \footnotetext{A \emph{pseudonatural identity} is a pseudonatural transformation whose components are identities.}
    The assignment sending each resolution to its apex defines a 2-functor from $\Res(T)$ to the 2-category of bicategories, pseudofunctors, and pseudonatural identities, which is faithful on 1-cells and 2-cells (\cf~\cite[Theorem~3.2]{lack2010icons}).
\end{definition}

\Cref{resolution} defines the notion of resolution in the strictest possible sense, rather than up to a notion of equivalence. This is because we are interested in universal properties with respect to a fixed relative pseudomonad. Note that, as observed in \cite[Proposition 2.2.1(7)]{miranda2023enriched}, an analogous notion of resolution for a (non-relative) pseudomonad in pseudomonoid form (\ie in terms of a multiplication operator $\mu \colon TT \tto T$ rather than an extension operator) would be too strict: given a pseudomonad $T$, the composite pseudofunctor for the Kleisli pseudoadjunction sends a 1-cell $f \colon X \to Y$ to the 1-cell
\[TX \xto{Tf} TY \xto{Ti_Y} TTY \xto{m_Y} TY\]
which is not equal to $Tf \colon TX \to TY$, but only isomorphic via one of the unitors for $T$. Consequently, an advantage of working with pseudomonads in extension form (\ie{} as pseudomonads relative to an identity) is that the Kleisli pseudoadjunction induces $T$ strictly (\cf{}~\cite[Remark~4.5]{fiore2018relative}). This is because, in the relative setting, the pseudofunctoriality of $T$ is derived rather than extra structure.

The definition of a morphism of resolutions may appear too weak, as it does not require all of the structure of a relative pseudoadjunction to be preserved; as the following lemma shows, this is in fact automatic (for organisational convenience, we make reference to the subsequent \cref{right-adjoint-is-pseudoalgebra}, while noting that this does not depend on any of the preceding material).

\begin{lemma}
    Let $(L \colon \A \to \C, R \colon \C \to \E, i, \sharp, \eta, \varepsilon)$ and $(L \colon \A \to \C', R \colon \C' \to \E, i', \sharp', \eta', \varepsilon')$ be resolutions of a given relative pseudomonad and let $M \colon \C \to \C'$ be a morphism of resolutions.
    \begin{enumerate}
        \item \label{resolution-morphism-unit-compatibility} The pseudonatural transformations $i$ and $i'$ are equal, the modifications $\eta$ and $\eta'$ are equal, and the following two natural transformations are equal, pseudonaturally in $X \in \A$ and $C \in \C$.
        \[
        \begin{tikzcd}[column sep=2.7em]
            {\E[JX, RC]} & {\C[LX, C]} & {\C'[L'X, MC]} \\
            & {\E[JX, RC]} & {\E[JX, R'MC]}
            \arrow["{\sharp_{X, C}}", from=1-1, to=1-2]
            \arrow[""{name=0, anchor=center, inner sep=0}, Rightarrow, no head, from=1-1, to=2-2]
            \arrow["{M_{LX, C}}", from=1-2, to=1-3]
            \arrow["{R_{LX, C} \ph i_X}"{description}, from=1-2, to=2-2]
            \arrow["{R'_{L'X, MC} \ph i'_X}", from=1-3, to=2-3]
            \arrow[Rightarrow, no head, from=2-2, to=2-3]
            \arrow["{\eta_{X, C}}", shorten <=3pt, shorten >=3pt, Rightarrow, from=0, to=1-2]
        \end{tikzcd}
        \hspace{-1.2em}
        \begin{tikzcd}
            {\E[JX, R'MC]} & {\C'[L'X, MC]} \\
            & {\E[JX, R'MC]}
            \arrow["{\sharp'_{X, MC}}", from=1-1, to=1-2]
            \arrow[""{name=0, anchor=center, inner sep=0}, Rightarrow, no head, from=1-1, to=2-2]
            \arrow["{R'_{L'X, MC} \ph i'_X}", from=1-2, to=2-2]
            \arrow["{\eta'_{X, MC}}", shorten <=3pt, shorten >=3pt, Rightarrow, from=0, to=1-2]
        \end{tikzcd}
        \]
        \item \label{resolution-morphism-sharp-compatibility} The following diagram commutes, pseudonaturally in $X \in \A$ and $C \in \C$.
        \[\begin{tikzcd}[column sep=large]
            {\E[JX, RC]} & {\C[LX, C]} \\
            {\E[JX, R'MC]} & {\C'[L'X, MC]}
            \arrow["{\sharp_{X, C}}", from=1-1, to=1-2]
            \arrow[Rightarrow, no head, from=1-1, to=2-1]
            \arrow["{M_{LX, C}}", from=1-2, to=2-2]
            \arrow["{\sharp'_{X, MC}}"', from=2-1, to=2-2]
        \end{tikzcd}\]
        \item \label{morphism-of-resolutions-left-pseudoadjoints-commute} The pseudofunctors $ML$ and $L'$ are equal.
        \item \label{resolution-morphism-counit-compatibility} The following two natural transformations are equal, pseudonaturally in $X \in \A$ and $C \in \C$.
        \[
        \begin{tikzcd}
            {\C[LX, C]} \\
            {\E[JX, RC]} & {\C[LX, C]} & {\C'[L'X, MC]}
            \arrow["{R_{LX, C} \ph i_X}"', from=1-1, to=2-1]
            \arrow[""{name=0, anchor=center, inner sep=0}, Rightarrow, no head, from=1-1, to=2-2]
            \arrow["{\sharp_{X, C}}"', from=2-1, to=2-2]
            \arrow["{M_{LX, C}}"', from=2-2, to=2-3]
            \arrow["{\varepsilon_{X, C}}"', shorten >=3pt, Rightarrow, from=2-1, to=0]
        \end{tikzcd}
        \hspace{-4em}
        \begin{tikzcd}
            {\C[LX, C]} & {\C'[L'X, MC]} \\
            & {\E[JX, R'MC]} & {\C'[L'X, MC]}
            \arrow["{M_{LX, C}}", from=1-1, to=1-2]
            \arrow["{R'_{L'X, MC} \ph i'_X}"', from=1-2, to=2-2]
            \arrow[""{name=0, anchor=center, inner sep=0}, Rightarrow, no head, from=1-2, to=2-3]
            \arrow["{\sharp'_{X, MC}}"', from=2-2, to=2-3]
            \arrow["{\varepsilon'_{X, MC}}"', shorten >=3pt, Rightarrow, from=2-2, to=0]
        \end{tikzcd}
        \]
        \item \label{resolution-morphism-coincident-algebras} For each $C \in \C$, the canonical pseudoalgebra structure on $R(C)$ induced by \cref{right-adjoint-is-pseudoalgebra} is equal to the canonical pseudoalgebra structure on $R'(MC)$.\footnote{The analogous condition in dimension one was taken as an assumption in \cite[Theorem~2.12]{altenkirch2015monads}, but is automatic, as observed in \cite[Lemma~5.26]{arkor2024formal}.} Furthermore, for each 1-cell $h \colon C \to D$ in $\C$, the canonical pseudomorphism structure on $R(h)$ induced by \cref{right-adjoint-is-pseudoalgebra} is equal to the canonical pseudomorphism structure on $R(Mh)$.
    \end{enumerate}
\end{lemma}

\begin{proof}
    \begin{enumerate}
        \item By the definition of a resolution, the units of the relative pseudomonads induced by both relative pseudoadjunctions are equal, hence induce identical pseudonatural transformations; and similarly for the modifications. The equality of the two induced natural transformations is then trivial.
        \item We have the following pasting diagram, where the middle square commutes since $M$ is a morphism of resolutions.
        \begin{equation}
        \label{eq:resolution-morphism-unit-counit-compatibility}
        \begin{tikzcd}
            {\E[JX, RC]} & {\C[LX, C]} & {\C'[L'X, MC]} \\
            \\
            & {\E[JX, RC]} & {\E[JX, R'MC]} & {\C[LX, MC]}
            \arrow["{\sharp_{X, C}}", from=1-1, to=1-2]
            \arrow[""{name=0, anchor=center, inner sep=0}, Rightarrow, no head, from=1-1, to=3-2]
            \arrow["{M_{LX, C}}", from=1-2, to=1-3]
            \arrow[""{name=1, anchor=center, inner sep=0}, "{R_{LX, C} \ph i_X}"{description}, from=1-2, to=3-2]
            \arrow[""{name=2, anchor=center, inner sep=0}, "{R'_{L'X, MC} \ph i'_X}"{description}, from=1-3, to=3-3]
            \arrow[""{name=3, anchor=center, inner sep=0}, Rightarrow, no head, from=1-3, to=3-4]
            \arrow[Rightarrow, no head, from=3-2, to=3-3]
            \arrow["{\sharp'_{X, MC}}"', from=3-3, to=3-4]
            \arrow["{\eta_{X, C}}", shorten <=7pt, shorten >=7pt, Rightarrow, from=0, to=1-2]
            \arrow["{=}"{description}, draw=none, from=1, to=2]
            \arrow["{\varepsilon'_{X, MC}}"', shorten <=8pt, shorten >=8pt, Rightarrow, from=3-3, to=3]
        \end{tikzcd}
        \end{equation}
        Using \cref{resolution-morphism-unit-compatibility}, and the right triangle law for $\eta'$ and $\varepsilon'$, this 2-cell is equal to the identity.
        \item By definition, we have that $MLX = L'X$ for each $X \in \A$. We also have that the following diagram commutes pseudonaturally in $X, Y \in \A$.
        \[\begin{tikzcd}[column sep=huge]
            {\A[X, Y]} & {\E[JX, JY]} & {\E[JX, RLY]} & {\C[LX, LY]} \\
            {\A[X, Y]} & {\E[JX, JY]} & {\E[JX, R'L'Y]} & {\C'[L'X, L'Y]}
            \arrow["{J_{X, Y}}", from=1-1, to=1-2]
            \arrow[Rightarrow, no head, from=1-1, to=2-1]
            \arrow["{\E[JX, i_Y]}", from=1-2, to=1-3]
            \arrow[Rightarrow, no head, from=1-2, to=2-2]
            \arrow["{\sharp_{X, LY}}", from=1-3, to=1-4]
            \arrow[Rightarrow, no head, from=1-3, to=2-3]
            \arrow["{M_{LX, LY}}", from=1-4, to=2-4]
            \arrow["{J_{X, Y}}"', from=2-1, to=2-2]
            \arrow["{\E[JX, i_Y]}"', from=2-2, to=2-3]
            \arrow["{\sharp'_{X, L'Y}}"', from=2-3, to=2-4]
        \end{tikzcd}\]
        By \eqref{action-of-left-pseudoadjoint}, the top composite defines $L_{X, Y} \colon \A[X, Y] \to \C[LX, LY]$, while the bottom composite defines $L'_{X, Y} \colon \A[X, Y] \to \C'[L'X, L'Y]$.
        \item This follows from \cref{resolution-morphism-sharp-compatibility}, pasting $\varepsilon_{X, C}$ on the left-hand side of \eqref{eq:resolution-morphism-unit-counit-compatibility}, and then using the right triangle law.
        \item The following diagram commutes: the left square by \cref{resolution-morphism-sharp-compatibility} and the right square since $M$ is a morphism of resolutions.
        \[\begin{tikzcd}
            {\E[JX, RC]} & {\C[LX, C]} & {\E[RLX, RC]} \\
            & {\C'[MLX, MC]} \\
            {\E[JX, R'MC]} & {\C'[L'X, MC]} & {\E[R'L'X, R'MC]}
            \arrow["{\sharp_{X, C}}", from=1-1, to=1-2]
            \arrow[Rightarrow, no head, from=1-1, to=3-1]
            \arrow["{R_{LX, C}}", from=1-2, to=1-3]
            \arrow["{M_{LX, C}}"{description}, from=1-2, to=2-2]
            \arrow[Rightarrow, no head, from=1-3, to=3-3]
            \arrow[Rightarrow, no head, from=2-2, to=3-2]
            \arrow["{\sharp'_{X, MC}}"', from=3-1, to=3-2]
            \arrow["{R'_{L'X, MC}}"', from=3-2, to=3-3]
        \end{tikzcd}\]
        Hence the induced pseudoalgebras have the same extension operator.
        It remains to show that the associativity \eqref{right-adjoint-algebra-multiplicator} and unitality structure \eqref{right-adjoint-algebra-unitor} of the induced pseudoalgebras are equal: the former follows from \cref{resolution-morphism-sharp-compatibility,resolution-morphism-counit-compatibility}; the latter follows trivially from \cref{resolution-morphism-unit-compatibility}.

        That the pseudomorphism structures on $R(h)$ and $R'(Mh)$ coincide follows from the pseudonaturality of \cref{resolution-morphism-sharp-compatibility} together with the right pseudoadjoint compatibility condition for $M$.
        \qedhere
    \end{enumerate}
\end{proof}

\subsection{The bi-initial resolution}

We now show that the Kleisli bicategory for a relative pseudomonad $T$ admits a universal property with respect to resolutions of $T$. For non-relative pseudomonads, this was alluded to in \cite[Corollary~4.4]{cheng2003pseudo}, but does not appear to have been explicitly proven.

\begin{theorem}
    \label{initial-resolution}
    Let $\J$ be a pseudofunctor and let $T$ be a $J$-relative pseudomonad. The Kleisli resolution is bi-initial in $\Res(T)$.
\end{theorem}

\begin{proof}
    Let $(L, R, i, \sharp, \eta, \varepsilon)$ be a resolution of $T$ with apex $\C$. We construct a mediating pseudofunctor ${!}_{L, R} \colon \Kl(T) \to \C$ as follows.
    \begin{align*}
        \ob{{!}_{L, R}} & \defeq \ob L &
        ({!}_{L, R})_{X, Y} & \defeq \sharp_{X, LX} \colon \KlT[X, Y] = \E[JX, RLX] \to \C[LX, LY]
    \end{align*}
    Since $(L, R)$ is a resolution of $T$, we have $g^* = R(g^\sharp)$. Thus, given 1-cells $f \colon JX \to TY$ and $g \colon JY \to TZ$ in $\E$, the unit of the adjoint equivalence, the compositor of $R$, and the counit of the adjoint equivalence provide an invertible natural transformation:
    \begin{equation}
        \label{Kl-!-multiplicator}
        (g^* f)^\sharp = (R(g^\sharp)f)^\sharp \xto{\eta} (R(g^\sharp) R(f^\sharp) i_X)^\sharp \xto{\iso} (R(g^\sharp f^\sharp)i_X)^\sharp \xto{\varepsilon} g^\sharp f^\sharp
    \end{equation}
    A unitor of $\E$, the unitor of $R$, and the counit of the adjoint equivalence defining the pseudoadjunction provide an invertible natural transformation:
    \begin{equation}
        \label{Kl-!-unitor}
        {i_X}^{\sharp_{X, LX}} \xto{(\iso)^{\sharp_{X, LX}}} (1_{RLX} \c i_X)^{\sharp_{X, LX}} \xto{(\iso)^{\sharp_{X, LX}}} (R(1_{LX}) \c i_X)^{\sharp_{X, LX}} \xto{\varepsilon_{X, LX}} 1_{LX}
    \end{equation}
    The proof that these data define a pseudofunctor is directly analogous to the proof that $V_T \colon \KlT \to \E$ forms a pseudofunctor (\cite[Lemma~4.3]{fiore2018relative}), noting that the relevant structural isomorphisms $\mu$ and $\theta$ are induced by \eqref{Kl-!-multiplicator} and \eqref{Kl-!-unitor} respectively (\cf~\cite[Theorem~3.8]{fiore2018relative}).

    ${!}_{L, R} \colon \Kl(T) \to \C$ forms a morphism of resolutions of $T$: the compatibility condition for the left pseudoadjoints is immediate from the definition of $\ob{{!}_{L, R}}$, while the compatibility condition for the right pseudoadjoints follows from the fact that $(L, R, i, \sharp, \eta, \varepsilon)$ is a resolution of $T$.

    It remains to show that every morphism of resolutions is uniquely isomorphic to ${!}_{L, R}$. Let $M \colon \KlT \to \C$ be a morphism of resolutions. Since $\ob{\KlT} = \ob\A$, the compatibility condition for the left pseudoadjoints implies that
    \begin{equation}
        M(X) = M(K_T(X)) = L(X) = {!}_{L, R}(X)
    \end{equation}
    for each $X \in \A$. Now, for each 1-cell $f \colon JX \to TY$, the compatibility condition for the right pseudoadjoints states that $R(M(f)) = V_T(f) = f^* = R(f^\sharp)$. Therefore, we have an invertible 2-cell
    \begin{equation}
        \label{comparison-transformation}
        f^\sharp \xto{(\eta_f)^\sharp} (R(f^\sharp) i_X)^\sharp = (R(M(f)) i_X)^\sharp \xto{\varepsilon_{Mf}} Mf
    \end{equation}
    natural in $f$, defining a pseudonatural identity ${!}_{L, R} \tto M$. By \cref{resolution-morphism-counit-compatibility}, $\varepsilon_{Mf} = M(\eta_f\inv)$, and so applying $R$ to \eqref{comparison-transformation} gives the identity, as does applying $R$ to the inverse of \eqref{comparison-transformation}. Consequently, \eqref{comparison-transformation} defines an invertible transformation of resolutions.

    Finally, let $\varpi \colon {!}_{L, R} \tto {!}_{L, R}$ be a transformation of resolutions, so that, for each 1-cell ${f \colon JX \to RLY}$, we have an invertible 2-cell on the left below, for which the 2-cell on the right below is the identity.
    \[
    \begin{tikzcd}
        LX & LY \\
        LX & LY
        \arrow["{f^\sharp}", from=1-1, to=1-2]
        \arrow[""{name=0, anchor=center, inner sep=0}, Rightarrow, no head, from=1-1, to=2-1]
        \arrow[""{name=1, anchor=center, inner sep=0}, Rightarrow, no head, from=1-2, to=2-2]
        \arrow["{f^\sharp}"', from=2-1, to=2-2]
        \arrow["{\varpi_f}"{description}, draw=none, from=0, to=1]
    \end{tikzcd}
    \hspace{2cm}
    \begin{tikzcd}
        JX & RLX & RLY \\
        JX & RLX & RLY
        \arrow["{i_X}", from=1-1, to=1-2]
        \arrow[Rightarrow, no head, from=1-1, to=2-1]
        \arrow["{R(f^\sharp)}", from=1-2, to=1-3]
        \arrow[""{name=0, anchor=center, inner sep=0}, Rightarrow, no head, from=1-2, to=2-2]
        \arrow[""{name=1, anchor=center, inner sep=0}, Rightarrow, no head, from=1-3, to=2-3]
        \arrow["{i_X}"', from=2-1, to=2-2]
        \arrow["{R(f^\sharp)}"', from=2-2, to=2-3]
        \arrow["{R(\varpi_f)}"{description}, draw=none, from=0, to=1]
    \end{tikzcd}
    \]
    Now, since $\sharp_{X, RL}$ is functorial, applying it to the 2-cell above right gives the identity 2-cell. By pre- and postcomposing by the isomorphism $R(f^\sharp)^\sharp \iso f^\sharp$, we get that $\varpi_f$ is the identity. Consequently, ${!}_{L, R}$ has no nontrivial endomorphisms, which implies the uniqueness of the isomorphisms of morphisms of resolutions.
\end{proof}

We may use \cref{initial-resolution} to give a characterisation of those pseudoadjunctions that exhibit the Kleisli resolution: this is dual in a certain sense to the question of relative pseudomonadicity, and might thus be called a \emph{relative pseudo-opmonadicity theorem}. For relative monads, this appears as \cite[Proposition~3.1.11]{arkor2022monadic}, and for non-relative pseudomonads, appears as \cite[Theorem~4.2.2]{miranda2023enriched}.

Recall that a pseudofunctor $F \colon \A \to \B$ is \emph{essentially surjective} if, for each object $X \in \A$, there exists an object $Y \in \B$ such that $F(X) \equiv Y$; and is \emph{essentially \ff} if, for each pair of objects $X, Y \in \A$, the functor $F_{X, Y} \colon \A[X, Y] \to \B[FX, FY]$ is an equivalence.

\begin{theorem}
    \label{relative-opmonadicity}
    A relative pseudoadjunction exhibits the Kleisli resolution of a relative pseudomonad, up to biequivalence, if and only if the left pseudoadjoint is essentially surjective on objects.
\end{theorem}

\begin{proof}
    By definition, $K_T \colon \A \to \KlT$ is the identity on objects (\cref{Kleisli-resolution}). Conversely, suppose that the left pseudoadjoint $L \colon \A \to \C$ of a relative pseudoadjunction is essentially surjective on objects. Writing $T$ for the induced pseudomonad, there is an essentially \ff{} pseudofunctor $!_{L, R} \colon \KlT \to \C$ by \cref{initial-resolution}. By assumption, and the fact that $K_T$ is \ioo{}, this pseudofunctor is also essentially surjective on objects. Thus $!_{L, R}$ is a biequivalence~\cite[Theorem~7.4.1]{johnson20212}.
\end{proof}

We shall use the relative pseudo-opmonadicity theorem to give an alternative characterisation of the Kleisli bicategory for a relative pseudomonad.

\begin{lemma}
    \label{full-image}
    Every pseudofunctor $F \colon \A \to \B$ admits a unique factorisation, the \emph{full image factorisation}\footnote{The full image factorisation of a pseudofunctor was called the \emph{Gabriel factorization} of a pseudofunctor in \cite{gambino2017operads}.}, into an \ioo{} pseudofunctor followed by a \ff{} pseudofunctor.
    \[\A \epito \Im(F) \monoto \B\]
    Furthermore, if $\B$ is a 2-category, then so is $\Im(F)$.
\end{lemma}

\begin{proof}
    The bicategory $\Im(F)$ has the same objects as $\A$ and $\Im(F)[X, Y] \defeq \B[FX, FY]$, with the bicategory structure being inherited from $\B$. The pseudofunctor $\A \epito \Im(F)$ is given on hom-categories by $F$; whereas the (strict) pseudofunctor $\Im(F) \monoto \B$ is given on objects by $F$. Uniqueness follows by definition.
\end{proof}

Extending an observation in \cite[\S6.1]{arkor2022monadic}, taking the full image factorisation of a left relative pseudoadjoint produces a new relative pseudoadjunction inducing the same relative pseudomonad.

\begin{proposition}
    \label{full-image-resolution}
    Let $T$ be a $J$-relative pseudomonad and let $(L, R, i, \sharp, \eta, \varepsilon)$ be a resolution of $T$. Denote by $IL' = L$ the full image factorisation of $L$. Then:
    \begin{enumerate}
        \item $X \mapsto L'X$ defines a left $J$-relative pseudoadjoint to $RI$;
        \item the corresponding relative pseudoadjunction is a resolution of $T$;
        \item $I$ is a morphism of resolutions, and is the unique such.
    \end{enumerate}
\end{proposition}

\begin{proof}
    The unit of the relative pseudoadjunction is given by $i$, and the adjoint equivalence is given by the following.
    \[\E[JX, RIY] \xto{\ph^\sharp_{X, IY}} \C[IL'X, IY] = \Im(L)[L'X, Y]\]
    It is trivial that this defines a resolution of $T$. That $I$ defines a morphism of resolutions is by definition of the factorisation. Finally, any morphism of resolutions $I'$ must satisfy $I'L' = L$ by \cref{morphism-of-resolutions-left-pseudoadjoints-commute}, which implies uniqueness given that $L'$ is the full image.
\end{proof}

Consequently, every resolution of a relative pseudomonad gives an alternative presentation of the Kleisli bicategory.

\begin{corollary}
    \label{Kleisli-is-full-image-resolution}
    Let $T$ be a $J$-relative pseudomonad and let $(L \colon \A \to \C, R \colon \C \to \E, i, \sharp, \eta, \varepsilon)$ be a resolution of $T$. Then $\KlT$ is biequivalent to $\Im(L)$.
\end{corollary}

\begin{proof}
    Immediate from \cref{relative-opmonadicity}, since $\A \epito \Im(L)$ is \ioo{} by definition and forms a resolution of $T$ by \cref{full-image-resolution}.
\end{proof}

\begin{remark}
    A word of caution: while $\Im(L)$ forms a resolution of $T$ and is biequivalent to $\KlT$, it is not equivalent to $\KlT$ as an object of $\Res(T)$. A morphism of resolutions from $\Im(L)$ to $\KlT$ should send each 1-cell $f \colon LX \to LY$ to $Rf \c i_X \colon JX \to TY$ in $\KlT$. However, this does not satisfy the required compatibility condition for the right pseudoadjoints, since $Rf$ is merely isomorphic to $V_T(Rf \c i_X) = (Rf \c i_X)^*$. One could presumably rectify this by working with \emph{pseudomorphisms} of resolutions, which require the pseudoadjoints to commute only up to coherent pseudonatural isomorphism, but we shall have no need to do so here.
\end{remark}

\subsection{The 2-terminal resolution}

We now turn to the universal property of the bicategory of pseudoalgebras. First, we observe that right pseudoadjoints exhibit pseudoalgebras in the following sense.

\begin{lemma}
    \label{right-adjoint-is-pseudoalgebra}
    Let $(L, R, i, \sharp, \eta, \varepsilon)$ be a $J$-relative pseudoadjunction inducing a $J$-relative pseudomonad~$T$.
    \begin{enumerate}
        \item For each object $C \in \C$, the object $RC \in \E$ is canonically equipped with the structure of a $T$-pseudoalgebra.
        \item For each 1-cell $h \colon C \to D$ in $\C$, the 1-cell $Rh \colon RC \to RD$ in $\E$ is canonically equipped with the structure of a pseudomorphism.
        \item For each 2-cell $\alpha \colon h \tto h'$ in $\C$, the 2-cell $R\alpha \colon Rh \tto Rh'$ in $\E$ forms a transformation.
    \end{enumerate}
\end{lemma}

\begin{proof}
    Given an object $C \in \C$, we define an extension operator by the following assignment:
    \begin{equation}
        \label{right-adjoint-algebra-structure}
        \E[JX, RC] \xto{\sharp_{X, C}} \C[LX, C] \xto{R_{LX, C}} \E[RLX, RC]
    \end{equation}
    Observing that, since $(L, R)$ is a resolution of $T$, we have $f^* = R(f^\sharp)$ and hence $R(g^\sharp) R(f^\sharp) = R(g^\sharp) f^*$, the associativity structure is given by the following:
    \begin{equation}
        \label{right-adjoint-algebra-multiplicator}
        R((R(g^\sharp)f)^\sharp) \xto{R((R(g^\sharp)\eta)^\sharp)} R((R(g^\sharp) R(f^\sharp) i_X)^\sharp) \xto{\iso} R((R(g^\sharp f^\sharp)i_X)^\sharp) \xto{R(\varepsilon)} R(g^\sharp f^\sharp) \xto{\iso} R(g^\sharp) R(f^\sharp)
    \end{equation}
    and the unitality structure is given by the unit of the adjoint equivalence:
    \begin{equation}
        \label{right-adjoint-algebra-unitor}
        f \xto{\eta_f} R(f^\sharp) i_X
    \end{equation}
    The proof is directly analogous to the proof that relative pseudoadjunctions induce relative pseudomonads (\cite[Theorem~3.8]{fiore2018relative}), noting that the data of the pseudoalgebra \eqref{right-adjoint-algebra-structure}, \eqref{right-adjoint-algebra-multiplicator}, and \eqref{right-adjoint-algebra-unitor} correspond respectively to the data $\ph^*$, $\mu$, and $\eta$ of the relative pseudomonad induced by the relative pseudoadjunction.

    Given a 1-cell $h \colon C \to D$ in $\C$, the pseudomorphism structure on $R h \colon RC \to RD$ is given by the following invertible 2-cell, using pseudonaturality of $\sharp$ and pseudofunctoriality of $R$:
    \[\begin{tikzcd}[column sep=6em]
        {\E[JX, RC]} & {\E[JX, RD]} \\
        {\C[LX, C]} & {\C[LX, D]} \\
        {\E[RLX, RC]} & {\E[RLX, RD]}
        \arrow[""{name=0, anchor=center, inner sep=0}, "{\sharp_{X, C}}"', from=1-1, to=2-1]
        \arrow[""{name=1, anchor=center, inner sep=0}, "{R_{LX, C}}"', from=2-1, to=3-1]
        \arrow[""{name=2, anchor=center, inner sep=0}, "{\sharp_{X, D}}", from=1-2, to=2-2]
        \arrow[""{name=3, anchor=center, inner sep=0}, "{R_{LX, D}}", from=2-2, to=3-2]
        \arrow["{\E[RLX, Rh]}"', from=3-1, to=3-2]
        \arrow["{\E[JX, Rh]}", from=1-1, to=1-2]
        \arrow["{\C[LX, h]}"{description}, from=2-1, to=2-2]
        \arrow["\iso"{description}, draw=none, from=0, to=2]
        \arrow["\iso"{description}, draw=none, from=1, to=3]
    \end{tikzcd}\]
    That this defines a pseudomorphism follows from the adjoint equivalence $\ph^\sharp \adj R\ph i$. That, for each 2-cell $\alpha \colon h \tto h'$, the 2-cell $R\alpha \colon Rh \tto Rh'$ forms a transformation follows from pseudonaturality of $\sharp$.
\end{proof}

As anticipated in the introduction, the resolution associated to the bicategory of pseudoalgebras (\cref{Alg-resolution}) admits a universal property dual to that of the Kleisli bicategory (in fact, it is stricter).

\begin{theorem}
    \label{terminal-resolution}
    Let $\J$ be a pseudofunctor and let $T$ be a $J$-relative pseudomonad. The pseudoalgebra resolution is 2-terminal in $\Res(T)$.
\end{theorem}

\begin{proof}
    Let $(L, R, i, \sharp, \eta, \varepsilon)$ be a resolution of $T$. We construct a pseudofunctor ${!}_{L, R} \colon \C \to \AlgT$. The action of ${!}_{L, R}$ is given by equipping each object $RC \in \E$, for $C \in \C$, with the canonical pseudoalgebra structure of \cref{right-adjoint-is-pseudoalgebra}, and likewise for the 1-cells and 2-cells. Pseudofunctoriality of ${!}_{L, R}$ follows from that of $R$. The pseudofunctor $!_{L, R}$ forms a morphism of resolutions: the compatibility condition for the left pseudoadjoints follows since $(L, R, i, \sharp)$ is a resolution of $T$ and $!_{L, R}$ is defined in terms of the canonical pseudoalgebra structure on $R$, so that the pseudoalgebra structure on each $RLX$ is precisely the free pseudoalgebra structure on $X \in \A$ given by \cref{free-pseudoalgebra}; the compatibility condition for the right pseudoadjoints follows trivially.

    It remains to show that every morphism of resolutions is equal to $!_{L, R}$. Let $M \colon \C \to \AlgT$ be a morphism of resolutions. From \cref{Alg-resolution,resolution-morphism-coincident-algebras}, we have that, for each $C \in \C$, the pseudoalgebra structure on $U_TMC$ is equal to the canonical pseudoalgebra structure on $RC$, and that, for each 1-cell $h \colon C \to D$ in $\C$, the pseudomorphism structure on $U_TMh$ is equal to the canonical pseudomorphism structure on $Rh$. Since $U_T$ is locally faithful (\cref{forgetful-pseudofunctor}), this shows that ${!}_{L, R} = M$, which is automatically a transformation of resolutions.

    Finally, let $\varpi \colon !_{L, R} \tto !_{L, R}$ be a transformation of resolutions. We have that $U_T(\varpi_f) = 1_{Rf}$. Since $U_T$ is locally faithful, this implies that $\varpi_f$ is the identity.
\end{proof}

A resolution of a relative pseudomonad $T$ will be called \emph{relatively pseudomonadic} when the unique mediating pseudofunctor into the bicategory of pseudoalgebras is a biequivalence.

As a consequence of \cref{initial-resolution,terminal-resolution}, we deduce the existence of a canonical \emph{comparison pseudofunctor} from the Kleisli bicategory to the bicategory of algebras.

\begin{corollary}
    \label{comparison-pseudofunctor}
    There is a unique, \ff{} pseudofunctor $I_T \colon \KlT \to \AlgT$ such that the following diagrams commute: explicitly, $I_T$ is the embedding of the free $T$-pseudoalgebras.
    \[
    \begin{tikzcd}
        \KlT & \AlgT \\
        \A
        \arrow["{I_T}", from=1-1, to=1-2]
        \arrow["{K_T}", from=2-1, to=1-1]
        \arrow["{F_T}"', from=2-1, to=1-2]
    \end{tikzcd}
    \hspace{4em}
    \begin{tikzcd}
        \KlT & \AlgT \\
        & \E
        \arrow["{I_T}", from=1-1, to=1-2]
        \arrow["{V_T}"', from=1-1, to=2-2]
        \arrow["{U_T}", from=1-2, to=2-2]
    \end{tikzcd}
    \]
\end{corollary}

\begin{proof}
    The existence and uniqueness of such a pseudofunctor follows from \cref{terminal-resolution}: $I_T = {!}_{F_T, U_T}$ is the unique mediating pseudofunctor induced by the resolution associated to the bicategory of $T$-pseudoalgebras; \ffness{} follows from \cref{initial-resolution}. Explicitly, the action of $I_T$ is induced by the pseudofunctor $F_T$, which sends objects of $\A$ to free $T$-pseudoalgebras.
\end{proof}

\subsection{Relative pseudomonads induced by pseudomonads}

We use the universal properties exhibited above to make some observations regarding relative pseudomonads induced from pseudomonads by precomposition; these observations will be useful in \cref{presheaves-and-cocompleteness} when we study the relationship between the presheaf construction and the small presheaf construction.

\begin{lemma}
    \label{precomposition}
    Let $\J$ be a pseudofunctor.
    \begin{enumerate}
        \item Let $T$ be a pseudomonad on $\E$. The assignment $(A \in \A) \mapsto (TJA \in \E)$ extends to a $J$-relative monad.
        \item Let $R \colon \C \to \E$ be a pseudofunctor admitting a left pseudoadjoint $L \colon \E \to \C$. The assignment $(A \in \A) \mapsto (LJA \in \C)$ defines a left $J$-relative adjoint to $R$.
    \end{enumerate}
    Furthermore, if $(L, R)$ is a resolution of $T$, then $(LJ, R)$ is a resolution of $TJ$.
\end{lemma}

\begin{proof}
    It is trivial to check that the data of a pseudomonad and pseudoadjunction (viewed as being relative to the identity pseudofunctor on $\E$) restrict along $J$ as described.
\end{proof}

Given a pseudofunctor $\J$ and a pseudomonad $T$ on its codomain, it is natural to compare the Kleisli bicategories and bicategories of pseudoalgebras for $T$ and $TJ$. First, the Kleisli bicatgory for $TJ$ may be characterised simply in terms of that for $T$.

\begin{corollary}
    \label{restriction-of-Kleisli}
    Let $\J$ be a pseudofunctor and let $T$ be a pseudomonad on $E$. The Kleisli bicategory $\Kl(TJ)$ is biequivalent to the full image of the composite $\A \xto J \E \xto{K_T} \Kl(T)$. Consequently, the full image of the comparison pseudofunctor $!_{TJ} \colon \Kl(TJ) \to \PsAlg(TJ)$ is biequivalent to $\Im(F_T J)$.
\end{corollary}

\begin{proof}
    Both claims follows immediately from \cref{precomposition,Kleisli-is-full-image-resolution}, together with \cref{Kleisli-resolution,Alg-resolution} respectively.
\end{proof}

The bicategory of pseudoalgebras for $T$ is not in general biequivalent to that for $TJ$ (as is easily seen, for instance, by taking $J$ to be the unique pseudofunctor from the empty bicategory). However, there is a canonical pseudofunctor from the former to the latter.

\begin{corollary}
    \label{realisation-comparison}
    Let $T$ be a pseudomonad and let $\J$ be a pseudofunctor. There is a unique pseudofunctor $!_J \colon \AlgT \to \PsAlg(TJ)$ that commutes with the free and forgetful pseudofunctors.
    \[
    \begin{tikzcd}
        \AlgT & {\PsAlg(TJ)} \\
        \A
        \arrow["{!_J}", from=1-1, to=1-2]
        \arrow["{F_TJ}", from=2-1, to=1-1]
        \arrow["{F_{TJ}}"', from=2-1, to=1-2]
    \end{tikzcd}
    \hspace{4em}
    \begin{tikzcd}
        \AlgT & {\PsAlg(TJ)} \\
        & \E
        \arrow["{!_J}", from=1-1, to=1-2]
        \arrow["{U_T}"', from=1-1, to=2-2]
        \arrow["{U_{TJ}}", from=1-2, to=2-2]
    \end{tikzcd}
    \]
\end{corollary}

\begin{proof}
    By \cref{precomposition}, the following relative pseudoadjunction forms a resolution of $TJ$, after which the result follows from \cref{terminal-resolution}.
    \[\begin{tikzcd}
        & \AlgT \\
        \A && \E
        \arrow[""{name=0, anchor=center, inner sep=0}, "{F_TJ}", from=2-1, to=1-2]
        \arrow[""{name=1, anchor=center, inner sep=0}, "{U_T}", from=1-2, to=2-3]
        \arrow["J"', from=2-1, to=2-3]
        \arrow["\dashv"{anchor=center}, shift right=2, draw=none, from=0, to=1]
    \end{tikzcd}\qedshift\]
\end{proof}

\subsection{Coherence}

We end this section with an observation regarding an aspect of the embedding of free pseudoalgebras that does not arise in the one-dimensional setting. As mentioned in \cref{EM-2-category}, when $\E$ is a 2-category, so too is $\AlgT$. However, the same is not true for $\KlT$, which is a bicategory regardless of the strictness of the domain or codomain. The comparison pseudofunctor $I_T \colon \KlT \to \AlgT$ therefore provides an essentially \ff{} embedding of a bicategory into a 2-category. It is well known that every bicategory is biequivalent to a 2-category~\cite{maclane1985coherence,power1989general}. However, in general, the 2-category to which a given bicategory is biequivalent may not admit a convenient direct description. For instance, syntactic approaches to coherence give a description in terms of free 2-categories, whereas an approach via the bicategorical Yoneda embedding gives a description in terms of pseudopresheaves~\cite[Theorem~1.4]{gordon1995coherence}. Yet it is known that, somewhat surprisingly, in many cases naturally occurring bicategories do admit a natural strictification, in that there exists a biequivalent 2-category that arises by natural considerations, rather than for the purposes of strictification (\cf{}~\cite[\S1.5]{lack2010companion}). Our contention is that this tends to occur due to the following consequence of the following observation, which may be viewed as a coherence theorem associated to a relative pseudomonad, and does not appear to have been explicitly stated even for non-relative pseudomonads.

\begin{corollary}
    \label{coherence-theorem}
    Let $\J$ be a pseudofunctor and let $T$ be a $J$-relative pseudomonad. Suppose that $\E$ is a 2-category. The bicategory $\KlT$ is biequivalent to $\Im(F_T)$, the full and locally full sub-2-category of $\AlgT$ spanned by the free $T$-pseudoalgebras.
\end{corollary}

\begin{proof}
    From \cref{EM-2-category,comparison-pseudofunctor}.
\end{proof}

We shall give several applications of \cref{coherence-theorem} in \cref{examples-of-coherence}.

\section{Presheaves and cocompleteness}
\label{presheaves-and-cocompleteness}

In \cite[Example~4.2]{fiore2018relative}, it is established that the Kleisli bicategory $\Kl(P)$ associated to the presheaf construction is the bicategory $\Dist$ of distributors (\aka{} the bicategory $\b{Prof}$ of profunctors) between small categories (\cf{}~\cite{benabou1973distributeurs}). The objective of this section is to characterise the 2-category $\PsAlg(P)$ of pseudoalgebras and pseudomorphisms for the presheaf construction. In contrast to characterising the Kleisli bicategory, this is not a straightforward task. To illustrate why, we contrast it with a closely related construction, which we now recall.

\begin{definition}[{\cites[Remark~2.29]{ulmer1968properties}[Proposition~4.83]{kelly1982basic}[\S2]{day2007limits}}]
    The \emph{small presheaf construction} $A \mapsto [A\op, \Set]_s$ sends a locally small category $A$ to the full subcategory of $[A\op, \Set]$ spanned by presheaves that are small colimits of representables. This defines a pseudomonad $P_s$ on the 2-category $\CAT$ of locally small categories.
\end{definition}

Characterising the pseudoalgebras for $P_s$ is straightforward: since $P_s$ is a lax-idempotent pseudomonad, a locally small category $A$ admits the structure of a a $P_s$-pseudoalgebra if and only if the unit $A \to P_s(A)$ admits a left adjoint (\cref{lax-fixed-point-and-colax-algebras}). Every such category is reflective in a cocomplete category, so is itself cocomplete; conversely, every cocomplete category admits such a left adjoint via left extension. Consequently, a locally small category admits the structure of a $P_s$-pseudoalgebra if and only if it admits all small colimits. Similarly, the pseudomorphisms are precisely the cocontinuous functors, and the transformations are the natural transformations.

Following \cref{precomposition}, we may precompose any pseudomonad on $\CAT$ with the inclusion $J \colon \Cat \ffto \CAT$. Applying this construction to the small presheaf construction recovers the presheaf construction (\cref{presheaf-construction-pseudomonad}): in other words, $P = P_s J$.

In light of this observation, one might expect it to be similarly simple to prove that, like the $P_s$-pseudoalgebras, the $P$-pseudoalgebras are precisely the cocomplete locally small categories. However, there is a crucial difference. While $P$ is a lax-idempotent relative pseudomonad~\cite[Example~5.4]{fiore2018relative}, the pseudoalgebras for such do not admit a convenient characterisation in terms of reflectivity in a free pseudoalgebra, which was crucial for characterising the $P_s$-pseudoalgebras. Furthermore, although lax-idempotent pseudoalgebras \emph{do} admit a simpler description than arbitrary pseudoalgebras by \cref{lax-idempotent-algebra-is-left-extension-algebra}, the pseudoalgebras for a lax-idempotent relative pseudomonad are not necessarily also lax-idempotent (\cref{idempotence-counterexample}). Since we do not know a priori that the $P$-pseudoalgebras are lax-idempotent, we cannot use this description.

Consequently, a characterisation of the $P$-pseudoalgebras necessitates a fundamentally different proof strategy to the characterisation of the $P_s$-pseudoalgebras. We shall illustrate our approach by starting with a simpler example.

\begin{example}
    \label{adjoin-initial-object}
    Let $\ph_\bot$ be the 2-monad on $\CAT$ that freely adjoins an initial object. Suppose that we wish to characterise its algebras, but only permit ourselves to apply $\ph_\bot$ to small categories (in preparation for considering the restriction along $J \colon \Cat \ffto \CAT$).

    Let $\algA$ be a $\ph_\bot$-algebra (where $A$ is not necessarily small). The simplest example of a functor into $A$ with small domain is the unique functor $[]_A \colon 0 \to A$ from the empty category. Using the algebra structure, we obtain a functor $([]_A)^a \colon 0_\bot = 1 \to A$, which picks out an object $0_A$ in $A$.

    The next simplest example of a functor into $A$ with small domain is a functor $f \colon 1 \to A$ from the terminal category. Such a functor picks out an object $f(*)$ in $A$. Using the algebra structure, we obtain a functor $f^a \colon 1_\bot = \{ 0 \to 1 \} \to A$ from the interval category. Such a functor picks out a morphism in $A$. Furthermore, the unit law for the algebra implies that its codomain must be $f(*)$, while the extension law implies that its domain must be $0_A$. In other words, the structure of an algebra for $\ph_\bot$ exhibits a specified object $0_A$ and morphisms therefrom to each object in $A$. Furthermore, a calculation with the algebra axioms shows that every morphism from $0_A$ is given in this way (see \cref{colimits-in-pseudoalgebras}).

    Consequently, it is possible to prove that $\ph_\bot$-algebras admit initial objects, by extending only functors with domain $0$ and $1$. In the other direction, an easy calculation shows that categories with initial objects admit $\ph_\bot$-algebra structures. This shows that, for any pseudofunctor $J \colon \A \to \CAT$ with $\A$ a full sub-2-category of $\CAT$ containing $0$ and $1$, the algebras for the $J$-relative pseudomonad $J\ph_\bot$ are locally small categories with initial objects.
\end{example}

This example suggests it should be possible to characterise the $P$-pseudoalgebras by using their extension operators to construct colimiting cocones for any functor thereinto with small domain. We will show this is indeed the case in \cref{Phi-pseudoalgebras}. In fact, we will work more generally, and consider the relative pseudomonad for the free cocompletion under a class $\Phi$ of small colimits, in doing so characterising the pseudoalgebras for several more examples of relative pseudomonads from \cite{fiore2018relative}.

\begin{remark}
    It is perhaps worth saying a few words regarding the relationship between the presheaf construction and the small presheaf construction. Since the presheaf construction may be obtained from the small presheaf construction by restricting along $\Cat \ffto \CAT$, one might suppose the small presheaf construction to be the more fundamental of the two. However, as will be justified in \cref{Phi-pseudoalgebras}, the small presheaf construction may also be obtained from the presheaf construction. Thus the two constructions are, in an appropriate sense, equivalent. That the presheaf construction is the simpler of the two constructions suggests that it is the presheaf construction that deserves to be viewed as the fundamental construction, and the small presheaf construction as derived therefrom. This aligns with the perspective in \cite{altenkirch2015monads,arkor2022monadic} that relative monads may be fruitfully considered as presentations of monads on their codomain. As further evidence to support this perspective, it is really the Kleisli bicategory of $P$ (\ie the bicategory of small categories and distributors) rather than of $P_s$ (whose Kleisli bicategory comprises the locally small categories and \emph{small} distributors) that is of importance in category theory.
\end{remark}

\begin{remark}
    It was proven abstractly in \cite[Theorem~7.1]{kelly2000monadicity} that the small presheaf construction may be strictified, in the sense that there is a 2-monad on $\CAT$ whose 2-category of (strict) algebras and pseudomorphisms is biequivalent to $\PsAlg(P_s)$ over $\CAT$. By precomposing the inclusion $\Cat \ffto \CAT$, the presheaf construction $P$ may therefore similarly be strictified into a relative 2-monad. However, in doing so we lose the simplicity of the concrete description of the presheaf construction. An explicit description of the free strict cocompletion is given in \cite{beurier2021presentations}; the reader will note that it is significantly more complicated. Furthermore, such a description does not simplify the description of the algebras.
\end{remark}

\subsection{Colimits in pseudoalgebras}

We aim to carry out the construction of \cref{adjoin-initial-object} for more general classes of colimits. In fact, we will show that the argument therein applies under relatively few assumptions on the relative pseudomonad in question. The construction is essentially a study of cocones in pseudoalgebras. We recall in \cref{cocone} a description of cocones that will be convenient for what follows. This in turn requires a preparatory definition.

\begin{definition}
    Denote by $j_D \colon D \ffto D_\top$ the free category with a terminal object on $D$. Explicitly,
    \[\ob{D_\top} \defeq \ob D + \{ \top \}\]
    \begin{align*}
        D_\top(\top, \top) & \defeq \{ 1_{\top} \} &
        D_\top(d, \top) & \defeq \{ \unit_d \}  \\
        D_\top(d, d') & \defeq D(d, d') &
        D_\top(\top, d) & \defeq \varnothing
    \end{align*}
    where $d, d' \in \ob D$. For each functor $f \colon D \to A$ whose codomain has a chosen terminal object, denote by $f^\top \colon D_\top \to A$ the unique terminal object-preserving functor rendering the following diagram commutative.
    \[\begin{tikzcd}
        {D_\top} \\
        D & A
        \arrow["{f^\top}", dashed, from=1-1, to=2-2]
        \arrow["{j_D}", from=2-1, to=1-1]
        \arrow["f"', from=2-1, to=2-2]
    \end{tikzcd}\qedshift\]
\end{definition}

We will need a simple property of this construction.

\begin{lemma}
    \label{lax-factorisation-property}
    Let $g \colon A_\top \to B$ be a functor preserving the terminal object. Given any functor $f \colon A_\top \to B$ and a natural transformation $\phi \colon f j_A \tto g j_A$ as follows, there is a unique natural transformation $\tau_\phi \colon f \tto g$ such that $\phi = \tau_\phi j_A$.
\end{lemma}

\begin{proof}
	The components of $\tau_\phi$ on objects in $A$ is determined by $\phi$. The component at the terminal object is determined by the unique morphism into the terminal object in $B$. Naturality is trivially satisfied.
\end{proof}

We use the $\ph_\top$ construction to characterise cocones under $D$.

\begin{lemma}
    \label{cocone}
    Let $f \colon D \to A$ be a functor with small domain. Cocones under $f$ are in natural bijection with functors $D_\top \to A$ whose restriction along $j_D \colon D \to D_\top$ is $f$, and their morphisms are in bijection with natural transformations that restrict to $1_f$.
\end{lemma}

\begin{proof}
    The nadir of the cocone is given by evaluating a functor at the terminal object, while each leg is given by evaluating the functor at the unique morphism from the terminal object.
\end{proof}

The key lemma of this section is the following, which gives sufficient conditions for a pseudoalgebra for a relative pseudomonad with codomain $\CAT$ to admit a colimit of a given functor.

\begin{lemma}
    \label{colimits-in-pseudoalgebras}
    Let $\A$ be a full sub-2-category of $\CAT$, let $J$ be the inclusion $\A \hookrightarrow \CAT$, and let $T$ be a $J$-relative pseudomonad. Further suppose that $\algA$ is a $T$-pseudoalgebra and that $D \in \A$ is a category such that
    \begin{enumerate}
        \item $D_\top \in \A$;
        \item $T(D)$ admits a terminal object;
        \item $i_{D_\top} \colon J(D_\top) \to T(D_\top)$ preserves the terminal object.
    \end{enumerate}
    Then any functor $f \colon J(D) \to A$ admits a colimit.
\end{lemma}

\begin{proof}
    First, observe that we have a composite functor $f^a \c (i_D)^\top \colon J(D)_\top \to A$, which by \cref{cocone} exhibits a cocone under $f^a \c i_D$. We will show that this cocone is colimiting: this suffices to establish the result, since $f^a \c i_D \iso f$.
    \[\begin{tikzcd}[sep=large]
        {J(D)_\top} & {T(D)} & A \\
        & {J(D)}
        \arrow["{(i_D)^\top}", from=1-1, to=1-2]
        \arrow["{{f^a}}", from=1-2, to=1-3]
        \arrow[""{name=0, anchor=center, inner sep=0}, "{j_{J(D)}}", from=2-2, to=1-1]
        \arrow["{i_D}"{description}, from=2-2, to=1-2]
        \arrow[""{name=1, anchor=center, inner sep=0}, "f"', from=2-2, to=1-3]
        \arrow["{\tilde a_f}"', shorten <=2pt, Rightarrow, from=1, to=1-2]
        \arrow["{=}"{description}, draw=none, from=0, to=1-2]
    \end{tikzcd}\]
	We begin by establishing weak initiality. Let $k \colon J(D)_\top \to A$ be a cocone under $f^a \c i_D$ (so that $k \c j_{J(D)} = f^a i_D$). We form a morphism of cocones
    $\kappa \colon f^a (i_D)^\top \to k$
    as the following composite functor.
    \[\begin{tikzcd}[column sep=7em]
        {f^a (i_D)^\top} && k \\
        {(f^a i_D)^a(i_D)^\top} && {k^a i_{D_\top}} \\
        {(k j_{J(D)})^a (i_D)^\top} & {(k^a i_{D_\top} j_{J(D)})^a (i_D)^\top} & {k^a (i_{D_\top}j_{J(D)})^* (i_D)^\top}
        \arrow["\kappa", dashed, from=1-1, to=1-3]
        \arrow["{(\tilde{a}_f\inv)^a (i_D)^\top}"', from=1-1, to=2-1]
        \arrow[Rightarrow, no head, from=2-1, to=3-1]
        \arrow["{\tilde{a}_k\inv}"', from=2-3, to=1-3]
        \arrow["{(\tilde{a}_k j_{J(D)})^a (i_D)^\top}"', from=3-1, to=3-2]
        \arrow["{\hat{a}_{k, i_{D_\top}j_{J(D)}} (i_D)^\top}"', from=3-2, to=3-3]
        \arrow[from=3-3, to=2-3]
    \end{tikzcd}\]
    The unlabelled 2-cell on the right-hand side is defined by applying \cref{lax-factorisation-property} to the 2-cell below.
    \[\begin{tikzcd}[column sep=huge]
        {J(D)} & {J(D)_\top} \\
        & {T(D)} \\
        {J(D)_\top} & {T(D_\top)}
        \arrow["{j_{J(D)}}", from=1-1, to=1-2]
        \arrow[""{name=0, anchor=center, inner sep=0}, "{i_D}"{description}, from=1-1, to=2-2]
        \arrow["{j_{J(D)}}"', from=1-1, to=3-1]
        \arrow["{{i_D}^\top}", from=1-2, to=2-2]
        \arrow["{\eta\inv_{i_{D_\top} j_{J(D)}}}"'{pos=0.4}, shorten <=6pt, shorten >=6pt, Rightarrow, from=2-2, to=3-1]
        \arrow["{(i_{D_\top} j_{J(D)})^*}", from=2-2, to=3-2]
        \arrow["{i_{D_\top}}"', from=3-1, to=3-2]
        \arrow["{=}"{description, pos=0.4}, draw=none, from=1-2, to=0]
    \end{tikzcd}\]
    To see verify that $\kappa$ is a morphism of cocones we must show $\kappa \c j_{J(D)}$ is the identity. This follows from commutativity of the following diagram, where the boundary is obtained by precomposing the diagram defining $\kappa$ above by $j_{J(D)}$.
    \[\begin{tikzcd}[column sep=large]
        {f^a i_D} && f && {f^a i_D} \\
        {(f^a i_D)^a i_D} &&&& {k j_{J(D)}} \\
        &&&& {k^a i_{D_\top} j_{J(D)}} \\
        \\
        {(k j_{J(D)})^a i_D} && {(k^a i_{D_\top} j_{J(D)})^a i_D} && {k^a (i_{D_\top} j_{J(D)})^* i_D}
        \arrow[""{name=0, anchor=center, inner sep=0}, "{\tilde{a}_f\inv}", from=1-1, to=1-3]
        \arrow["{(\tilde{a}_f\inv)^a i_D}"', from=1-1, to=2-1]
        \arrow["{\tilde{a}_f}", from=1-3, to=1-5]
        \arrow[""{name=1, anchor=center, inner sep=0}, "{\tilde{a}_{f^a i_D}\inv}"{description}, curve={height=12pt}, from=2-1, to=1-5]
        \arrow[Rightarrow, no head, from=2-1, to=5-1]
        \arrow[Rightarrow, no head, from=2-5, to=1-5]
        \arrow["{\tilde{a}_k\inv j_{J(D)}}"', from=3-5, to=2-5]
        \arrow[""{name=2, anchor=center, inner sep=0}, "{\tilde{a}_{k j_{J(D)}}\inv}"{description}, curve={height=-30pt}, from=5-1, to=2-5]
        \arrow["{(\tilde{a}_k j_D)^a i_D}"', from=5-1, to=5-3]
        \arrow["{\tilde{a}_{k^a i_{D_\top} j_{J(D)}}\inv}"{description}, curve={height=-12pt}, from=5-3, to=3-5]
        \arrow[""{name=3, anchor=center, inner sep=0}, "{\hat{a}_{k, i_{D_\top}j_{J(D)}} i_D}"', from=5-3, to=5-5]
        \arrow[""{name=4, anchor=center, inner sep=0}, "{k^a \eta^{-1}_{i_{D_\top} j_{J(D)}}}"', from=5-5, to=3-5]
        \arrow["{\natof\tilde{a}}"{description}, draw=none, from=0, to=1]
        \arrow["{\natof\tilde{a}}"{description}, draw=none, from=2, to=5-3]
        \arrow["{\algeta{A}}"{description}, curve={height=-6pt}, draw=none, from=3, to=4]
    \end{tikzcd}\]

    Finally, we show any morphism of cocones must be equal to this one. To this end, first observe that the following diagram in $\CAT[T(D), A]$ commutes. Note that the composite $i_{D_\top} \c j_{J(D)}$ is well-typed because $J$ is injective-on-objects, so $J(D_\top) = D_\top = (JD)_\top$. We suppress some of the subscripts below for reasons of space.
   \begin{equation}\label{colimit-in-alg-1}
    \begin{tikzcd}
        {f^a} &&&& {(f^a i_D)^a} \\
        {f^a (i_D)^*} \\
        {f^a((i_D)^\top j_{J(D)})^*} &&&& {(f^a (i_D)^\top j_{J(D)})^a} \\
        {f^a(({i_D}^\top)^* i_{D_\top} j_{J(D)})^*} && {(f^a({i_D}^\top)^* i_{D_\top} j_{J(D)})^a} && {((f^a (i_D)^\top)^a i_{D_\top} j_{J(D)})^a} \\
        {f^a ((i_D)^\top)^* (i_{D_\top} j_{J(D)})^*} &&&& {(f^a (i_D)^\top)^a (i_{D_\top} j_{J(D)})^*}
        \arrow["{f^a \theta_D\inv}"', from=1-1, to=2-1]
        \arrow[""{name=0, anchor=center, inner sep=0}, "{(\tilde{a}_f\inv)^a}"', from=1-5, to=1-1]
        \arrow[Rightarrow, no head, from=1-5, to=3-5]
        \arrow[Rightarrow, no head, from=2-1, to=3-1]
        \arrow["{f^a(\eta_{i^\top} j)^*}"', from=3-1, to=4-1]
        \arrow[""{name=1, anchor=center, inner sep=0}, "{\hat{a}_{f, i^\top j}}"{description}, from=3-5, to=3-1]
        \arrow["{(f^a \eta_{i^\top} j)^a}"{description, pos=0.7}, curve={height=12pt}, from=3-5, to=4-3]
        \arrow[""{name=2, anchor=center, inner sep=0}, "{(\tilde{a}_{f^a i^\top} j)^a}", from=3-5, to=4-5]
        \arrow["{f^a \mu_{i^\top, ij}}"', from=4-1, to=5-1]
        \arrow["{\natof \hat a}"{description}, draw=none, from=4-3, to=3-1]
        \arrow["{\hat{a}_{f, (i^\top)^* i j}}"{description}, from=4-3, to=4-1]
        \arrow["{(\hat{a}_{f, i^\top} ij)^a}"{description}, from=4-5, to=4-3]
        \arrow["{\hat{a}_{f^a i^\top, ij} }", from=4-5, to=5-5]
        \arrow[""{name=3, anchor=center, inner sep=0}, "{\hat{a}_{f, i^\top} (ij)^*}", from=5-5, to=5-1]
        \arrow["{\algtheta{A}}"{description}, draw=none, from=1, to=0]
        \arrow["{\algeta{A}}"{description, pos=0.6}, shift left=2, draw=none, from=4-3, to=2]
        \arrow["{\algmu{A}}"{description}, draw=none, from=4-3, to=3]
    \end{tikzcd}
    \end{equation}
    Denoting by $\term$ the terminal object in $T(D)$, which has been assumed to exist, and using that $i_{D_\top}$ preserves the terminal object also by assumption, the following diagram commutes in $A$.
    \begin{equation}\label{colimit-in-alg-2}
    \begin{tikzcd}
        {f^a ((i_D)^\top)^* (i_{D_\top} j_{J(D)})^*(\term)} &&&& {(f^a (i_D)^\top)^a (i_{D_\top} j_{J(D)})^*(\term)} \\
        \\
        {f^a ((i_D)^\top)^*(\term)} &&&& {(f^a (i_D)^\top)^a(\term)} \\
        {f^a ((i_D)^\top)^* i_{D_\top}(\top)} &&&& {(f^a (i_D)^\top)^a i_{D_\top}(\top)} \\
        && {f^a (i_D)^\top(\top)}
        \arrow[""{name=0, anchor=center, inner sep=0}, "{f^a(i^\top)^*(\unit)}"', from=1-1, to=3-1]
        \arrow["{\hat{a}_{f, i^\top} (ij)^*(\term)}"', from=1-5, to=1-1]
        \arrow[""{name=1, anchor=center, inner sep=0}, "{(f^a i^\top)^a(\unit)}", from=1-5, to=3-5]
        \arrow["\iso"', from=3-1, to=4-1]
        \arrow[""{name=2, anchor=center, inner sep=0}, "{\hat{a}_{f,i^\top}(\term)}"{description}, from=3-5, to=3-1]
        \arrow["\iso", from=3-5, to=4-5]
        \arrow[""{name=3, anchor=center, inner sep=0}, "{f^a\eta_{i^\top}\inv(\top)}"', from=4-1, to=5-3]
        \arrow[""{name=4, anchor=center, inner sep=0}, "{\hat{a}_{f, i^\top} i_{D_\top}(\top)}"{description}, from=4-5, to=4-1]
        \arrow[""{name=5, anchor=center, inner sep=0}, "{\tilde{a}_{f i^\top}\inv(\top)}", from=4-5, to=5-3]
        \arrow["{\natof \hat a}"{description}, draw=none, from=0, to=1]
        \arrow["{i_{D_\top} \text{ preserves terminal obj.} }"{description}, draw=none, from=2, to=4]
        \arrow["{\algeta{A}}"{description}, shift left=2, draw=none, from=3, to=5]
    \end{tikzcd}
    \end{equation}
    Evaluating \eqref{colimit-in-alg-1} at $\term \in T(D)$ and pasting on top of \eqref{colimit-in-alg-2}, we obtain a commutative diagram of morphisms $(f^a i_D)^a(\term) \to (f^a(i_D)^\top)(\top) = f^a(\term)$ in $A$. Tracing the anticlockwise route from top-left ($f^a(\term)$) to bottom ($f^a(i_D)^\top(\top) = f^a(\term)$) in this composite diagram, we obtain a morphism of the form
    $f^a(\cdots) \colon f^a(\term) \to f^a(\term)$, which must, the universal property of the terminal object $\term \in T(D)$, be equal to the identity.
    It follows that $(\tilde{a}\inv_f)^a(\term)$ is equal to the clockwise composite from $(f^a i_D)^a(\term)$ to $(f^a(i_D)^\top)(\top)$.

    Now let $\gamma \colon f^a (i_D)^\top \tto k$ be any morphism of cocones under $f^a i_D$ (so that $\gamma \c j_{J(D)}$ is the identity). We must show that the components of $\kappa$ and $\gamma$ on the terminal object $\top \in J(D)_\top$ are equal. For this, observe that the following diagram commutes in $A$.
    \[\begin{tikzcd}[column sep=huge]
        {(f^a i^\top j)^a(\term)} & {(k j)^a(\term)} \\
        {((f^a i^\top)^a ij)^a(\term)} & {(k^a ij)^a(\term)} \\
        {(f^a i^\top)^a (ij)^*(\term)} & {k^a (ij)^*(\term)} \\
        {(f^a i^\top)^a(\term)} & {k^a(\term)} \\
        {(f^a i^\top)^a i(\top)} & {k^a i(\top)} \\
        {f^a i^\top(\top)} & {k(\top)}
        \arrow[""{name=0, anchor=center, inner sep=0}, "{(\gamma j)^a(\term)}", from=1-1, to=1-2]
        \arrow["{(\tilde a j)^a(\term)}"', from=1-1, to=2-1]
        \arrow["{(\tilde a j)^a(\term)}", from=1-2, to=2-2]
        \arrow[""{name=1, anchor=center, inner sep=0}, "{(\gamma^a ij)^a(\term)}"{description}, from=2-1, to=2-2]
        \arrow["{\hat a(\term)}"', from=2-1, to=3-1]
        \arrow["{\hat a(\term)}", from=2-2, to=3-2]
        \arrow[""{name=2, anchor=center, inner sep=0}, "{\gamma^a (ij)^*(\term)}"{description}, from=3-1, to=3-2]
        \arrow["{(f^a i^\top)^a(\unit)}"', from=3-1, to=4-1]
        \arrow["{k^a(\unit)}", from=3-2, to=4-2]
        \arrow[""{name=3, anchor=center, inner sep=0}, "{\gamma^a(\term)}"{description}, from=4-1, to=4-2]
        \arrow[""{name=4, anchor=center, inner sep=0}, "\iso"', from=4-1, to=5-1]
        \arrow[""{name=5, anchor=center, inner sep=0}, "\iso", from=4-2, to=5-2]
        \arrow[""{name=6, anchor=center, inner sep=0}, "{\gamma^a i(\top)}"{description}, from=5-1, to=5-2]
        \arrow["{\tilde a\inv(\top)}"', from=5-1, to=6-1]
        \arrow["{\tilde a\inv(\top)}", from=5-2, to=6-2]
        \arrow[""{name=7, anchor=center, inner sep=0}, "{\gamma(\top)}"', from=6-1, to=6-2]
        \arrow["{\natof \tilde{a}}"{description}, draw=none, from=0, to=1]
        \arrow["{\natof \hat a}"{description}, draw=none, from=1, to=2]
        \arrow["{\natof \gamma^a}"{description}, draw=none, from=2, to=3]
        \arrow["{\natof \gamma^a}"{description}, draw=none, from=4, to=5]
        \arrow["{\natof \tilde a}"{description}, draw=none, from=6, to=7]
    \end{tikzcd}\]
    The top morphism $(\gamma j_{J(D)})^a(\term)$ is the identity, since $\gamma$ is a cocone. By the calculation above, the composite morphism on the left is equal to $(\tilde{a}\inv_f)^a(\term)$.
    It follows that $\gamma(\top)$ is equal to the composite morphism on the right, precomposed with $(\tilde{a_f}\inv)^a( \term)$. But this is exactly the definition of $\kappa(\top)$, as required.
\end{proof}

The fundamental assumption of \cref{colimits-in-pseudoalgebras} is the existence of terminal objects in free pseudoalgebras. The following gives a useful sufficient condition therefor.

\begin{lemma}
    \label{colimit-of-dense-functor}
    A dense functor admits a colimit if and only if its codomain admits a terminal object, in which case the terminal object exhibits the colimit.
\end{lemma}

\begin{proof}
    Let $j \colon A \to B$ be a functor. The colimit of $j$, if it exists, is given by the left extension $\unit_A \lx j$ of $j$ along the unique functor $\unit_A \colon A \to 1$~\cite[Proposition~3.7.5]{borceux1994handbook1}. The terminal object of $B$, if it exists, is given by the colimit of $1_B$~\cite[Corollary~2.11.6]{borceux1994handbook1}, hence by the left extension $\unit_B \lx 1_B$ of $1_B$ along $\unit_B \colon B \to 1$. We have the following natural isomorphisms, using density of $j$.
    \[\unit_A \lx j \iso (\unit_B j) \lx j \iso \unit_B \lx (j \lx j) \iso \unit_B \lx 1_B \qedhere\]
\end{proof}

\subsection{Free \texorpdfstring{$\Phi$}{Phi}-cocompletions}

We now use the techniques of the previous subsection to prove the main result of this section, which is a characterisation of the bicategories of pseudoalgebras for relative pseudomonads arising from presheaf constructions.

\begin{definition}
    Let $\Phi$ be a class of small categories. Denote by $\Phi\h\COC$ the 2-category whose objects are locally small categories equipped with a choice of $\Phi$-colimits: specifically, an object is a locally small category $A$ along with, for each category $D \in \Phi$ and functor $f \colon D \to A$, a cocone under $f$ that is colimiting. A 1-cell is a functor that sends the colimiting cocones to colimiting cocones (not necessarily the chosen ones). The 2-cells are natural transformations.
\end{definition}

\begin{example}
    \label{Phi-cocompletion}
    Let $\Phi$ be a class of small categories.
    Denote by $J \colon \Cat \ffto \CAT$ the inclusion of the 2-category of small categories into the 2-category of locally small categories. The \emph{$\Phi$-cocompletion} sends a small category $A$ to the smallest full subcategory of $[A\op, \Set]$ closed under representables and $\Phi$-colimits. This exhibits a left $J$-relative pseudoadjoint $\PhiP \colon \Cat \to \Phi\h\COC$ to the forgetful 2-functor $U \colon \Phi\h\COC \to \CAT$.
    \[\begin{tikzcd}
        & {\Phi\h\COC} \\
        \Cat && \CAT
        \arrow[""{name=0, anchor=center, inner sep=0}, "U", from=1-2, to=2-3]
        \arrow[""{name=1, anchor=center, inner sep=0}, "\PhiP", from=2-1, to=1-2]
        \arrow["J"', from=2-1, to=2-3]
        \arrow["\dashv"{anchor=center}, shift right=2, draw=none, from=1, to=0]
    \end{tikzcd}\]
    The unit $\phi_A \colon A \to \PhiP(A)$ is given by the Yoneda embedding. For each small category $X$ and $\Phi$-cocomplete locally small category $A$, the functor
    \[\CAT[X, U(A)] \xfrom{U\ph \c \phi_A} \Phi\h\COC[\PhiP(X), A]\]
    admits a left adjoint equivalence, given by left extension along $\phi_A$~\cite[Theorem~5.35]{kelly1982basic}. We denote the induced $J$-relative pseudomonad by $\PhiP$.
\end{example}

In particular, taking $\Phi$ to be the class of small categories, $\PhiP$ is simply the presheaf construction $P$ of \cref{presheaf-construction-pseudomonad}. We may now instantiate \cref{colimits-in-pseudoalgebras} to the case of relative pseudomonads exhibiting free cocompletions under classes of colimits.

\begin{proposition}
    \label{Phi-pseudoalgebras-are-Phi-cocomplete}
    \label{Phi-pseudomorphisms-are-Phi-cocontinuous}
    Let $\Phi \subseteq \Phi'$ be classes of small categories for which $D \in \Phi$ implies $D_\top \in \Phi'$. Denote by $J \colon \Phi' \ffto \CAT$ the full sub-2-category inclusion of the class $\Phi'$. Every $\PhiP J$-pseudoalgebra is $\Phi$-cocomplete, and every pseudomorphism is $\Phi$-cocontinuous.
\end{proposition}

\begin{proof}
    Let $\algA$ be a $\PhiP J$-pseudoalgebra and let $f \colon J(D) \to A$ be a functor. The functor $\phi_D \colon J(D) \to \PhiP(D)$ is dense and admits a colimit since $D \in \Phi$, so $\PhiP(D)$ admits a terminal object by \cref{colimit-of-dense-functor}. Hence, since the Yoneda embedding preserves limits, \cref{colimits-in-pseudoalgebras} implies that $f$ admits a colimit.

    Now let $(g, \overline g) \colon \algA \to \algB$ be a pseudomorphism of $\PhiP J$-pseudoalgebras. Given a small diagram $f \colon D \to A$, we have a natural transformation
    \[\overline h_f \colon (h f)^b {\phi_B}^\top \tto h f^a {\phi_A}^\top\]
    exhibiting a morphism of cocones. By definition of the colimiting cocones in $A$ and $B$ described by \cref{colimits-in-pseudoalgebras}, this morphism restricts at the nadirs to a morphism $\colim f d \tto f \colim d$, which is precisely the canonical morphism of cocones whose invertibility exhibits cocontinuity.
\end{proof}

Consequently, we observe that every pseudoalgebra for a free cocompletion is lax-idempotent. This significantly strengthens \cite[Example~5.4]{fiore2018relative}, which establishes lax-idempotence for free $P$-pseudoalgebras (recalling \cref{lax-idempotent-iff-free-algebras-are}).

\begin{corollary}
    \label{Phi-pseudoalgebras-are-lax-idempotent}
    Under the same assumptions as \cref{Phi-pseudoalgebras-are-Phi-cocomplete}, every $\PhiP J$-pseudoalgebra is lax-idempotent.
\end{corollary}

\begin{proof}
    Let $\algA$ be a $\PhiP J$-pseudoalgebra. Let $f \colon D \to A$ be a diagram for which $D \in \Phi$. By \cref{Phi-pseudomorphisms-are-Phi-cocontinuous,f-a has ps-morphism structure}, $f^a \colon \PhiP(D) \to A$ is $\Phi$-cocontinuous, and $f^a \phi_A$ is isomorphic to $f$. Since $\PhiP(D)$ is the free $\Phi$-cocompletion of $D$, $f^a$ is therefore the left extension of $f$ along $\phi_D$ by \cite[Theorem~5.35]{kelly1982basic}. Consequently, by \cref{lax-idempotent-iff-left-ext}, $\algA$ is lax-idempotent.
\end{proof}

We now have everything we need in place to prove the main theorem of this section.

\begin{theorem}
    \label{Phi-pseudoalgebras}
    The following relative pseudoadjunction is relatively pseudomonadic, exhibiting $\PsAlg(\PhiP)$ as biequivalent to $\Phi\h\COC$.
    \[\begin{tikzcd}
        & {\Phi\h\COC} \\
        \Cat && \CAT
        \arrow[""{name=0, anchor=center, inner sep=0}, "U", from=1-2, to=2-3]
        \arrow[""{name=1, anchor=center, inner sep=0}, "\PhiP", from=2-1, to=1-2]
        \arrow["J"', from=2-1, to=2-3]
        \arrow["\dashv"{anchor=center}, shift right=2, draw=none, from=1, to=0]
    \end{tikzcd}\]
    More generally, this is true with respect to any $J$ satisfying the assumptions of \cref{Phi-pseudoalgebras-are-Phi-cocomplete}.
\end{theorem}

\begin{proof}
    By definition (\cref{Phi-cocompletion}), the relative pseudoadjunction is a resolution of $\PhiP$ in the sense of \cref{resolution}. Consequently, by \cref{terminal-resolution}, there is a unique comparison pseudofunctor ${!_\Phi} \colon \Phi\h\COC \to \PsAlg(\PhiP)$ commuting with the left and right pseudoadjoints. It suffices to show that the comparison is essentially surjective and essentially \ff.

    First, since every $\PhiP$-pseudoalgebra $\algA$ is $\Phi$-cocomplete by \cref{Phi-pseudoalgebras-are-Phi-cocomplete}, the category $A$ is $\Phi$-cocomplete, and thus induces a $\PhiP$-pseudoalgebra $!_\Phi(A)$, which is necessarily lax-idempotent by \cref{Phi-pseudoalgebras-are-lax-idempotent}. Since any two lax-idempotent pseudoalgebra structures on the same object are unique by \cref{lax-pseudoalgebra-structures-are-essentially-unique}, the comparison pseudofunctor is consequently essentially surjective.

    Second, observe that $\Phi\h\COC \to \CAT$ is locally \ff, with its image being the \mbox{$\Phi$-cocontinuous} functors and all natural transformations. By \cref{2-cells-between-lax-morphisms-of-lax-idempotent-pseudoalgebras-are-transformations}, since every $\PhiP$-pseudoalgebra is lax-idempotent, $\PsAlg(\PhiP) \to \CAT$ is also locally \ff, and its 1-cells are $\Phi$-cocontinuous functors by \cref{Phi-pseudomorphisms-are-Phi-cocontinuous}. Thus the comparison pseudofunctor is locally an equivalence.
\end{proof}

While we are primarily interested in \cref{Phi-pseudoalgebras} when $J \colon \Cat \ffto \CAT$ is the inclusion of small categories, it is worth emphasising that the strongest result is obtained taking $\Phi'$ to be the closure of the class $\Phi$ under $\ph_\top$. In general, the smaller the domain of a relative pseudomonad, the more pseudoalgebras there are, because in that case the pseudoalgebras need to admit extensions with respect to fewer 1-cells. The fact that, in the situation of \cref{Phi-pseudoalgebras}, restricting the domain to $\Phi'$ does not change the pseudoalgebras means, intuitively, that the pseudoalgebra structure is determined by `probing' merely with objects in $\Phi'$.

\begin{example}
    \begin{itemize}
        \item Let $\Phi$ be the class of small categories. Then $\PhiP$ is the presheaf construction relative pseudomonad of \cref{presheaf-construction-pseudomonad}, and its pseudoalgebras are precisely the locally small categories admitting small colimits. The free $\PhiP$-pseudoalgebras are precisely the presheaf categories.
        \item Let $\Phi$ be the class of filtered categories. Then $\PhiP$ is the Ind-completion relative pseudomonad~\cite[Example~3.9]{fiore2018relative}, and its pseudoalgebras are precisely the locally small categories admitting filtered colimits. The free $\PhiP$-pseudoalgebras are precisely the finitely accessible categories. Similar remarks apply taking $\Phi$ to be the class of $\kappa$-filtered categories, for a regular cardinal $\kappa$.
        \item Let $\Phi$ be the class of sifted categories. Then $\PhiP$ is the Sind-completion relative pseudomonad, and its pseudoalgebras are precisely the locally small categories admitting sifted colimits. The free $\PhiP$-pseudoalgebras are precisely the generalised varieties in the sense of \cite{adamek2001sifted}.
        \item Let $\Phi$ be the class of small discrete categories. Then $\PhiP$ is the Fam-completion relative pseudomonad, and its pseudoalgebras are precisely the locally small categories admitting coproducts. This resolves a question in \cite[Chapter 4, Footnote~39]{lewicki2020categories}.
        \item Let $\Phi$ be the singleton class $\{ \varnothing \}$. Then $\PhiP$ freely adjoins an initial object, and we recover \cref{adjoin-initial-object}.
        \qedhere
    \end{itemize}
\end{example}

For completeness, we make explicit the relationship between the relative pseudomonad of \cref{Phi-cocompletion} exhibiting the free $\Phi$-cocompletion of \emph{small} categories and pseudomonad exhibiting the free $\Phi$-cocompletion of \emph{locally small} categories.

\begin{definition}
    Let $\Phi$ be a class of small categories. The assignment of a locally small category $A$ to the smallest full subcategory of $[A\op, \Set]$ closed under representables and $\Phi$-colimits defines a pseudomonad $\PhiP_s$ on the 2-category $\CAT$ of locally small categories~\cite[Theorem~5.35]{kelly1982basic}.
\end{definition}

Let $J \colon \Cat \ffto \CAT$ be the inclusion of small categories. For each class $\Phi$ of small categories, we have $\PhiP = \PhiP_s J$, which induces a canonical comparison pseudofunctor $!_J \colon \PsAlg(\PhiP_s) \to \PsAlg(\PhiP)$ by \cref{realisation-comparison}.

\begin{corollary}
    The canonical pseudofunctor $!_J \colon \PsAlg(\PhiP_s) \to \PsAlg(\PhiP)$ is a biequivalence. More generally, this is true with respect to any $J$ satisfying the assumptions of \cref{Phi-pseudoalgebras-are-Phi-cocomplete}.
\end{corollary}

\begin{proof}
    As described in the introduction to the section, $\PsAlg(\PhiP_s)$ is concretely biequivalent to $\Phi\h\COC$, and this biequivalence trivially factors through $!_J$.
\end{proof}

Consequently, following the terminology of \cite[Definition~5.4.2]{arkor2022monadic}, the $\Phi$-cocompletion $\PhiP_s$ is \emph{$J$-ary}, in that restriction along $J$ preserves the 2-categories of pseudoalgebras and pseudomorphisms.

\begin{remark}
    It is likely that a more abstract proof of \cref{Phi-pseudoalgebras} along the lines of \cite[\S4.4]{altenkirch2015monads} or \cite[\S5.4]{arkor2022monadic} is possible, since the inclusion 2-functor $\Cat \ffto \CAT$ is well-behaved in a two-dimensional sense. We defer such an investigation to future work.
\end{remark}

\begin{remark}
    We may also characterise the 2-categories of $\PhiP$-pseudoalgebras and strict/lax morphisms. The strict morphisms are those that send chosen colimiting cocones to chosen colimiting cocones. By \cref{lax-idempotent-iff-unique-lax-morphism-structure}, lax morphisms are simply arbitrary functors, since $\PhiP$-pseudoalgebras are lax-idempotent by \cref{Phi-pseudoalgebras-are-lax-idempotent}. A characterisation of the colax morphisms would likely require a resolution to \cref{are-colax-morphisms-pseudo}.
\end{remark}

\begin{remark}
    Our proof strategy for \cref{Phi-pseudoalgebras} relies crucially upon \cref{colimits-in-pseudoalgebras}, which uses the fact that cocones under functors $f \colon D \to A$ may be represented by functors $D_\top \to A$. This approach is effective for ordinary categories, and likely also for internal categories. However, it is not clear how to extend this technique to capture cocompleteness for enriched categories. Indeed, it seems likely that, although free cocompletions under classes $\Phi$ of enriched weights do induce lax-idempotent relative pseudomonads, their algebras are not $\Phi$-cocomplete in general. We leave the exploration of such questions to future work.
\end{remark}

\subsection{Applications of coherence}
\label{examples-of-coherence}

We conclude this section by giving some promised applications of the coherence theorem of \cref{coherence-theorem} to $\Phi$-cocompletion relative pseudomonads. We first recall an alternative presentation of free cocompletions in terms of discrete fibrations rather than presheaves.

\begin{example}[Discrete fibration construction]
    \label{discrete-fibration-construction}
    Let $X$ be a small category. For each presheaf $p \colon X\op \to \Set$, the projection from its category of elements $\pi \colon \El(p) \to X$ is a discrete fibration over $X$. This assignment is functorial, and defines an adjoint equivalence between the category $P(X)$ of presheaves on $X$ and natural transformations, and the category $Q(X)$ of discrete fibrations over $X$ and commutative triangles. Furthermore, the following triangle commutes, where $X/\ph$ is the functor sending each object $x \in X$ to the slice category $X/x$.
    \[\begin{tikzcd}
        {P(X)} && {Q(X)} \\
        & X
        \arrow["\El", from=1-1, to=1-3]
        \arrow["\equiv"', draw=none, from=1-1, to=1-3]
        \arrow["{y_X}", from=2-2, to=1-1]
        \arrow["{X/\ph}"', from=2-2, to=1-3]
    \end{tikzcd}\]
    Consequently, by \cref{presheaf-construction-pseudoadjunction}, the discrete fibration construction defines a left relative pseudoadjoint to the forgetful 2-functor $U \colon \COC \to \CAT$.
    \[\begin{tikzcd}
        & \COC \\
        \Cat && \CAT
        \arrow[""{name=0, anchor=center, inner sep=0}, "U", from=1-2, to=2-3]
        \arrow[""{name=1, anchor=center, inner sep=0}, "Q", from=2-1, to=1-2]
        \arrow["J"', from=2-1, to=2-3]
        \arrow["\dashv"{anchor=center}, shift right=2, draw=none, from=1, to=0]
    \end{tikzcd}\]
    This in turn defines a $J$-relative pseudomonad $Q$ that is equivalent, in a suitable sense\footnotemark{}, to the presheaf construction.%
    \footnotetext{We leave it to the reader to formulate the appropriate definitions of morphism and transformation of relative pseudomonads necessary to make this statement precise; it will not be necessary for what follows.}
\end{example}

Recall from \cref{Dist} that the bicategory $\Dist$ of categories, distributors, and natural transformations is defined to be the Kleisli bicategory $\Kl(P)$ associated to the presheaf construction (\cref{presheaf-construction-pseudomonad}). This motivates consideration of the Kleisli bicategory associated to the discrete fibration construction.

\begin{definition}
    The bicategory $\TSDFib$ of categories, two-sided discrete fibrations, and span morphisms is defined to be the Kleisli bicategory $\Kl(Q)$ associated to the discrete fibration construction (\cref{discrete-fibration-construction}).
\end{definition}

We note that in the above definition that a 1-cell in $\Kl(Q)$ is not exactly the usual presentation of a two-sided discrete fibration in terms of a span of categories satisfying a fibration property~\cite{street1980fibrations}, but the data is elementarily equivalent, so it is harmless to identify the notions.

The following theorem appears to be folklore; while this intuition is already present in \cite{street1980fibrations}, we were unable to find a proof in the literature, but it follows easily using the techniques developed in \cref{resolutions-and-coherence}.

\begin{theorem}
    $\Dist$ is biequivalent to $\TSDFib$.
\end{theorem}

\begin{proof}
    $P$ and $D$ both exhibit left pseudoadjoints to $U \colon \COC \to \CAT$. By \cref{Kleisli-is-full-image-resolution}, their Kleisli bicategories are given by taking the full image factorisations of $P$ and $Q$ respectively, and this process evidently preserves biequivalence.
\end{proof}

We now give a number of examples of the coherence theorem for relative pseudomonads. None are new, but have typically been established on a case-by-case basis; our work shows that they are consequences of a single phenomenon.

\begin{example}
    \Cref{coherence-theorem}, in combination with \cref{restriction-of-Kleisli}, gives various well-known coherence theorems as consequences (\cf~\cite[\S1.5]{lack2010companion}). In each of the following, we denote by $J \colon \Set \ffto \Cat \ffto \CAT$ the inclusion of sets viewed as discrete categories.
    \begin{enumerate}
        \item From \cref{Phi-pseudoalgebras}, the 2-category of pseudoalgebras for the presheaf construction is the 2-category of locally small categories with a choice of small colimits. Thus \cref{coherence-theorem} exhibits a biequivalence between the bicategory $\Dist$ of distributors and the full sub-2-category $\Psh$ of $\COC$ spanned by the presheaf categories. This provides a proof of the claim in \cite[Proposition~4.2.4]{cattani1996presheaf}. Similarly, there is a biequivalence between the bicategory $\TSDFib$ of two-sided discrete fibrations and the full sub-2-category $\DFib$ of $\COC$ spanned by categories of discrete fibrations.
        \item Restricting each of $\Dist$, $\TSDFib$, $\Psh$, and $\DFib$ along $J$, we observe that the following bicategories are biequivalent to one another.
        \begin{enumerate}
            \item The bicategory $\Kl(PJ) \defeq \b{Mat}$ of sets and matrices.
            \item The bicategory $\Kl(QJ) \defeq \b{Span}$ of sets and spans.
            \item The 2-category $\Im(F_{PJ})$ of indexed sets.
            \item The 2-category $\Im(F_{QJ})$ of slice categories.
        \end{enumerate}
        \item Let $\Fam$ be the pseudomonad for free coproduct cocompletion. Since the free coproduct completion of a discrete category is equivalent to its free colimit completion, $\Fam \c J$ is biequivalent to $PJ$. Consequently, $\Kl(\Fam \c J)$, which is biequivalent to $\Im(F_{\Fam \c J})$, is biequivalent to the bicategory of sets and spans, recovering \cite[\S A]{lack2002formal}.
        \item Let $T$ be the $J$-relative 2-monad freely adjoining an initial object. Then $\KlT$ has as objects sets, as 1-cells partial functions, and has a 2-cell from $f$ to $g$ if and only if for all $x$, either $f(x)$ is undefined or $f(x) = g(x)$. It is biequivalent to the 2-category of pointed sets, point-preserving functions, in which there is a 2-cell from $f$ to $g$ if and only if for all $x$, either $f(x)$ is the point or $f(x) = g(x)$.
        \qedhere
    \end{enumerate}
\end{example}

\begin{example}
    From \cref{Kleisli-is-full-image-resolution}, it also follows that the bicategory of distributors is biequivalent to the full image of the left pseudoadjoint of any resolution of the presheaf construction. This recovers, for instance, the perspective of \cite[\S1.3]{gambino2017operads}, who define the bicategory of distributors as the full image of the left pseudoadjoint of the forgetful pseudofunctor $U \colon \b{LP}_L \to \CAT$ from the 2-category $\b{LP}_L$ of locally presentable categories, left-adjoint functors, and natural transformations.
\end{example}

\section{Cocompletion-relative monads as cocontinuous monads}
\label{relative-monads-and-monads}

We conclude with an application of the theory of relative pseudomonads to the theory of relative monads, which makes use of two of our main theorems: the \ffness{} of the embedding of $\Kl(T)$ in $\PsAlg(T)$, and the characterisation of the pseudoalgebras for relative pseudomonads arising from free cocompletions.

For any class $\Phi$ of small categories, and denoting by $J \colon \Cat \ffto \CAT$ the full sub-2-category inclusion of small categories, \cref{Phi-pseudoalgebras} establishes that there exists a $J$-relative pseudomonad $\PhiP$ whose pseudoalgebras are precisely the locally small categories admitting $\Phi$-colimits. Now, let $X$ be a small category and consider its free cocompletion $\phi_X \colon X \ffto \PhiP(X)$ under $\Phi$-colimits. There is a monoidal hom-category $\Kl(\PhiP)(X, X)$, whose underlying category is given by $\CAT[X, \PhiP(X)]$. It was proven in \cite[Theorem~4.7]{arkor2024formal} that the category $\CAT[X, \PhiP(X)]$ is also equipped with skew-multicategory structure $\CAT[\phi_X]$, in which the monoids are precisely the $\phi_X$-relative monads~\cite[Theorem~4.16]{arkor2024formal}. Is it natural to wonder whether these two structures are related: the following lemma shows that they coincide; essentially the same observation appears in \cite[215]{szlachanyi2017tensor}, albeit without the two-dimensional perspective.

\begin{lemma}
    \label{skew-multicategory-is-monoidal}
    Let $\Phi$ be a class of small categories, and let $X$ be a small category. The skew-multicategory $\CAT[\phi_X]$ is represented by the monoidal hom-category $\Kl(\PhiP)[X, X]$.
\end{lemma}

\begin{proof}
    First, observe that, since $\phi_X \colon X \to \PhiP(X)$ is a cocompletion, it satisfies conditions (1 -- 3) of \cite[Theorem~4.29]{arkor2024formal}, so that the skew-multicategory $\CAT[\phi_X]$ is represented by \emph{some} monoidal category. We show it is monoidally equivalent to $\Kl(\PhiP)[X, X]$. The underlying categories are isomorphic. The monoidal unit in both cases is given by $\phi_X$. The tensor product of $f \colon X \to \PhiP(X)$ and $g \colon A \to \PhiP(X)$ in the former is given by $f \d (\phi_X \plx g)$, the composition of $f$ with the (pointwise) left extension of $g$ along $\phi_X$. This is, by definition, the composition in the Kleisli bicategory~\cite[Theorem~4.1]{fiore2018relative}.
\end{proof}

For the main theorem of this section, we shall need a preparatory lemma.

\begin{lemma}
    \label{Kl-is-laxly-reflective}
    Let $\J$ be a pseudofunctor and let $T$ be a lax-idempotent $J$-relative pseudomonad. Then, for each $X \in \A$, the functor $(V_T)_{X, X} \colon \KlT[X, X] \to \E[TX, TX]$ admits a right adjoint with invertible counit, and this adjunction lifts to the 2-category of monoidal categories, lax monoidal functors, and monoidal natural transformations.
\end{lemma}

\begin{proof}
    Recall from \cref{Kleisli-resolution} that $(V_T)_{X, X}$ is given by $\ph^*_{X, Y} \colon \E[JX, TY] \to \E[TX, TX]$, which admits a right adjoint with invertible counit when $T$ is lax-idempotent by \cref{lax-idempotence-condition}.
    That $(V_T)_{X, X}$ is pseudo monoidal follows from pseudofunctoriality of $V_T$, whence the adjunction lifts to monoidal categories by doctrinal adjunction (\cref{adjunctions-lift-to-algebras}).
\end{proof}

The first half of the following theorem was proven in \cite[Theorem~4.8]{altenkirch2015monads} via explicit calculations; the second half was proven in \cite[Proposition~3.2.5]{arkor2022monadic}. We give a new, more conceptual proof, of both.

\begin{theorem}
    \label{cocontinuous-monads-are-relative-monads}
    Let $\Phi$ be a class of small categories, and let $A$ be a small category. Left extension along $\phi_X \colon X \ffto \PhiP(X)$, the free cocompletion of $X$ under $\Phi$-colimits, defines a coreflection
    \begin{equation}
        \label{RMnd-Mnd-adjunction}
        \RMnd(\phi_X \colon X \ffto \PhiP(X)) \rightleftarrows \Mnd(\PhiP(X))
    \end{equation}
    which restricts to an equivalence of categories
    \begin{equation}
        \label{RMnd-Mnd-equivalence}
        \RMnd(\phi_X \colon X \ffto \PhiP(X)) \equiv \Mnd_\Phi(\PhiP(X))
    \end{equation}
    between the category of $\phi_X$-relative monads, and the category of $\Phi$-cocontinuous monads on $\PhiP(X)$. Furthermore, this equivalence respects the process of taking algebras, in that the following diagram of categories commutes up to natural isomorphism.
    \[\begin{tikzcd}
        {\RMnd(\phi_X \colon X \ffto \PhiP(X))\op} && {\Mnd_\Phi(\PhiP(X))\op} \\
        & {\CAT/\PhiP(X)}
        \arrow["\equiv", from=1-1, to=1-3]
        \arrow[""{name=0, anchor=center, inner sep=0}, "\Alg"', from=1-1, to=2-2]
        \arrow[""{name=1, anchor=center, inner sep=0}, "\Alg", from=1-3, to=2-2]
        \arrow["\iso"{description}, shift left=2, draw=none, from=0, to=1]
    \end{tikzcd}\]
\end{theorem}

\begin{proof}
    Since $\Phi$ is a lax-idempotent pseudomonad, there there is a lax monoidal coreflection
    \[\Kl(\PhiP)[X, X] \rightleftarrows \CAT[\PhiP(X), \PhiP(X)]\]
    by \cref{Kl-is-laxly-reflective}. The coreflection \eqref{RMnd-Mnd-adjunction} then follows by taking monoids on both sides, which gives relative monads on the left-hand side by \cref{skew-multicategory-is-monoidal}. The equivalence \eqref{RMnd-Mnd-equivalence} follows by restricting from $\CAT[\PhiP(X), \PhiP(X)]$ to $\PsAlg(\PhiP)[\PhiP(X), \PhiP(X)]$ on the right-hand side, using \cref{comparison-pseudofunctor}. That this equivalence commutes with the process of taking algebras follows from the relative monadicity theorem~\cite[Example~4.8]{arkor2024relative}.
\end{proof}

\begin{example}
    Let $\kappa$ be a regular cardinal, and let $X$ be a small category. Take $\Phi$ to be the class of $\kappa$-filtered categories. Then \cref{cocontinuous-monads-are-relative-monads} gives an equivalence between monads relative to the cocompletion of $X$ under $\kappa$-filtered colimits, and $\kappa$-accessible monads (\ie{} monads whose underlying endofunctor preserves $\kappa$-filtered colimits) on $\rm{Ind}_\kappa(X)$. For instance, the category of monads relative to $\b{FinSet} \ffto \Set$, the inclusion of finite sets into sets, is equivalent to the category of finitary monads on $\Set$.
\end{example}

\printbibliography

\end{document}